\documentclass{article}
\usepackage{amsmath}      % Advanced math typesetting
\usepackage{amssymb}      % More symbols
\usepackage{amsthm}       % Theorem environments
\usepackage{mathtools}    % Extensions to amsmath
\usepackage{reptheorem}
\usepackage{aliascnt}
\usepackage{bm}           % Bold math symbols
\usepackage{mathrsfs}     % Special math fonts (like script)
\usepackage{esint}        % Additional integral symbols
\usepackage{cancel}       % Cancel terms in equations
\usepackage{enumitem}     % Customizable lists
\usepackage{todonotes}    % To-do notes for collaborators
\usepackage{graphicx} % Required for inserting images
\usepackage{hyperref}
\usepackage{xcolor}
\usepackage[normalem]{ulem}
\usepackage{verbatim}
\usepackage[T1]{fontenc}
\usepackage[style=alphabetic,maxbibnames=99,maxcitenames=2]{biblatex}
\AtEveryBibitem{\clearfield{urldate}}  % suppress "(Visited on...)" for submission
\usepackage[margin=1in]{geometry}
\DeclareMathOperator{\Dim}{Dim}
\DeclareMathOperator{\Var}{Var}
\newtheorem{theorem}{Theorem}[section]
\newaliascnt{lemma}{theorem}
\newtheorem{lemma}[lemma]{Lemma}
\newaliascnt{corollary}{theorem}
\newtheorem{corollary}[corollary]{Corollary}
\aliascntresetthe{lemma}
\aliascntresetthe{corollary}

\newtheorem{conjecture}[theorem]{Conjecture}
\newtheorem{thmalpha}{Theorem}

\theoremstyle{definition}
\newtheorem{definition}[theorem]{Definition}

\newtheorem{example}[theorem]{Example}

\theoremstyle{remark}

\newtheorem{remark}[theorem]{Remark}

\title{Point-to-set principles in dynamical systems}
\author{Emma Dinowitz}
\date{July 2026}

\begin{document}

\maketitle
\begin{abstract}
   In analogy to the point-to-set principle for Hausdorff and packing
   dimension of Lutz, Lutz, and Mayordomo \cite{lutz-2023}, we develop a
   Kolmogorov complexity point-to-set principle framework for upper and
   lower topological pressure and BS dimension in dynamical systems (and
   by extension for upper and lower entropy). As a demonstration of this
   framework we expand upon the Bowen pressure equation results of
      Climenhaga in \cite{climenhaga-2010} and \cite{CLIMENHAGA2014The}, as well as prove a pointwise version of the Bowen pressure equation.
\end{abstract}
\tableofcontents

\section{Introduction}

A recurring goal of researchers in the dimension theory of dynamical systems is to
characterize set-level dimension quantities pointwise, that is, to
attach to each point a quantity whose supremum over a set recovers the dimension of the
set. Classically, this is partially achieved with local dimensions relative
to a measure, but the measures that effectively compute dimension are
Gibbs measures, highly structured objects which in many systems either do
not exist or are very hard to construct. This paper offers an alternative
framework for pointwise dimension quantities via the point-to-set
principle. Where local dimension depends on a choice of measure, the
pointwise quantities here (the $A$-effective dimensions) depend on a
choice of oracle $A \subset \mathbb{N}$. Each oracle induces a
Kolmogorov complexity $K^A$ on finite binary strings, and the oracles are
naturally preordered by Turing reducibility $\leq_T$ (see
Section~\ref{sec:cones}). The point-to-set principle states that the
supremum of the $A$-effective dimensions over a set recovers its
dimension for every sufficiently strong oracle. Since the pointwise
dimensions are nonincreasing under $\leq_T$ and an achieving oracle
always exists, this is phrased as an attained minimum over oracles. The
point-to-set principles in the literature are for Hausdorff and packing
dimension \cite{lutz-2023}:
\[
\dim_H(Z)=\min\limits_{A\subset \mathbb{N}}\sup_{x\in Z}\dim^A(x),
\qquad
\dim_P(Z)=\min\limits_{A\subset \mathbb{N}}\sup_{x\in Z}\Dim^A(x),
\]
where $\dim^A$ and $\Dim^A$ denote the effective and strong effective
dimensions (Section~\ref{sec:intro-effdim}). The minimum is attained on a
set of oracles upward-closed under $\leq_T$. The main thesis of this
paper is that classical relationships between dimension quantities in
dynamical systems can be studied by passing to pointwise algorithmic
counterparts and proving the relationships at every point, for all
oracles in an upper set. Since the minimum in the point-to-set principle
may be computed within any upper cone (Lemma~\ref{lem:cone-min}), the
set-level statements then follow by oracle minimization alone.

Towards these goals we prove point-to-set principles for the dynamical
dimension quantities of lower and upper pressure (known in the
literature as Pesin--Pitskel and packing pressure respectively) and
lower and upper BS dimension. These are Theorems~\ref{thm:A}
and~\ref{thm:B}. As special cases, we obtain point-to-set principles for
lower and upper entropy (Bowen and packing entropy respectively). Our
main demonstration of the pointwise conversion framework follows. In
Theorem~\ref{thm:C}, we prove pointwise versions of the Bowen pressure
equation and upgrade them to their set-level versions, Theorem~\ref{thm:D}. In
Theorem~\ref{thm:E}, we use a pointwise conversion argument to prove
generalizations of the classical relationship between Hausdorff and
packing dimension, BS dimension, and the Bowen pressure equation,
recovering and extending the results of \cite{climenhaga-2010} to the
packing and noncompact settings. Beyond these immediate demonstrations,
this is intended as a framework paper enabling future work in the
dimension theory of nonuniformly hyperbolic systems.

Point-to-set principles admit natural strategies for proving dimension bounds:
\begin{itemize}
    \item To prove a lower bound $s$: for every oracle $A$ and every
    $\varepsilon > 0$, exhibit a point of the set whose $A$-pointwise
    dimension is at least $s - \varepsilon$.
    \item To prove an upper bound $s$: exhibit a single oracle $B$ such
    that for every $A \geq_T B$, the $A$-pointwise dimension of every
    point of the set is at most $s$.
\end{itemize}
Neither strategy requires the subset to be compact, invariant, or to
support any invariant measure, so both remain available where
the measure-theoretic machinery breaks down. The impetus for this
research program is a problem of exactly this kind, namely computing the
Hausdorff and packing dimensions of the level sets $L(\alpha)$,
$\alpha \geq 0$, of Lyapunov exponents for geodesic flows on surfaces
of nonpositive curvature, resolving the conjectures of
\cite{burns-2013}. For the level set of Lyapunov exponents $L(\alpha)$, letting $h(Z,T)=\underline{h}(Z,T)=\overline{h}(Z,T)$ when the two are equal, the conjectures are
\begin{conjecture}
For $(T^1M,f^t)$ geodesic flow on a compact rank 1 surface of nonpositive curvature
\begin{enumerate}
    \item  when $\alpha>0$
    \[\dim_H(L(\alpha))=\dim_P(L(\alpha))= 1+2\frac{h(L(\alpha),T)}{\alpha}\]
    \item when $\alpha=0$
    \[\dim_H(L(0))=\dim_P(L(0))=3\]
\end{enumerate}
\end{conjecture}

This setting is nonuniformly hyperbolic and resists
the standard methods. For the first conjecture, the measures constructed in \cite{burns-2018}
concentrating on the level sets for $\alpha > 0$ are not known to
have the Gibbs-type properties that would yield upper bounds using local dimension. For the second the set
$L(0)$ has entropy 0 and so supports no invariant measure of positive entropy at all. In addition,
the local product structure in this setting has only leafwise
regularity rather than uniform Lipschitz regularity, making a splitting into positive and negative time complicated. While the present
paper does not carry out the strategies above, the forthcoming companion paper
\cite{dinowitz2026dimension}
applies them to this problem. Our other planned future work intends to apply these methods to study the Ledrappier-Young dimension formulas \cite{Ledrappier-young} from the algorithmic information theory perspective. We also intend to study the non conformal uniformly hyperbolic setting, as well as the problem of variational principles for dimension quantities on noncompact settings under weak specification type properties (see \cite{pfister_2007_on} for an example of this and \cite{kwietniak2015} for an overview of specification type properties).

\subsection{Effective dimension and oracles}
\label{sec:intro-effdim}

An oracle $A \subset \mathbb{N}$ encodes a countable list of auxiliary information made available to an algorithm. Oracles serve two purposes in this paper: they record the data needed to do computable analysis in general settings (a dense sequence in a complete separable metric space, the distances between its points, and a description of the transformation $T$), and they are chosen to realize point-to-set principles. Any
countable collection of such data can be encoded into a single oracle,
so one may always imagine working in the computable setting relative to
a list of everything one needs. Oracles also interact well with measures: every measure admits an oracle relative to which it is computable, a fact used extensively in the companion paper \cite{dinowitz2026dimension}. Two operations on oracles are used
throughout: the joins $A \oplus B$ and $\bigoplus\limits_{i\in \mathbb{N}}A_i$, which are encodings of the concatenations of
a list of lists of information, and the upper cone above $B$, the collection
of oracles that compute $B$. Formal
definitions appear in Section~\ref{sec:cones}.

Each oracle induces a prefix-free Kolmogorov complexity $K^A(\sigma)$
on finite binary strings: the minimum length of a program which, with
access to $A$, outputs $\sigma$. The choice of universal machine
changes $K^A$ only by an additive constant. The effective dimensions
measure the asymptotic cost of locating a point to a given precision.

\begin{definition}[Effective dimension]
\label{def:intro-effective-dim}
Let $(X,d)$ be an $A$-computable metric space with dense sequence
$(x_i)_{i\in\mathbb{N}}$, and for $x \in X$, $\delta > 0$ set
\[
K^A_\delta(x) = \min\{\, K^A(i) \mid d(x_i, x) < \delta \,\}.
\]
The \emph{$A$-effective dimension} and \emph{strong $A$-effective
dimension} of $x$ are
\[
\dim^A(x) = \liminf\limits_{\delta \to 0} \frac{K^A_\delta(x)}{-\log \delta},
\qquad
\Dim^A(x) = \limsup\limits_{\delta \to 0} \frac{K^A_\delta(x)}{-\log \delta}.
\]
\end{definition}

\begin{example}[The middle-thirds Cantor set]
\label{ex:intro-cantor}
Identify $C$ with the ternary strings on $\{0,2\}$. At scale
$\delta = 3^{-k}$, specifying $x \in C$ to precision $\delta$ amounts
to specifying its first $k$ digits, viewed as a binary string, so
$K^A_{3^{-k}}(x) = K^A(x{\restriction}k) + O(\log k)$. Since every
binary string satisfies $K^A(x{\restriction}k) \leq k + O(\log k)$, every
$x \in C$ has $\Dim^A(x) \leq \log 2/\log 3$ for every $A$; and for $x$
with $A$-Martin-L\"of random digit sequence,
$K^A(x{\restriction}k) \geq k - O(1)$ gives
$\dim^A(x) \geq \log 2/\log 3$. The point-to-set principle then yields
$\dim_H(C) = \dim_P(C) = \log 2/\log 3$. The upper bound holds for
every oracle, and the lower bound is witnessed, for each oracle $A$, by
any point with $A$-random digits. These are precisely the two strategies described in the introduction.
\end{example}

We will use one structural fact about oracles throughout: any two oracles
have a common upper bound (their join), so ``sufficiently strong'' is a
coherent notion, and the minimum in a point-to-set principle may be
sought within the cone above any fixed base oracle - for instance, one
encoding the system $(X,d,T)$ and the potentials under study. The achieving oracle itself is produced by a countable join of almost-achieving oracles.

\subsection{Main results}
\label{sec:intro-main-results}

Throughout the paper $(X,d)$ is a complete separable metric space,
$T \colon X \to X$ is a continuous map, and the potentials
$\varphi, u \colon X \to \mathbb{R}$ are uniformly continuous. We fix a
dense sequence $(x_i)_{i \in \mathbb{N}}$ in $X$ and write
$B_n(x,T,\varepsilon)$ for the Bowen ball of order $n$ and radius
$\varepsilon$. The pointwise quantities at the center
of the paper are the lower and upper $A$-pointwise pressures
\[
\underline{P}^A(x,T,\varphi,\varepsilon)
= \liminf\limits_{n\to\infty}
\min\limits_{x_i \in B_n(x,T,\varepsilon)}
\frac{K^A(i) + S_n\varphi(x_i)}{n},\ 
\underline{P}^A(x,T,\varphi)=\lim\limits_{\varepsilon\to 0} \underline{P}^A(x,T,\varphi,\varepsilon)
\]
\[\overline{P}^A(x,T,\varphi,\varepsilon)
= \limsup\limits_{n\to\infty}
\min\limits_{x_i \in B_n(x,T,\varepsilon)}
\frac{K^A(i) + S_n\varphi(x_i)}{n},\ \overline{P}^A(x,T,\varphi)=\lim\limits_{\varepsilon\to 0}\overline{P}^A(x,T,\varphi,\varepsilon)\]
 and the lower/upper pointwise BS dimensions
\[
\dim^A_{\mathrm{BS}}(x,T,u,\varepsilon)
= \liminf\limits_{n\to\infty}\,
\min\limits_{x_i \in B_n(x,T,\varepsilon)}
\frac{K^A(i)}{S_n u(x_i)},\ \dim^A_{\mathrm{BS}}(x,T,u)=\lim\limits_{\varepsilon\to 0}\dim^A_{\mathrm{BS}}(x,T,u,\varepsilon)
\]
\[
\Dim^A_{\mathrm{BS}}(x,T,u,\varepsilon)
= \limsup\limits_{n\to\infty}\,
\min\limits_{x_i \in B_n(x,T,\varepsilon)}
\frac{K^A(i)}{S_n u(x_i)},\ \Dim^A_{\mathrm{BS}}(x,T,u)=\lim\limits_{\varepsilon\to 0}\Dim^A_{\mathrm{BS}}(x,T,u,\varepsilon)
\]
Each pointwise quantity measures the exponential cost of
describing the orbit of $x$ to precision $\varepsilon$ for $n$ units of
time. The pressure quantities weight this cost additively by the
potential, while the BS quantities reparametrize time through $u$.
Precise definitions appear in Section~\ref{sec_point_to_set}.

Our first pair of results are the point-to-set principles for
topological pressure, in direct analogy with the principles for
Hausdorff and packing dimension. Here $\underline{P}(Z,T,\varphi)$ denotes the
Pesin--Pitskel pressure of the (arbitrary, not necessarily compact or
invariant) set $Z$, and $\overline{P}(Z,T,\varphi)$ the packing
pressure.

\begin{thmalpha}[Point-to-set principle for topological pressure]
\label{thm:A}
Let $Z \subset X$. Then
\[
\underline{P}(Z,T,\varphi)
= \min\limits_{A \subset \mathbb{N}} \sup_{x \in Z}
\underline{P}^A(x,T,\varphi),
\qquad
\overline{P}(Z,T,\varphi)
= \min\limits_{A \subset \mathbb{N}} \sup_{x \in Z}
\overline{P}^A(x,T,\varphi).
\]
\end{thmalpha}

Our second pair of results treats the BS (Barreira--Schmeling) dimension, the
quantity that solves Bowen's equation. We work with a generalization of
the classical notion: rather than requiring $u$ to be positive on a
compact space, we allow $u$ to be an arbitrary uniformly continuous
real-valued function and require positivity only asymptotically on
average, pointwise on the set in question. Writing
$\underline{u}(x) = \liminf\limits_{n\to\infty} \frac{S_n u(x)}{n}$ we assume
$\underline{u}(x) > 0$ for all $x \in Z$.

\begin{thmalpha}[Point-to-set principle for BS dimension]
\label{thm:B}
Let $Z \subset X$ with $\underline{u}(x) > 0$ for all $x \in Z$. Then
\[
\dim_{\mathrm{BS}}(Z,T,u)
= \min\limits_{A \subset \mathbb{N}} \sup_{x \in Z}
\dim^A_{\mathrm{BS}}(x,T,u),
\qquad
\Dim_{\mathrm{BS}}(Z,T,u)
= \min\limits_{A \subset \mathbb{N}} \sup_{x \in Z}
\Dim^A_{\mathrm{BS}}(x,T,u).
\]
\end{thmalpha}
In both of theorems \ref{thm:A} and \ref{thm:B} the minimum is attained, and attained on an upper
set of oracles, meaning if an oracle A achieves the minimum then any oracle B which computes A also achieves the minimum. The lower and upper statements require different
discretizations. The lower quantities use the discretization of null $s$-covers, which was introduced in \cite{galatolo-2009} for the lower entropy setting. The upper quantities use a new discretization called uniform
$s$-nets, which discretizes the equivalent definition of modified upper capacity pressure/BS dimension instead of the packing pressure/BS dimension definition. This discretization applies in the packing dimension setting as well and is a new addition to the effective dimension literature. The proofs of the point-to-set principles occupy
Section~\ref{sec_point_to_set}.
Since topological entropy is the special case $\varphi \equiv  0$ of pressure
(equivalently $u \equiv 1$ of BS dimension), Theorems~\ref{thm:A}
and~\ref{thm:B} contain point-to-set principles for Bowen and packing
entropy as corollaries. The Bowen entropy case is essentially due to
Galatolo, Hoyrup, and Rojas \cite{galatolo-2009}, whose framework
we recall and relativize in Section~\ref{subsec_orbit_complexity}. The packing entropy result and all results about pressure and BS dimension are new.

As an immediate illustration of the pointwise method we compute the
topological entropy of isometries (Example~\ref{ex:isometry}) as well as recover the classical relation
$\dim_H(Z) = h_{\mathrm{top}}(Z)/(-\log\beta)$ between Hausdorff
dimension and Bowen entropy for arbitrary subsets of symbolic spaces
with the standard metric $d_\beta$ (Example~\ref{ex:symbolic}).

Our third result examines a classical tool of the dimension theory of dynamical systems from the point-to-set perspective: Bowen's equation
as a pointwise phenomenon. 

\begin{thmalpha}[Pointwise Bowen equation]
\label{thm:C}
Let $x \in X$ with $0 < \underline{u}(x) \leq \overline{u}(x) < \infty$
and let $A \subset \mathbb{N}$ be any oracle. If $\underline{K}^A(x,T)<\infty$ then
$\dim^A_{\mathrm{BS}}(x,T,u)$ is the unique solution $s$ of
\[
\underline{P}^A(x,T,-su) = 0,
\]
and if $\overline{K}^A(x,T)<\infty$ then $\Dim^A_{\mathrm{BS}}(x,T,u)$ is the unique solution of
$\overline{P}^A(x,T,-su) = 0$.
\end{thmalpha}

The proof directly analyzes the defining $\liminf$/$\limsup$
expressions. The map $s \mapsto \underline{P}^A(x,T,-su)$ is strictly
decreasing with slope bounded away from zero by $\underline{u}(x)$, and
its sign at $s$ is controlled by the comparison of $s$ with the
pointwise BS dimension. The proof uses no covering argument, variational principle, or compactness. Applying Theorems~\ref{thm:A} and~\ref{thm:B} to the pointwise
statement then yields the set-level equation.

\begin{thmalpha}[Set-level Bowen pressure equation]
\label{thm:D}
Let $Z \subset X$ with $0 < \underline{u}(x) \leq \overline{u}(x) < \infty$
for all $x \in Z$. Then
\[
\dim_{\mathrm{BS}}(Z,T,u)
= \sup\{\, s \mid \underline{P}(Z,T,-su) > 0 \,\}
= \inf\{\, s \mid \underline{P}(Z,T,-su) \leq 0 \,\},
\]
\[
\Dim_{\mathrm{BS}}(Z,T,u)
= \sup\{\, s \mid \overline{P}(Z,T,-su) > 0 \,\}
= \inf\{\, s \mid \overline{P}(Z,T,-su) \leq 0 \,\}.
\]
Assume moreover that $\underline{h}(Z,T) < \infty$ (respectively
$\overline{h}(Z,T) < \infty$).
\begin{enumerate}
    \item[(i)] \emph{(Existence)} If $\overline{u}(x) \leq \beta < \infty$
    on $Z$, then the lower (respectively upper) Bowen pressure equation
    \[
    \underline{P}(Z,T,-su) = 0,
    \qquad
    \overline{P}(Z,T,-su) = 0,
    \]
    has a solution, given by $s_0 = \dim_{\mathrm{BS}}(Z,T,u)$
    (respectively $s_0 = \Dim_{\mathrm{BS}}(Z,T,u)$).
    \item[(ii)] \emph{(Uniqueness)} If $\underline{u}(x) \geq \alpha > 0$
    on $Z$, then $s \mapsto \underline{P}(Z,T,-su)$ and
    $s \mapsto \overline{P}(Z,T,-su)$ are strictly decreasing; in
    particular each equation has at most one solution.
\end{enumerate}
In particular, under
$0 < \alpha \leq \underline{u}(x) \leq \overline{u}(x) \leq \beta < \infty$
on $Z$, the lower and upper Bowen pressure equations have unique
solutions, given by $\dim_{\mathrm{BS}}(Z,T,u)$ and
$\Dim_{\mathrm{BS}}(Z,T,u)$ respectively.
\end{thmalpha}

Theorem~\ref{thm:D} is proved in Section~\ref{subsec:bowen-pts} as
Theorems~\ref{thm:BPE} and~\ref{thm:BPE_unique}. This derivation (pointwise equation first, set equation as a
corollary via oracle minimization) is our template for converting
between dimension quantities.
Our final result applies this machinery to conformal nonuniformly
expanding systems, recovering and extending the Bowen equation results
of Climenhaga \cite{climenhaga-2010}. Let $T$ be conformal with
derivative $a(x)$ and set $u = \log a$. Let $\mathcal{A}(0,\infty)$ denote the
set of points with $0 < \underline{u}(x) \leq \overline{u}(x) <
\infty$, let $\mathcal{B}$ denote the set of points with tempered contraction,
and let $\mathcal{C}$ denote the set of points with $u(T^n x) = o(n)$; precise
definitions appear in Section~\ref{subsec:conformal}.

\begin{thmalpha}[Dimension via Bowen's equation, nonuniform noncompact
setting]
\label{thm:E}
Let $T$ be conformal as above and let
$Z \subset \mathcal{A}(0,\infty) \cap \mathcal{B} \cap \mathcal{C}$. Then
\[
\dim_H(Z) = \dim_{\mathrm{BS}}(Z,T,u),
\qquad
\dim_P(Z) = \Dim_{\mathrm{BS}}(Z,T,u).
\]
In particular, if additionally $\underline{u}(x) \geq \alpha > 0$ on
$Z$, then $\dim_H(Z)$ and $\dim_P(Z)$ are the unique solutions of the
lower and upper Bowen equations $\underline{P}(Z,T,-su) = 0$ and
$\overline{P}(Z,T,-su) = 0$.
\end{thmalpha}

Comparing Theorem~\ref{thm:E} with the main theorem of
\cite{climenhaga-2010} and the extension of \cite{CLIMENHAGA2014The}, we find four differences. First,
Hausdorff and packing dimension are treated simultaneously by the same
argument. The packing statement appears to be new even in the compact
setting. Second, the ambient space is complete separable rather than
compact, with the role of compactness replaced by the sublinear growth condition defining $\mathcal{C}$.
Third, our definition of the BS dimension does not assume property (P) as in \cite{CLIMENHAGA2014The} (see section \ref{sec:intro-history}), allowing us to show equality of the Hausdorff/packing and BS dimensions in all cases under consideration. Fourth, the
proof contains no covering argument at the level of sets:
Theorem~\ref{thm:E} is deduced from a pointwise conversion
(Theorem~\ref{thm:pointwise-conformal}) showing that at every point of
$\mathcal{A}(0,\infty) \cap \mathcal{B} \cap \mathcal{C}$ the effective dimension and the pointwise
BS dimension coincide,
$\dim^A(x) = \dim^A_{\mathrm{BS}}(x,T,u)$, for every oracle $A$. The
set-level theorem is then a chain of oracle minimizations.

\subsection{Context and related work}
\label{sec:intro-history}

The connection between algorithmic information and dynamics dates to
Kolmogorov's conception of Kolmogorov complexity as an alternate
formulation of entropy. Following the development of Kolmogorov--Sinai
entropy, Brudno \cite{Brudno1978The}
defined a pointwise entropy via Kolmogorov complexity and proved that
for an ergodic measure $\mu$ it agrees $\mu$-almost everywhere with the
entropy of $\mu$. White generalized this in his thesis
\cite{white1991algorithmicthesis}, 
parts of which appeared in \cite{white_1993_algorithmic}. 
These frameworks encode points symbolically through finite absolutely
continuous open covers, not necessarily computable, and measure the
information content of orbits through the resulting symbolic sequences.
Galatolo \cite{galatolo-earlier}
and then Galatolo, Hoyrup, and Rojas
\cite{galatolo-2009, galatolo_2011_dynamics} merged this line of
work with computable analysis to produce a streamlined notion of orbit
complexity flexible enough to support a point-to-set principle.
Going from \cite{galatolo-2009} to a point-to-set
principle for Bowen entropy consists only of relativization to
arbitrary oracles and the achieving-oracle construction via countable
joins, and we present the entropy principle in
Section~\ref{subsec_orbit_complexity} as an exposition of their results. The
same cannot be said of the rest of the paper. The upper (packing-type)
quantities do not discretize through covers at all. They require the
uniform $s$-net machinery of
Sections~\ref{subsubsec:upper-pressure-disc}--\ref{subsec_pts_BS_upper}, whose
effectivization is where most of the new technical work of the paper
lies. The BS dimension we treat is more general than the classical one,
replacing global positivity of $u$ by asymptotic-average positivity on
the set in question.

Dimension theory of dynamical systems has developed along two broad
directions. The first is measure-theoretic: in the classical work of
Barreira and others (eg \cite{Ruelle1982Repellers},\cite{young_1982_dimension},\cite{Ledrappier-young},\cite{barreira_2002_higherdimensional}),
thermodynamic formalism produces measures that characterize the
dimension of distinguished sets. This is the standard engine of
multifractal analysis. A dynamically meaningful quantity defines a
parametrized family of level sets. The invariant measures supported on
each level set are studied, the dimension of the set is related to the
dimensions of these measures, and those in turn to entropy and Lyapunov
exponents, yielding a dimension formula. The second direction is purely
dimension-theoretic, and originates with Bowen's definition of
topological entropy as a dimension-type quantity defined for arbitrary
subsets. Pesin generalized this idea into the Carath\'eodory dimension
characteristic framework, within which Pesin and Pitskel defined the
dimension quantity associated to topological pressure. In the same
framework Barreira and Schmeling \cite{Barreira2000Sets}
introduced the BS dimension (originally the $u$-dimension), the
quantity characterizing the solution of the Bowen pressure equation.
The notions studied throughout this literature are those analogous to
Hausdorff dimension: in the naming convention of this paper, the
\emph{lower} pressure and \emph{lower} BS dimension.

Climenhaga \cite{climenhaga-2010}, inspired by Rugh \cite{Rugh2008On}, 
departed from the measure-theoretic pattern by working directly with
the dimension quantities, proving that Hausdorff dimension solves the
Bowen pressure equation for nonuniformly expanding subsets of conformal
systems on compact metric spaces. That paper avoided using the BS dimension
because at that point it was defined only for positive potentials
$u \colon X \to (0,\infty)$, far too restrictive for the nonuniform
setting. Later in \cite{CLIMENHAGA2014The}
Climenhaga generalized the BS dimension to continuous
$u \colon X \to \mathbb{R}$ on compact systems satisfying a global
covering condition he called property~(P): for every
$\delta > 0$ there exist covers $E_N \in \mathcal{P}(X, N, \delta)$
such that $\lim_{N \to \infty} \inf_{(x,n) \in E_N} S_n u(x)
= +\infty$, in the notation of Definition~\ref{def:lower-pressure}.
He extended the Bowen-equation characterization to this generalized
definition, and stated the results of \cite{climenhaga-2010} for systems satisfying property~(P). The corresponding packing-type notions were not introduced until after \cite{climenhaga-2010}. Packing entropy was introduced in \cite{Feng2012Variational_entropy}, 
and packing pressure and packing BS dimension in \cite{Wang2012Variational_BS}.
To our knowledge there has been no further work generalizing BS dimension
beyond strictly positive potentials, and the packing analogues of the
Bowen pressure equation do not appear to have been explored. From this
perspective, the present paper fills several gaps. It provides a unified framework encompassing both the lower and upper dimension quantities for pressure and BS dimension, and hence for entropy; it works in the complete separable setting with uniformly continuous potentials in place of the compact setting; and it establishes Bowen pressure equations for the generalized BS dimension (in particular one usable in an extension of Climenhaga's argument which applies fully to the original setting of his paper), together with the extension of \cite{climenhaga-2010} to the packing setting.

As a final motivation for our approach, we note that much of the technical difficulty in dimension theory lies in the combinatorics of covers. By reframing every quantity through algorithmic information theory we trade those difficulties for the technicalities of computability, but from it we gain the ability to talk about points individually and more directly engage with the geometry and algorithmic information theory of the setting. In our experience, the pointwise arguments are the simpler ones, and lead to the ability to reason in a more natural way.

\subsection{Structure of the paper}
\label{sec:intro-structure}

Section~\ref{sec:prelim} collects preliminaries: computability and
computable analysis, prefix-free Kolmogorov complexity and the coding
theorem, orbit complexity, and the join and upper-cone mechanics
underlying oracle minimization (including
Lemma~\ref{lem:cone-min}). Section~\ref{sec_point_to_set} proves the
point-to-set principles: we first present the Bowen entropy case as a
relativization of Galatolo--Hoyrup--Rojas
(Section~\ref{subsec_orbit_complexity}), then develop the lower and upper
pressure principles via null $s$-covers and uniform $s$-nets
respectively, and then the lower and upper BS dimension principles
under the asymptotic positivity hypothesis. Section~\ref{sec:bowen}
develops the Bowen pressure equation: the pointwise equation
(Theorem~\ref{thm:C}), its set-level consequences, existence and
uniqueness of solutions, and the application to conformal nonuniformly
expanding systems (Theorem~\ref{thm:E}). The appendix establishes the
equivalence of the packing-type definitions used in the body with the
classical packing pressure and packing BS dimension.

\subsection{Notation}
\label{sec:intro-notation}

The following tables collect the recurring notation of the paper, with
pointers to the definitions.

\medskip
\noindent\textbf{Computability and algorithmic information.}\par\medskip
\noindent\begin{tabular}{@{}p{0.30\textwidth}p{0.64\textwidth}@{}}
$A, B \subset \mathbb{N}$ & oracles (Section~\ref{sec:prelim-computability}) \\
$\leq_T$, $\equiv_T$ & Turing reducibility and equivalence \\
$\uparrow\! A$ & upper cone $\{B \mid A \leq_T B\}$ (Definition~\ref{def:cones-joins}) \\
$A \oplus B$, $\bigoplus_j A_j$ & joins of oracles (Definition~\ref{def:cones-joins}) \\
$K^A(\sigma)$ & prefix-free Kolmogorov complexity relative to $A$ (Definition~\ref{def:complexity}) \\
$K^A(S)$ & $\min_{\sigma \in S} K^A(\sigma)$ for a set of strings $S$ \\
$x{\restriction}n$ & length-$n$ prefix of $x \in \mathcal{A}^{\mathbb{N}}$ for $\mathcal{A}$ a finite or countable alphabet. \\
$(x_i)_{i\in\mathbb{N}}$ & ideal points of an A-computable metric space structure (Definition~\ref{def:computable-metric}) \\
admissible oracle & Definition~\ref{def:admissible} \\
$K^A_\delta(x)$ & $\min\{K^A(i) \mid d(x_i,x)<\delta\}$ (Definition~\ref{def:intro-effective-dim}) \\
$\dim^A(x)$, $\Dim^A(x)$ & effective and strong effective dimension (Definition~\ref{def:intro-effective-dim}) \\
\end{tabular}

\medskip
\noindent\textbf{Dynamical data.}\par\medskip
\noindent\begin{tabular}{@{}p{0.30\textwidth}p{0.64\textwidth}@{}}
$(X,d)$, $T$, $\varphi$, $u$ & complete separable metric space, continuous map, uniformly continuous potentials \\
$S_n\varphi$ & Birkhoff sum $\sum_{k=0}^{n-1}\varphi\circ T^k$ \\
$B_n(x,T,\varepsilon)$, $d_n$ & Bowen ball and Bowen metric (Section~\ref{sec:prelim-computable-analysis}) \\
$\Var(\varphi,\varepsilon)$ & modulus $\sup\{|\varphi(x)-\varphi(y)| : d(x,y)\leq\varepsilon\}$ \\
$\varphi(x,n,\varepsilon)$, $u(x,n,\varepsilon)$ & sup-weights $\sup_{y\in B_n(x,T,\varepsilon)}S_n\varphi(y)$ (Section~\ref{subsec:lower-pressure}, Definition~\ref{def:lower-bs}) \\
$\underline{u}(x)$, $\overline{u}(x)$ & $\liminf_n S_nu(x)/n$ and $\limsup_n S_nu(x)/n$ \\
\end{tabular}

\medskip
\noindent\textbf{Pointwise quantities.}\par\medskip
\noindent\begin{tabular}{@{}p{0.30\textwidth}p{0.64\textwidth}@{}}
$K^A_n(x,T,\varepsilon)$ & time-$n$ orbit complexity (Section~\ref{subsec_orbit_complexity}) \\
$\underline{K}^A(x,T)$, $\overline{K}^A(x,T)$ & lower and upper orbit complexity (Section~\ref{subsec_orbit_complexity}) \\
$\underline{P}^A(x,T,\varphi)$, $\overline{P}^A(x,T,\varphi)$ & lower and upper pointwise pressure (Definition~\ref{def:pointwise-pressure}) \\
$\dim^A_{\mathrm{BS}}(x,T,u)$, $\Dim^A_{\mathrm{BS}}(x,T,u)$ & lower and upper pointwise BS dimension (Definition~\ref{def:pointwise-bs}) \\
\end{tabular}

\medskip
\noindent\textbf{Set-level quantities.}\par\medskip
\noindent\begin{tabular}{@{}p{0.30\textwidth}p{0.64\textwidth}@{}}
$\dim_H(Z)$, $\dim_P(Z)$ & Hausdorff and packing dimension \\
$\underline{h}(Z,T)$, $\overline{h}(Z,T)$ & Bowen (lower) and packing (upper) entropy (Definition~\ref{def:bowen-entropy}; Definition~\ref{def:packing-pressure} at $\varphi=0$) \\
$\underline{P}(Z,T,\varphi)$ & lower (Pesin--Pitskel) pressure (Definition~\ref{def:lower-pressure}) \\
$\overline{P}(Z,T,\varphi)$ & upper (packing) pressure, realized as $P_{\mathrm{mUC}}$ (Definition~\ref{def:mUC1}, Appendix~\ref{app:packing-pressure}) \\
$\dim_{\mathrm{BS}}(Z,T,u)$, $\Dim_{\mathrm{BS}}(Z,T,u)$ & lower and upper BS dimension (Definitions~\ref{def:lower-bs} and~\ref{def:mBSC}) \\
$P_{\mathrm{UC}}$, $P_{\mathrm{mUC}}$ & (modified) upper capacity pressure (Definition~\ref{def:mUC1}) \\
$\dim_{\mathrm{BSC}}$ & upper capacity BS dimension (Definition~\ref{def:mBSC}) \\
\end{tabular}

\medskip
\noindent\textbf{Discretizations and effective quantities.}\par\medskip
\noindent\begin{tabular}{@{}p{0.30\textwidth}p{0.64\textwidth}@{}}
null $s$-cover & Definitions~\ref{def:null-s-cover} and~\ref{def:null_cover}, and the BS analogue of Section~\ref{subsec:bs-lower} \\
uniform $s$-net & Definitions~\ref{def:s-net} and~\ref{def:bs-s-net} \\
$E_{n,p,k}$, $Z^{(p)}_k$ & slices of a uniform $s$-net and the induced decomposition of $Z$ \\
$\underline{h}^A_{\mathrm{eff}}$, $\underline{P}^A_{\mathrm{eff}}$, $\overline{P}^A_{\mathrm{eff}}$ & effective entropy and pressures, via $A$-c.e.\ discretizations \\
$\dim^A_{BS,\mathrm{eff}}$, $\Dim^A_{BS,\mathrm{eff}}$ & effective BS dimensions (Definitions~\ref{def:eff-lower-bs} and~\ref{def:eff-upper-bs}) \\
\end{tabular}

\medskip
\noindent\textbf{Appendix and conformal systems.}\par\medskip
\noindent\begin{tabular}{@{}p{0.30\textwidth}p{0.64\textwidth}@{}}
$P^{\mathrm{pack}}$, $\dim_{\mathrm{BSP}}$ & classical packing pressure and packing BS dimension (Definition~\ref{def:packing-pressure}; Appendix~\ref{app:packing-bs}) \\
$Q_n$, $P^{\mathrm{sep}}_{\mathrm{UC}}$, $P^{\mathrm{sep}}_{\mathrm{mUC}}$, $\Dim^{\mathrm{sep}}_{BS}$ & separated-set quantities (Definitions~\ref{def:UC-separated} and~\ref{def:mUCsep}) \\
$a(x)$, $\zeta(x,y)$ & conformal derivative and its two-variable extension; $u = \log a$ (Definition~\ref{def:conformal}) \\
$\mathcal{A}(\alpha_1,\alpha_2)$, $\mathcal{B}$, $\mathcal{C}$ & exponent range, tempered contraction, sublinear growth (Definition~\ref{def:conformal-sets}) \\
$\varepsilon_\delta$ & $\Var(\zeta,\delta)$, the modulus of the conformal derivative \\
\end{tabular}

\medskip
\noindent Throughout, $\log$ denotes the base-$2$ logarithm and
$\exp_2(y) = 2^y$.

\section{Preliminaries}
\label{sec:prelim}

Throughout the paper $\log$ denotes the base-$2$ logarithm, and all
complexity, entropy, pressure, and dimension quantities are normalized
accordingly. We write $\exp_2(y)=2^{y}$ where the exponent is large.

This section collects the computability-theoretic preliminaries. The two tools to take away are the Coding
Theorem of Section~\ref{sec:prelim-kolmogorov}
(Lemma~\ref{lem:coding-theorem}), which converts summable weight
estimates into complexity upper bounds and replaces the covering
combinatorics of the classical theory, and the minimization machinery
of Section~\ref{sec:cones}
(Lemmas~\ref{lem:monotone}--\ref{lem:cone-min}), which
produces the achieving oracles in every point-to-set principle. Readers
with a background in algorithmic randomness may skim directly to
Section~\ref{sec:prelim-computable-analysis} for the dynamical setup
and Section~\ref{sec:cones} for the minimization schema. Standard
references for everything else are \cite{downey_2010_algorithmic} and
\cite{robertirvingsoare_2016_turing}. We use the word computable where older sources use
recursive.

\subsection{Computability and oracles}
\label{sec:prelim-computability}

An oracle is a set $A \subset \mathbb{N}$, regarded as a countable
list of information available to a program. There is a standard
enumeration $\Phi^A_e$ of the $A$-partial computable functions, where
$e$ encodes an instruction set for a program with query access to $A$;
we treat the machine model as a black box and refer to
\cite{robertirvingsoare_2016_turing, downey_2010_algorithmic} for its construction. A set
$E \subset \mathbb{N}$ is \emph{$A$-computably enumerable}
($A$-c.e.) if it is the domain of some $\Phi^A_e$, equivalently the
range of an $A$-total computable function, and \emph{$A$-computable}
if both $E$ and its complement are $A$-c.e.

For oracles $A, B$ we write $A \leq_T B$ ($B$ \emph{computes} $A$) if
the characteristic function of $A$ is $B$-computable, and
$A \equiv_T B$ if both $A \leq_T B$ and $B \leq_T A$. The
\emph{Turing degrees} are the equivalence classes of $\equiv_T$, with
the partial order induced by $\leq_T$.

As sketched in Section~\ref{sec:intro-effdim}, the role oracles play
in this paper is to carry auxiliary data. Typical contents:
\begin{enumerate}
    \item a dense sequence of points in a complete separable metric
    space, together with approximations to the distances between them;
    \item an $\varepsilon$--$\delta$ description (a modulus) of a
    continuous function;
    \item a description of a measure $\mu$ through the values
    $\mu(B_i)$ on a countable basis;
    \item a description of a dynamical system $(X,d,T)$, comprising the
    metric data and the transformation;
    \item in general, any countable list of parameters one wishes an
    algorithm to consult.
\end{enumerate}
Any countable collection of such data may be encoded into a single
oracle, so one may always imagine working in the computable setting
relative to a list of everything needed.

\subsection{Computable analysis and computable dynamics}
\label{sec:prelim-computable-analysis}

\begin{definition}[$A$-computable metric space structure]
\label{def:computable-metric}
Let $(X,d)$ be a complete separable metric space. An
\emph{$A$-computable metric space structure} on $(X,d)$ is a dense
sequence $(x_i)_{i \in \mathbb{N}}$ in $X$ such that the distances are
$A$-computable uniformly: the map sending $(i,j,p)$ to a rational
$q$ with $|q - d(x_i,x_j)| < 2^{-p}$ is $A$-computable. We call the
$x_i$ \emph{ideal points}. The \emph{basic open sets} are the balls
$B_q(x_i)$ with $q \in \mathbb{Q}_{>0}$, encoded by the pairs
$\langle i, q \rangle$; we write $\mathcal{U}(X,d,(x_i))$ for this
encoded collection.
\end{definition}

\begin{lemma}
\label{lem:csms-oracle}
Every complete separable metric space $(X,d)$ admits an oracle $A$ and
an $A$-computable metric space structure.
\end{lemma}

\begin{proof}
Choose any dense sequence $(x_i)$ by separability and let $A$ encode
the binary expansions of the countably many reals $d(x_i, x_j)$.
\end{proof}

\begin{definition}[$A$-computable function]
\label{def:computable-function}
Let $(X,d,(x_i))$ and $(Y,d',(y_j))$ carry $A$-computable metric space
structures. A function $f \colon X \to Y$ is \emph{$A$-computable} if
there is an $A$-c.e.\ set
$W \subset \mathcal{U}(X,d,(x_i)) \times \mathcal{U}(Y,d',(y_j))$
such that:
\begin{enumerate}
    \item if $(\langle i,q \rangle, \langle j,r \rangle) \in W$ then
    $f(B_q(x_i)) \subset B_r(y_j)$;
    \item for every $x \in X$ and every basic ball
    $B_r(y_j) \ni f(x)$ there is a basic ball $B_q(x_i) \ni x$ with
    $(\langle i,q \rangle, \langle j,r \rangle) \in W$.
\end{enumerate}
$A$-computable functions are closed under composition, uniformly.
\end{definition}

\begin{lemma}
\label{lem:cont-oracle}
A function $f \colon (X,d) \to (Y,d')$ between complete separable
metric spaces is continuous if and only if there is an oracle $A$,
together with $A$-computable structures on both spaces, making $f$
$A$-computable.
\end{lemma}

\begin{proof}
If $f$ is $A$-computable then by (1) and (2) the preimage of every
basic open ball is a union of basic balls, hence open, so $f$ is
continuous. Conversely, if $f$ is continuous, fix structures as in
Lemma~\ref{lem:csms-oracle} and let $A$ additionally encode the full
witness relation
$W = \{(\langle i,q\rangle, \langle j,r\rangle) \mid f(B_q(x_i))
\subset B_r(y_j)\}$, which satisfies (2) by continuity.
\end{proof}

\begin{remark}[Independence of the structure]
\label{rem:structure-independence}
Let $(x_i)$ and $(x'_j)$ be two computable metric space structures on
the same space $(X,d)$, and let $B$ be any oracle computing both
structures together with the mixed distances $d(x_i, x'_j)$ uniformly
(such a $B$ exists above any oracle computing the two structures, by
Lemma~\ref{lem:csms-oracle} applied to the interleaved sequence). Then
each structure interprets the other by search: given $i$ and $p$, a
$B$-search through $j$ locates $x'_j$ with $d(x_i, x'_j) < 2^{-p}$,
which exists by density and is detected by the mixed-distance data.
Consequently, for every oracle $A \geq_T B$, a $2^{-p}$-approximation
of $x$ in one structure converts to a $2^{-(p-1)}$-approximation in
the other at cost $O(1)$ in complexity, so all pointwise quantities in
this paper computed relative to $A$ agree in the two structures. By
Lemma~\ref{lem:cone-min} below, the set-level characterizations are
therefore independent of the choice of structure outright.
\end{remark}

We now fix the standing dynamical data. Throughout the paper $(X,d)$
is a complete separable metric space, $T \colon X \to X$ is
continuous, and the potentials under consideration
($\varphi$ and $u$, as the context requires) are uniformly continuous
real-valued functions on $X$; the real line carries its standard
computable structure.

\begin{definition}[Admissible oracle]
\label{def:admissible}
An \emph{$A$-computable dynamical system structure} for
$(X,d,T;\varphi_1,\dots,\varphi_k)$ is an $A$-computable metric space
structure $(x_i)$ on $(X,d)$ relative to which $T$ and each potential
$\varphi_1,\dots,\varphi_k$ are $A$-computable. We call $A$
\emph{admissible} (for the standing data) if such a structure exists,
and whenever $A$ is admissible we regard a structure as fixed.
\end{definition}

\begin{remark}
\label{rem:admissible-cone}
Admissible oracles form an upper set of Turing degrees containing an
upper cone: if $A$ is admissible and $B \geq_T A$, the same structure
is $B$-computable; and some admissible oracle exists by
Lemmas~\ref{lem:csms-oracle} and~\ref{lem:cont-oracle} together with a
finite join. Combined with Lemma~\ref{lem:cone-min} below, every
point-to-set minimum in this paper may be computed among admissible
oracles.
\end{remark}

For $n \in \mathbb{N}$ and $\varepsilon > 0$ the \emph{Bowen ball} of
order $n$ and radius $\varepsilon$ is
\[
B_n(x, T, \varepsilon)
= \{\, y \in X \mid d(T^i x, T^i y) < \varepsilon
\text{ for } 0 \leq i < n \,\},
\]
the open ball of the Bowen metric
$d_n(x,y) = \max\limits_{0 \leq i < n} d(T^i x, T^i y)$. Since $T$ is
continuous, each $d_n$ induces the topology of $d$; in particular the
ideal points $(x_i)$ are dense in every Bowen metric. We write
$S_n \varphi = \sum_{k=0}^{n-1} \varphi \circ T^k$ for Birkhoff sums.
The following uniformity is used throughout
Section~\ref{sec_point_to_set}.

\begin{lemma}
\label{lem:birkhoff-uniform}
Let $A$ be admissible with structure $(x_i)$, and let $\varphi$ be one
of the standing potentials. Then $S_n\varphi(x_i)$ is $A$-computable
uniformly in $(i,n)$: the map sending $(i,n,p)$ to a rational $q$ with
$|q - S_n\varphi(x_i)| < 2^{-p}$ is $A$-computable.
\end{lemma}

\begin{proof}
Each ideal point $x_i$ is an $A$-computable point of $X$, and $T$,
$\varphi$ are $A$-computable, so each $\varphi(T^k x_i)$ is an
$A$-computable real, uniformly in $(i,k)$, by closure of
$A$-computable functions under composition; sum $n$ of them to
precision $2^{-(p + \lceil \log n \rceil + 1)}$ each.
\end{proof}

\subsection{Kolmogorov complexity and algorithmic randomness}
\label{sec:prelim-kolmogorov}

\begin{definition}[Plain and prefix-free complexity]
\label{def:complexity}
Fix a universal oracle Turing machine $M$ and a universal prefix-free
oracle machine $U$ (one whose domain is an antichain under the prefix
relation for every oracle). See \cite{downey_2010_algorithmic} for the
constructions. The \emph{plain $A$-Kolmogorov complexity} of a finite
binary string $\sigma$ is
$C^A(\sigma) = \min\{ |\tau| \mid M^A(\tau) = \sigma \}$, and the
\emph{prefix-free $A$-Kolmogorov complexity} is
$K^A(\sigma) = \min\{ |\tau| \mid U^A(\tau) = \sigma \}$. For a set
$S$ of strings, $K^A(S) = \min\limits_{\sigma \in S} K^A(\sigma)$. Natural
numbers and tuples of natural numbers are identified with binary
strings through fixed computable codings, so $K^A(i)$, $K^A(i,n)$,
etc.\ are defined.
\end{definition}

Different universal machines change these quantities by additive
constants only, since any universal machine simulates any other at
constant cost. Every asymptotic statement in this paper is insensitive
to the choice. The two notions are related by
\[
C^A(\sigma) \leq K^A(\sigma) + O(1)
\leq C^A(\sigma) + 2\log|\sigma| + O(1),
\]
the first inequality because prefix-free machines are machines, the
second because a plain program can be made self-delimiting by
prefixing a self-delimiting description of its length. We use $K^A$
exclusively from here on, relying on two standard consequences of prefix-freeness.

The first is the Kraft--Chaitin theorem: an effectively listed set of requests for short descriptions can be met, provided the total requested weight is bounded.

\begin{lemma}[Kraft--Chaitin; see \cite{downey_2010_algorithmic}]
\label{lem:kraft-chaitin}
Let $R$ be an $A$-c.e.\ set of \emph{requests} $(n_l, \sigma_l)$,
$n_l \in \mathbb{N}$, $\sigma_l \in 2^{<\mathbb{N}}$, with
$\sum_l 2^{-n_l} \leq 1$. Then there is a prefix-free $A$-machine $N$
such that for every request there is $\tau_l$ with $|\tau_l| = n_l$
and $N^A(\tau_l) = \sigma_l$. Consequently
$K^A(\sigma_l) \leq n_l + O(1)$, the constant depending only on $R$.
Conversely (Kraft inequality), $\sum_\sigma 2^{-K^A(\sigma)} \leq 1$.
\end{lemma}

The form we actually apply is an easy consequence, used extensively in Section~\ref{sec_point_to_set}: weight functions built from Birkhoff sums are shown to have bounded total mass over a c.e.\ index set, and one concludes a complexity bound on the indices.

\begin{lemma}[Coding Theorem]
\label{lem:coding-theorem}
Let $F$ be an $A$-c.e.\ set of tuples and let $f \colon F \to
\mathbb{R}$ be upper semicomputable in $A$ uniformly on $F$ (in
particular, $A$-computable suffices), with
\[
\sum_{w \in F} 2^{-f(w)} \leq 1.
\]
Then $K^A(w) \leq f(w) + O(1)$ for all $w \in F$, the constant
depending only on $(F, f)$.
\end{lemma}
We will need the following extension, which relaxes the bound $1$ on the sum to an arbitrary finite bound.
\begin{lemma}[Extended coding Theorem]\label{lem:extended-coding-theorem}
    Let $F$ be an $A$-c.e.\ set of tuples and let $f \colon F \to
\mathbb{R}$ be upper semicomputable in $A$ uniformly on $F$ (in
particular, $A$-computable suffices), with
\[
\sum_{w \in F} 2^{-f(w)} \leq M<\infty.
\]
Then $K^A(w) \leq f(w) + O(1)$ for all $w \in F$, the constant
depending only on $(F, f)$.
\end{lemma}
\begin{proof}
Define $F_{\geq N}=\{w\in F\mid |w|\geq N\}$. Since
\[
\sum_{w \in F_{\geq N}} 2^{-f(w)} \to 0 \quad \text{as } N\to\infty,
\]
there is $N_0$ with $\sum_{w \in F_{\geq N_0}} 2^{-f(w)}\leq 1$, so the coding theorem applies to $F_{\geq N_0}$. Since $F\setminus F_{\geq N_0}$ is finite, enlarging the constant covers the finitely many remaining $w$, and $K^A(w)\leq f(w)+O(1)$ holds on all of $F$.
\end{proof}

In the algorithmic randomness literature, functions $f$ with
$\sum 2^{-f} \leq 1$ and upper semicomputable are called
($A$-)\emph{information content measures}, and
Lemma~\ref{lem:coding-theorem} is the minimality of $K^A$ among them
\cite{downey_2010_algorithmic}. We will not need the terminology again.
\begin{remark}\label{rem:ce-thresholds}
Sets of the form $\{w : K^A(w) < f(w)\}$, where $f$ is
$A$-computable (or lower semicomputable in $A$) uniformly, are
$A$-c.e.: enumerate $w$ upon witnessing
$K^A_\tau(w) < q < f_\tau(w)$ for some stage $\tau$ and rational
$q$, which occurs if and only if the strict inequality holds.
\end{remark}
We also use the fact that complexity is subadditive under
computable composition up to a uniform constant (depending on the number $n$ of compositions happening).

\begin{lemma}[Composition lemma]
\label{lem:composition}
Fix $n$ and let $N$ be an $A$-machine such that
$N^A(\sigma_1, \dots, \sigma_n) \in Q$ whenever each
$\sigma_i \in P_i$, for sets of strings $P_1, \dots, P_n, Q$. Then
\[
K^A(Q) \leq \sum_{i=1}^{n} K^A(P_i) + O(1),
\]
with the constant uniform in the $P_i, Q$ for fixed $N$ and $n$.
\end{lemma}

\begin{proof}
The machine that reads $n$ prefix-free programs in sequence, runs each
in $U^A$, and applies $N^A$ to the outputs is itself prefix-free
(self-delimiting inputs concatenate to a self-delimiting input), so it
is simulated by $U^A$ at constant cost; run it on minimizing programs
for the $P_i$.
\end{proof}

\begin{remark}
\label{rem:composition-use}
Lemma~\ref{lem:composition} is invoked with the $\sigma_i$ encoding
orbit data (indices of ideal points approximating an orbit segment)
and lengths of time for which the data is valid. See
Section~\ref{sec_point_to_set}. Two specializations recur often enough to
record: $K^A(i) \leq K^A(i,n) + O(1)$ and, in the other direction,
\[
K^A(i,n) \leq K^A(i) + 2\log n + O(1),
\]
since $n$ admits a self-delimiting description of length
$2 \log n + O(1)$.
\end{remark}

Finally, algorithmic randomness is used in this paper only to produce maximally complex points, which witness lower bounds in examples such as
Example~\ref{ex:intro-cantor}. A sequence $x \in 2^{\mathbb{N}}$ is
\emph{$A$-Martin-L\"of random} if it passes every $A$-effective
sequence of measure-vanishing tests. By the Levin--Schnorr theorem
this is equivalent to the complexity characterization we use:
$x$ is $A$-Martin-L\"of random if and only if
\[
K^A(x{\restriction}n) \geq n - O(1),
\]
where $x{\restriction}n$ is the length-$n$ prefix. Almost every
sequence (for the uniform measure) is $A$-Martin-L\"of random, so such
witnesses always exist. See \cite{downey_2010_algorithmic} for the test
definition, the Levin--Schnorr theorem, and history.

Randomness and measure play a substantive role in the companion paper
\cite{dinowitz2026dimension}.

We record one further tool, used in Example~\ref{ex:symbolic} and in
Theorem~\ref{thm:pointwise-conformal}, which lets us compute the
pointwise dimensions along a sufficiently slowly varying sequence of
scales rather than along all scales.

\begin{lemma}[Dimension along a sequence of scales]
\label{lem:scale-seq}
Let $(\delta_m)_{m\in\mathbb{N}}$ be a sequence of positive scales with
$\delta_m \to 0$ and
$\log\delta_{m+1}/\log\delta_m \to 1$. Then for every oracle $A$ and
every $x$,
\[
\dim^A(x)=\liminf\limits_{m\to\infty}
\frac{K^A_{\delta_m}(x)}{-\log\delta_m},
\qquad
\Dim^A(x)=\limsup\limits_{m\to\infty}
\frac{K^A_{\delta_m}(x)}{-\log\delta_m}.
\]
\end{lemma}
\begin{proof}
Since the sequence $(\delta_m)$ runs over a subfamily of scales,
\[
\dim^A(x)\leq\liminf\limits_{m\to\infty}
\frac{K^A_{\delta_m}(x)}{-\log\delta_m},
\qquad
\Dim^A(x)\geq\limsup\limits_{m\to\infty}
\frac{K^A_{\delta_m}(x)}{-\log\delta_m},
\]
so only the reverse inequalities require proof.

Since $\delta_m\to 0$, for every $r\in(0,\delta_1]$ the set
$\{m\mid \delta_m\geq r\}$ is finite and nonempty, so we may define
\[
m(r)=\max\{m\mid \delta_m\geq r\},
\]
which satisfies
\[
\delta_{m(r)+1}<r\leq\delta_{m(r)}.
\]
Moreover $m(r)\to\infty$ as $r\to 0$: for any $M$, once
$r\leq\delta_M$ we have $m(r)\geq M$.

The quantity $K^A_\delta(x)$ is nonincreasing in $\delta$, since
enlarging the ball enlarges the set over which the minimum is taken.
Hence
\[
K^A_{\delta_{m(r)}}(x)\;\leq\;K^A_r(x)\;\leq\;K^A_{\delta_{m(r)+1}}(x),
\]
while
$0<-\log\delta_{m(r)}\leq-\log r<-\log\delta_{m(r)+1}$ once $r<1$.
Combining the two brackets,
\[
\frac{K^A_{\delta_{m(r)}}(x)}{-\log\delta_{m(r)}}
\cdot\frac{\log\delta_{m(r)}}{\log\delta_{m(r)+1}}
\;\leq\;
\frac{K^A_r(x)}{-\log r}
\;\leq\;
\frac{K^A_{\delta_{m(r)+1}}(x)}{-\log\delta_{m(r)+1}}
\cdot\frac{\log\delta_{m(r)+1}}{\log\delta_{m(r)}}.
\]
As $r\to 0$ we have $m(r)\to\infty$, so by hypothesis both ratios of
logarithms tend to $1$. Taking $\liminf_{r\to 0}$ in the left-hand
bound therefore gives
\[
\dim^A(x)
=\liminf\limits_{r\to 0}\frac{K^A_r(x)}{-\log r}
\;\geq\;
\liminf\limits_{r\to 0}\frac{K^A_{\delta_{m(r)}}(x)}{-\log\delta_{m(r)}}
\;\geq\;
\liminf\limits_{m\to\infty}\frac{K^A_{\delta_m}(x)}{-\log\delta_m},
\]
where the last inequality holds because $m(r)\to\infty$, and taking
$\limsup_{r\to 0}$ in the right-hand bound gives
\[
\Dim^A(x)
=\limsup\limits_{r\to 0}\frac{K^A_r(x)}{-\log r}
\;\leq\;
\limsup\limits_{r\to 0}\frac{K^A_{\delta_{m(r)+1}}(x)}{-\log\delta_{m(r)+1}}
\;\leq\;
\limsup\limits_{m\to\infty}\frac{K^A_{\delta_m}(x)}{-\log\delta_m}.
\qedhere
\]
\end{proof}
\subsection{The oracle order: cones, joins, and minimization}
\label{sec:cones}

\begin{definition}[Cones, upper sets, joins]
\label{def:cones-joins}
The \emph{upper cone} above an oracle $A$ is
$\uparrow\! A = \{ B \mid A \leq_T B \}$. A set of Turing degrees is
an \emph{upper set} if it contains the upper cone above each of its
members. The \emph{join} of $A$ and $B$ is the interleaving
$(A \oplus B)_{2n} = A_n$, $(A \oplus B)_{2n+1} = B_n$; the join
$\bigoplus_n A_n$ of countably many oracles is defined through a fixed
computable pairing of $\mathbb{N}^2$ with $\mathbb{N}$.
\end{definition}

\begin{lemma}
\label{lem:join-cone}
$\uparrow\!(A \oplus B) = \uparrow\! A \cap \uparrow\! B$.
\end{lemma}

\begin{proof}
$A \oplus B$ computes each of $A, B$ by projection, and any $C$
computing both computes the interleaving.
\end{proof}

Throughout this subsection, $\mathrm{d}^A$ stands for any one of the
pointwise quantities of this paper ($\dim^A$, $\Dim^A$, the orbit
complexities $\underline{K}^A, \overline{K}^A$, the pointwise
pressures $\underline{P}^A, \overline{P}^A$, or the pointwise BS
dimensions $\dim^A_{\mathrm{BS}}, \Dim^A_{\mathrm{BS}}$), and
$Q$ stands for the corresponding set-level quantity.

Every point-to-set principle in this paper is proved by the same
four-step pattern, which we outline here once and instantiate
quantity by quantity in Section~\ref{sec_point_to_set}.
\begin{enumerate}
    \item \emph{Dimensional definition.} The set-level quantity $Q$
    is a dimension quantity defined through covers by Bowen balls.
    \item \emph{Discretization.} Covers and packings are replaced
    by subsets of $\mathbb{N}^3$ or $\mathbb{N}^4$ indexing Bowen
    balls centered at ideal points (the null $s$-covers and uniform
    $s$-nets of Section~\ref{sec_point_to_set}), without changing
    the critical value. A level is proven characterizing the dimension as the $\inf$ over the set of all $s$ where either null s-covers or uniform s-nets exists
    \item \emph{Effectivization.} Restricting to $A$-computably
    enumerable discretizations defines an effective quantity
    $Q^A_{\mathrm{eff}}(Z)$. Two facts are proved about it: the
    inequality $Q(Z) \leq Q^A_{\mathrm{eff}}(Z)$, valid for every
    oracle because fewer discretizations are admitted; and, for
    admissible $A$ (Definition~\ref{def:admissible}), the pointwise
    characterization
    $Q^A_{\mathrm{eff}}(Z) = \sup_{x \in Z}\mathrm{d}^A(x)$,
    with the Coding Theorem (Lemma~\ref{lem:coding-theorem})
    giving one inequality and c.e.\ threshold sets
    (Remark~\ref{rem:ce-thresholds}) the other. Together these give
    $Q(Z) \leq \sup_{x\in Z}\mathrm{d}^A(x)$ for every
    admissible $A$.
    \item \emph{Oracle minimization.} Encoding near-optimal
    discretizations into a single oracle produces almost-achieving
    oracles; since the admissible oracles contain an upper cone
    (Remark~\ref{rem:admissible-cone}), the schema below
    (Lemma~\ref{lem:achieving-schema}, applied with $C$ admissible)
    then delivers the point-to-set principle, with the minimum
    attained on an upper set of Turing degrees.
\end{enumerate}
The next lemmas supply the oracle-monotonicity and minimization
mechanics this pattern relies on.

\begin{lemma}[Monotonicity]
\label{lem:monotone}
If $A \leq_T B$ then $K^B(\sigma) \leq K^A(\sigma) + O(1)$ for all
$\sigma$, and consequently $\mathrm{d}^B(x) \leq \mathrm{d}^A(x)$ for
every $x$ and every quantity $\mathrm{d}$ in the family above (for the
BS pair, at points with $\underline{u}(x)>0$, so that the normalizing
denominators $S_nu(x_i)$ tend to infinity along realizers and the
additive constant is absorbed).
\end{lemma}

\begin{proof}
A $B$-machine simulates any $A$-machine after computing $A$, at
constant cost. Each quantity in the family is obtained from
$\sigma \mapsto K^A(\sigma)$ by minimization over realizers followed
by limit operations, all monotone, in which the additive constant is
absorbed by normalizing denominators tending to infinity.

\end{proof}

\begin{lemma}[Achieving-oracle schema]
\label{lem:achieving-schema}
Let $(Q, \mathrm{d})$ be a corresponding pair and let $C$ be an
oracle such that $Q(Z) \leq \sup_{x \in Z} \mathrm{d}^{A}(x)$
for every $A \geq_T C$. Suppose there are oracles $A_j$ with
$\sup_{x \in Z} \mathrm{d}^{A_j}(x) \leq Q(Z) + 2^{-j}$ for each $j$.
Then $B = C \oplus \bigoplus_j A_j$ satisfies
$Q(Z) = \sup_{x \in Z} \mathrm{d}^B(x)$, as does every oracle
$\geq_T B$. In particular
\[
Q(Z) = \min\limits_{A \subset \mathbb{N}} \sup_{x \in Z} \mathrm{d}^A(x),
\]
with the minimum attained on an upper set of Turing degrees.
\end{lemma}

\begin{proof}
By Lemma~\ref{lem:monotone},
$\sup_x \mathrm{d}^B(x) \leq \sup_x \mathrm{d}^{A_j}(x)
\leq Q(Z) + 2^{-j}$ for every $j$, so
$\sup_x \mathrm{d}^B(x) \leq Q(Z)$; the reverse inequality holds
since $B \geq_T C$. The same computation applies to any oracle
above $B$. Finally, for an arbitrary oracle $A$, monotonicity gives
$\sup_x \mathrm{d}^A(x) \geq \sup_x \mathrm{d}^{A \oplus B}(x)
= Q(Z)$, so the minimum over all oracles equals $Q(Z)$ and is
attained on the upper cone above $B$.
\end{proof}

\begin{lemma}[Minimization within a cone]
\label{lem:cone-min}
For every oracle $C$,
\[
\min\limits_{A \subset \mathbb{N}} \sup_{x \in Z} \mathrm{d}^A(x)
= \min\limits_{A \geq_T C} \sup_{x \in Z} \mathrm{d}^A(x).
\]
\end{lemma}

\begin{proof}
If the left minimum is attained at $B$, then by
Lemma~\ref{lem:monotone} applied to $B \leq_T B \oplus C$,
$\sup_x \mathrm{d}^{B \oplus C}(x) \leq \sup_x \mathrm{d}^B(x)$, and
$B \oplus C \geq_T C$; the reverse inequality holds because the right
side minimizes over fewer oracles.
\end{proof}

\begin{remark}
\label{rem:lower-bound-cone-insensitive}
Lemma~\ref{lem:monotone} makes the lower-bound proof strategy of the
introduction cone-insensitive: to show
$\min\limits_A \sup_x \mathrm{d}^A(x) \geq s$ it suffices to produce, for
every oracle in the cone above some fixed $C$, a point of
$A$-pointwise dimension at least $s - \varepsilon$, since for
arbitrary $B$ the witness for $B \oplus C$ already gives
$\sup_x \mathrm{d}^B(x) \geq \sup_x \mathrm{d}^{B \oplus C}(x)
\geq s - \varepsilon$. Upper bounds, by contrast, are genuinely
cone-relative. A single oracle $B$ bounding every point bounds the
minimum, by Lemma~\ref{lem:cone-min}.
\end{remark}
\section{The point-to-set principle for dynamical dimension quantities}
\label{sec_point_to_set}
In \cite{galatolo-2009} the authors developed a framework for working with a computable
version of lower topological entropy. With minimal work this framework
can be turned into a point-to-set principle for lower entropy, which we
exposit in Section~\ref{subsec_orbit_complexity}. Since the results there
are relativizations of \cite{galatolo-2009}, we give only sketches
indicating the necessary changes. In the remainder of the section we
construct two frameworks directly generalizing this principle for
lower and upper topological pressure, and for lower and upper BS
dimension with full proofs. This provides a complete analogy with the point-to-set
principles for Hausdorff and packing dimension in
\cite{lutz-2023}. Entropy is a special case of both pressure and BS dimension, being
the topological pressure with respect to the potential $\varphi \equiv 0$
and the BS dimension with respect to $u \equiv  1$.
\subsection{Orbit complexity and the point-to-set principle for lower
entropy}\label{subsec_orbit_complexity}
In this subsection we adapt the results in \cite{galatolo-2009}. The
main difference in our formulation of orbit complexity is that we use Bowen balls
$B_n(x,T,\varepsilon)=\{y\mid d(T^ix,T^iy)<\varepsilon \text{ for }
0\leq i< n\}$ in place of pseudo-orbits. As noted in
\cite{galatolo-2009}, the two definitions agree at scale $0$ by a
search argument. Throughout this section we fix an admissible oracle
$A$ for $(X,d)$ with structure $(x_i)_{i\in\mathbb{N}}$
(Definition~\ref{def:admissible}).\\
The proofs in this subsection are sketches. Each is subsumed by the
corresponding pressure and BS dimension arguments proved in full in
Sections~\ref{subsubsec:lower-pressure-pts}, and~\ref{subsubsec:lower-bs-pts}, which
recover the entropy statements as the special cases $\varphi = 0$ and
$u \equiv 1$.\\
Our definitions of the dimension quantities throughout this section
differ from the formulation in \cite{galatolo-2009} and from the
classical formulations in two ways chosen to make them readily
discretizable. We measure covers and packings by Bowen balls rather
than by the $\varepsilon$-size of arbitrary open sets, and we allow
centers to range over all of $X$ rather than over the set $Z$ under
study, so that centers may be moved onto the dense sequence of ideal
points. Both modifications leave the critical values unchanged. The first is
justified by the scale-$0$ comparison of \cite{galatolo-2009}, the
second by a standard doubling argument at cost
$\Var(\varphi,\varepsilon)$. For the upper quantities we
additionally work with the modified upper capacity formulations of
packing pressure and packing BS dimension, which discretize directly
through uniform $s$-nets. The equivalence with the classical packing
definitions is established in the appendix.

\begin{definition}[Orbit complexity]
The time-$n$ orbit complexity of $x$ at scale $\varepsilon$ (with
respect to $T:X\to X$) is
\[K^A_n(x,T,\varepsilon)=\min\{\,K^A(i)\mid x_i\in
B_n(x,T,\varepsilon)\,\}.\]
The $A$-lower and $A$-upper orbit complexities of $x$ are
\[\underline{K}^A(x,T,\varepsilon)=\liminf\limits_{n\to\infty}
\frac{K^A_n(x,T,\varepsilon)}{n},
\qquad
\underline{K}^A(x,T)=\lim\limits_{\varepsilon\to 0}
\underline{K}^A(x,T,\varepsilon),\]
\[\overline{K}^A(x,T,\varepsilon)=\limsup\limits_{n\to\infty}
\frac{K^A_n(x,T,\varepsilon)}{n},
\qquad
\overline{K}^A(x,T)=\lim\limits_{\varepsilon\to 0}
\overline{K}^A(x,T,\varepsilon).\]
\end{definition}

We recall a number of definitions and theorems from
\cite{galatolo-2009}.
%-------------------------------------------------
%  ε–size and 2–entropy dimension (following Galatolo et al.)
%-------------------------------------------------
\begin{definition}[Bowen topological entropy for arbitrary sets]\label{def:bowen-entropy}
    Let \(X\) be a metric space and \(T:X\to X\) a continuous map.  
The \(\varepsilon\)-\emph{size} of a set \(E\subset X\) is the dyadic
value \(2^{-n_\varepsilon(E)}\), where
\[
n_\varepsilon(E) \;=\; \sup\!\bigl\{\, n\ge 0 :
\operatorname{diam}\!\bigl(T^{i}E\bigr)\le\varepsilon
\text{ for } 0\le i<n \bigr\}.
\]
In words, \(n_\varepsilon(E)\) measures how long the orbits starting
from \(E\) stay \(\varepsilon\)-close. As \(\varepsilon\downarrow 0\), the \(\varepsilon\)-size of any set is non-decreasing.\\
For instance, the \(2\varepsilon\)-size of a Bowen ball \(B_{n}(x,T,\varepsilon)\) is at most \(2^{-n}\).

\medskip
Mimicking the construction of Hausdorff measure, define

\[
m_\delta^{s}(Z,\varepsilon)\;=\;
\inf_{\mathcal{G}}
\Bigl\{\;
   \sum_{U\in\mathcal{G}}\bigl(\varepsilon\text{-size}(U)\bigr)^{s}
\Bigr\},
\]

where the infimum ranges over all countable covers
\(\mathcal{G}\) of \(Z\) by open sets of \(\varepsilon\)-size \(<\delta\).
Because \(m_\delta^{s}(Z,\varepsilon)\) is monotonically \emph{increasing} as \(\delta\to 0\),
the limit  

\[
m^{s}(Z,\varepsilon)\;:=\;\lim\limits_{\delta\to 0^{+}} m_\delta^{s}(Z,\varepsilon)
\]

exists and equals the supremum of the values \(m_\delta^{s}\).
There is a critical exponent \(s_{0}\) such that  

\[
m^{s}(Z,\varepsilon)=\infty \quad\text{for } s<s_{0},
\qquad
m^{s}(Z,\varepsilon)=0     \quad\text{for } s>s_{0}.
\]

Set this critical value to be

\[
\underline{h}(Z,T,\varepsilon)
   \;:=\;
   \inf\{s : m^{s}(Z,\varepsilon)=0\}
   \;=\;
   \sup\{s : m^{s}(Z,\varepsilon)=\infty\}=s_0.
\]

Since the class of admissible covers shrinks as \(\varepsilon\to 0\)
(the \(\varepsilon\)-size of any fixed set does not decrease),
the following limit exists:

\[
\underline{h}(Z,T)
   \;:=\;
   \lim\limits_{\varepsilon\to 0^{+}} \underline{h}(Z,T,\varepsilon),
\]
and it is a supremum.\\
\end{definition}
This is an equivalent form of the base-$2$ definition of Bowen
entropy. In \cite{galatolo-2009} it is called $h_2$, to distinguish it
from upper capacity entropy, there called $h_1$. Next we discretize
this notion in terms of subsets of $\mathbb{N}^3$.

\begin{definition}[Null $s$-cover for Bowen entropy]\label{def:null-s-cover}
Let $Z \subset X$. A \emph{null $s$-cover} of $Z$ is a set
$E \subset \mathbb{N}^{3}$ satisfying
\begin{enumerate}
  \item $\displaystyle\sum_{(i,n,p)\in E} 2^{-s n} < \infty$;
  \item for every $k,p \in \mathbb{N}$, the family
        $\bigl\{\, B_{n}(x_{i},T,2^{-p}) : (i,n,p)\in E,\ n \ge k
        \bigr\}$ is a cover of $Z$.
\end{enumerate}
\end{definition}
Condition 1 is the $\varepsilon$-size sum of
Definition~\ref{def:bowen-entropy}, restricted to covers by indexed
Bowen balls.\\
Galatolo, Hoyrup, and Rojas proved the following
\cite[Theorem~2.4]{galatolo-2009}.

\begin{lemma}\label{lem:h2-null-s-cover}
For a continuous map $T:X\to X$ and $Z\subset X$,
\[
\underline{h}(Z,T)\;=\;\inf\bigl\{\,s : Z\text{ admits a null
}s\text{-cover}\bigr\}.
\]
\end{lemma}

They also gave an effective version of this characterization, which we
relativize to an arbitrary oracle $A$.

\begin{definition}[$A$-effective null $s$-cover]
An \emph{$A$-effective null $s$-cover} is a null $s$-cover
$E\subset\mathbb{N}^{3}$ that is $A$-computably enumerable.
\end{definition}
%-------------------------------------------------
%  Effective topological entropy
%-------------------------------------------------
\begin{definition}[$A$-effective topological entropy]
\label{def:eff-top-entropy}
For a continuous map $T$ and $Z\subset X$, the \emph{$A$-effective
topological entropy} is
\[
  \underline{h}^A_{\mathrm{eff}}(Z,T)
     \;=\;
     \inf\bigl\{\,s : Z
        \text{ admits an $A$-effective null $s$-cover}\bigr\}.
\]
The unrelativized notion ($A = \emptyset$) is the effective topological
entropy of \cite{galatolo-2009}.
\end{definition}

Because fewer $s$-covers are allowed in the $A$-effective setting,
$\underline{h}(Z,T)\le \underline{h}^A_{\mathrm{eff}}(Z,T)$ for every oracle $A$; and if
$Z\subset Z'$ then
$\underline{h}^A_{\mathrm{eff}}(Z,T)\le \underline{h}^A_{\mathrm{eff}}(Z',T)$.\\
We now give the first step of the entropy point-to-set principle: a
relativization lemma saying the Bowen entropy of any set is matched by
a sufficiently strong oracle.

\begin{lemma}\label{exists-oracle}
Given a set $Z \subset X$ there is an oracle $A$ with
\[\underline{h}(Z,T)=\underline{h}_{\mathrm{eff}}^A(Z,T).\]
\end{lemma}
\begin{proof}
By Lemma~\ref{lem:h2-null-s-cover}, choose null $s_j$-covers
$E_j \subset \mathbb{N}^3$ for $Z$ with
$s_j \downarrow \underline{h}(Z,T)$, and let
$A = \bigoplus_j E_j$. Each $E_j$ is $A$-c.e., hence an $A$-effective
null $s_j$-cover, so $\underline{h}^A_{\mathrm{eff}}(Z,T) \le s_j$ for every $j$;
the reverse inequality holds for every oracle.
\end{proof}

The principle then follows from the pointwise characterization of
effective entropy.

\begin{theorem}[Effective entropy is the sup of lower orbit
complexity]\label{thm:eff-entropy-loc}
Let $A$ be admissible. For every $Z\subset X$,
\[
  \underline{h}^A_{\mathrm{eff}}(Z,T)
     \;=\;
     \sup_{x\in Z}\,\underline{K}^A(x,T).
\]
\end{theorem}
\begin{proof}
Relativize the proof of \cite[Theorem~6.2.2]{galatolo-2009} to $A$:
every computable or effective object, quantity, and theorem in that
proof is replaced by its $A$-relativized version.\\
Alternatively, this is the case $\varphi = 0$ of
Theorem~\ref{thm:sup_eq_pressure_lower}, proved in full in
Section~\ref{subsec:lower-pressure}.
\end{proof}

\begin{theorem}[Point-to-set principle for entropy]
\label{thm:pts-entropy}
Let $T:X\to X$ be a continuous map on a complete separable metric
space and let $Z\subset X$. Then
\[\underline{h}(Z,T)=\min\limits_{A\subset\mathbb{N}}\,\sup_{x\in Z}
\underline{K}^A(x,T),\]
with the minimum attained on an upper set of Turing degrees.
\end{theorem}
\begin{proof}
For every oracle $A$ we have
$\underline{h}(Z,T)\le \underline{h}^A_{\mathrm{eff}}(Z,T)$, and for admissible $A$,
$\underline{h}^A_{\mathrm{eff}}(Z,T)=\sup_{x\in Z}\underline{K}^A(x,T)$ by
Theorem~\ref{thm:eff-entropy-loc}. By Lemma~\ref{exists-oracle} there
is an oracle $B$ with $\underline{h}(Z,T)=h^B_{\mathrm{eff}}(Z,T)$. Joining
with an admissible oracle (Remark~\ref{rem:admissible-cone}) and
using monotonicity (Lemma~\ref{lem:monotone}), we may take $B$
admissible, and the same holds for every oracle above $B$. The
minimum is therefore attained at $B$, on the upper cone above it.
\end{proof}
We now give two short applications of this theorem.
\begin{example}[Topological entropy of isometries]
\label{ex:isometry}
Let $f:X\to X$ be an isometry of a complete separable metric space.
Then $\underline{h}(X,f)=0$. Indeed, fix an admissible oracle $A$, $x \in
X$, and $\varepsilon>0$, and pick an ideal point $x_i$ with
$d(x,x_i)<\varepsilon$. Since $f$ is an isometry,
$d(f^m(x_i),f^m(x))<\varepsilon$ for all $m$, so
$K^A_n(x,f,\varepsilon)\le K^A(i)=O(1)$ and
$\underline{K}^A(x,f,\varepsilon)=0$. By
Theorem~\ref{thm:pts-entropy}, computing the minimum within the
admissible cone (Lemma~\ref{lem:cone-min}),
$\underline{h}(X,f)=\min\limits_{A\subset\mathbb{N}}\sup\limits_{x\in X}\underline{K}^A(x,f)=0$.
\end{example}
\begin{example}[Dimension and entropy for symbolic dynamics with
standard metric]\label{ex:symbolic}
Let $\mathcal{A}^\mathbb{N}$ be the full shift $\sigma$ on a finite or
countable alphabet $\mathcal{A}$, with the standard metric
$d(x,y)=\beta^{n(x,y)}$ for fixed $\beta\in(0,1)$, where $n(x,y)$ is
the length of the longest common prefix of $x$ and $y$. Then for every
$Z\subset\mathcal{A}^\mathbb{N}$,
\[\dim_H(Z)=\frac{\underline{h}(Z,\sigma)}{-\log\beta}.\]
To see this, fix an admissible oracle $A$ and $x \in
\mathcal{A}^\mathbb{N}$. At scale $\varepsilon=\beta^k$ a Bowen ball
of order $n$ is a metric ball of radius $\beta^{n+k}$, so
$K^A_n(x,\sigma,\beta^k)=K^A_{\beta^{n+k}}(x)$, and
\[\frac{K^A_n(x,\sigma,\beta^k)}{n}
=\frac{K^A_{\beta^{n+k}}(x)}{-(n+k)\log\beta}
\cdot(-\log\beta)\cdot\frac{n+k}{n}.\]
The scales $\delta_n=\beta^{n+k}$ satisfy the hypothesis of
Lemma~\ref{lem:scale-seq}, so
\begin{align*}
\underline{K}^A(x,\sigma,\beta^k)
&=\liminf\limits_{n\to\infty}\frac{K^A_n(x,\sigma,\beta^k)}{n}
=(-\log\beta)\,\liminf\limits_{n\to\infty}
 \frac{K^A_{\beta^{n+k}}(x)}{-(n+k)\log\beta}
=(-\log\beta)\,\dim^A(x).
\end{align*}
The right-hand side is independent of $k$, so
\[\underline{K}^A(x,\sigma)
=\lim\limits_{k\to\infty}\underline{K}^A(x,\sigma,\beta^k)
=(-\log\beta)\,\dim^A(x).\]
Taking the join of oracles achieving the minima in
Theorem~\ref{thm:pts-entropy} and the Hausdorff dimension
point-to-set principle (Lemma~\ref{lem:join-cone}), both minima are
computed at a single oracle, and
\[\dim_H(Z)=\frac{\underline{h}(Z,\sigma)}{-\log\beta}.\]
The same argument with $\limsup$ throughout, using the $\limsup$ half
of Lemma~\ref{lem:scale-seq}, gives
$\overline{K}^A(x,\sigma)=(-\log\beta)\,\Dim^A(x)$, and the
point-to-set principle for upper entropy
(Theorem~\ref{thm:pts_upper_pressure} at $\varphi=0$, proved later in
this section) then yields
\[\dim_P(Z)=\frac{\overline{h}(Z,\sigma)}{-\log\beta}.\]
\end{example}
This example should be contrasted with the Bowen pressure equation
results of Section~\ref{sec:bowen-equation} in the case
$\underline{u}(x)=\overline{u}(x)=\alpha$
(Corollary~\ref{cor:constant-exponent}). Here the shift is conformal
with constant derivative $a(x)=\beta^{-1}$, so $u=\log a=-\log\beta$,
and the example is the simplest instance of
Theorem~\ref{thm:conformal-dim}.

\subsection{The point-to-set principle for lower and upper topological
pressure}\label{subsec:lower-pressure}

\subsubsection{Pointwise pressure}
Throughout Section~\ref{subsec:lower-pressure}, $T:X\to X$ is a
continuous map on a complete separable metric space $(X,d)$ and
$\varphi:X\to\mathbb{R}$ is uniformly continuous. For the pointwise
quantities we assume only that $(X,d)$ carries an $A$-computable
metric space structure $(x_i)_{i\in\mathbb{N}}$
(Definition~\ref{def:computable-metric}); $A$-computability of $\varphi$ is not needed until the
effectivization arguments of
Section~\ref{subsubsec:lower-pressure-pts}.

\begin{definition}[Lower/upper $A$-pointwise pressure]
\label{def:pointwise-pressure}
For $x \in X$ and $\varepsilon > 0$,
\[
\underline{P}^A(x,T,\varphi,\varepsilon)
= \liminf\limits_{n\to\infty}\;
\min\limits_{x_i \in B_n(x,T,\varepsilon)}
\frac{K^A(i) + S_n\varphi(x_i)}{n},
\qquad
\underline{P}^A(x,T,\varphi)
= \lim\limits_{\varepsilon\to 0}
\underline{P}^A(x,T,\varphi,\varepsilon),
\]
\[
\overline{P}^A(x,T,\varphi,\varepsilon)
= \limsup\limits_{n\to\infty}\;
\min\limits_{x_i \in B_n(x,T,\varepsilon)}
\frac{K^A(i) + S_n\varphi(x_i)}{n},
\qquad
\overline{P}^A(x,T,\varphi)
= \lim\limits_{\varepsilon\to 0}
\overline{P}^A(x,T,\varphi,\varepsilon).
\]
\end{definition}

\subsubsection{Lower pressure and its discretization}
In this subsection we recall the Pesin--Pitskel pressure, which we
call the lower pressure. As throughout this paper, we use base $2$ in
place of the traditional base $e$.

For $x\in X$, $n\in\mathbb{N}$, and $\varepsilon>0$ we write
\[
\varphi(x,n,\varepsilon)=\sup_{y\in B_n(x,T,\varepsilon)}S_n\varphi(y)
\]
for the sup-weight of the Bowen ball $B_n(x,T,\varepsilon)$. The
analogous sup-weight $u(x,n,\varepsilon)$ appears for the BS
quantities in Section~3.3. We also write
$\Var(\varphi,\varepsilon)=\sup\{|\varphi(x)-\varphi(y)| : d(x,y)\leq\varepsilon\}$
for the modulus of $\varphi$, so that $y\in B_n(x,T,\varepsilon)$
implies $|S_n\varphi(y)-S_n\varphi(x)|\leq n\Var(\varphi,\varepsilon)$
and hence
$S_n\varphi(x)\leq\varphi(x,n,\varepsilon)\leq
S_n\varphi(x)+n\Var(\varphi,\varepsilon)$;
$\Var(\varphi,\varepsilon)\to 0$ as $\varepsilon\to 0$ by uniform
continuity.

\begin{definition}[Lower pressure]\label{def:lower-pressure}
Given $Z \subset X$, $\varepsilon > 0$, and $N \in \mathbb{N}$, let
$\mathcal{P}(Z, N, \varepsilon)$ be the collection of countable
families $\{(y_j, n_j)\} \subset X \times \{N, N+1, \ldots\}$ such
that $Z \subset \bigcup_j B_{n_j}(y_j, T, \varepsilon)$. For each
$s \in \mathbb{R}$, set
\begin{equation}\label{eq:lower-pressure-measure}
m_P(Z, s, \varphi, N, \varepsilon)
= \inf_{\mathcal{P}(Z,N,\varepsilon)}
\sum_{(y_j, n_j)} 2^{-n_j s + \varphi(y_j,n_j,\varepsilon)},
\qquad
m_P(Z, s, \varphi, \varepsilon)
= \lim\limits_{N \to \infty} m_P(Z, s, \varphi, N, \varepsilon).
\end{equation}
The function $s \mapsto m_P(Z, s, \varphi, \varepsilon)$ takes the
value $\infty$ for all $s$ below a critical value and $0$ for all $s$
above it; we denote the critical value by
\[
\underline{P}(Z,T,\varphi,\varepsilon)
= \inf\{ s \in \mathbb{R} \mid m_P(Z, s, \varphi, \varepsilon) = 0 \}
= \sup\{ s \in \mathbb{R} \mid m_P(Z, s, \varphi, \varepsilon)
= \infty \}.
\]
The \emph{lower pressure} of $\varphi$ on $Z$ is
\[\underline{P}(Z,T,\varphi)
= \lim\limits_{\varepsilon \to 0}
\underline{P}(Z,T,\varphi,\varepsilon);\]
the limit exists since the quantity is nondecreasing in $\varepsilon$. 
\end{definition}

This is the classical Pesin--Pitskel pressure (in base $2$). We now discretize this definition.
\begin{definition}[Null $s$-cover for topological pressure]\label{def:null_cover}
    Let $Z\subset X$. A \emph{null $s$-cover} of $Z$ is a subset $E$ of $\mathbb{N}^3$ so that
    \begin{enumerate}
        \item \label{ncov:mass}
        \[\sum\limits_{(i,n,p)\in E}2^{-sn+S_n\varphi(x_i)}<\infty\]
        \item \label{ncov:cover}for every $k,p \in \mathbb{N}$, the family
        \[
        \bigl\{\, B_{n}(x_{i},T,\,2^{-p}) \;:\; (i,n,p)\in E,\; n \ge k \bigr\}
        \]
        is a cover of $Z$.
    \end{enumerate}
\end{definition}
We now show this discretization characterizes the lower pressure.

\begin{lemma}\label{lem:null_cover}
$\underline{P}(Z,T,\varphi)=\inf\{s\mid Z \text{ has a null
$s$-cover}\}$.
\end{lemma}
\begin{proof}
($\leq$) Suppose $Z$ has a null $s$-cover $E$. Fix $p$, let
$t>s+\Var(\varphi,2^{-p})$ be rational, and fix $N$. By
condition~\ref{ncov:cover} the family
$\{(x_i,n)\mid (i,n,p)\in E,\ n\geq N\}$ belongs to
$\mathcal{P}(Z,N,2^{-p})$, and
$\varphi(x_i,n,2^{-p})\leq S_n\varphi(x_i)+n\Var(\varphi,2^{-p})$, so
\[
m_P(Z,t,\varphi,N,2^{-p})
\leq\sum_{\substack{(i,n,p)\in E\\ n\geq N}}2^{-tn+\varphi(x_i,n,2^{-p})}
\leq 2^{-(t-\Var(\varphi,2^{-p})-s)N}\sum_{(i,n,p)\in E}2^{-sn+S_n\varphi(x_i)}
\xrightarrow{\,N\to\infty\,}0
\]
by condition~\ref{ncov:mass}. Hence
$\underline{P}(Z,T,\varphi,2^{-p})\leq s+\Var(\varphi,2^{-p})$;
letting $p\to\infty$ gives $\underline{P}(Z,T,\varphi)\leq s$.

($\geq$) Suppose $s>\underline{P}(Z,T,\varphi)$ (if
$\underline{P}(Z,T,\varphi)=\infty$ there is nothing to prove).
Since $\underline{P}(Z,T,\varphi,\varepsilon)\uparrow
\underline{P}(Z,T,\varphi)$ as $\varepsilon\to 0$ we have $\underline{P}(Z,T,\varphi,2^{-(p+1)})<s$ for all $p$. Therefore for all $p$,
$m_P(Z,s,\varphi,2^{-(p+1)})=0$; since
$m_P(Z,s,\varphi,N,\cdot)$ is nondecreasing in $N$, its vanishing
limit forces vanishing at every $N$, so for any $p$ there
is a sequence $N^p_k\uparrow\infty$ and countable families
$\Gamma_{p,k}\in\mathcal{P}(Z,N^p_k,2^{-(p+1)})$ with
\[
\sum_{(y,n)\in\Gamma_{p,k}} 2^{-sn+\varphi(y,n,2^{-(p+1)})}
\leq 2^{-(p+k+1)}.
\]
Since $(x_i)$ is dense in each Bowen metric $d_n$, for each
$(y,n)\in\Gamma_{p,k}$ we may choose an ideal point $x_{i(y,n)}$
with $d_n(y,x_{i(y,n)})<2^{-(p+1)}$; then
$B_n(y,T,2^{-(p+1)})\subset B_n(x_{i(y,n)},T,2^{-p})$ by the
triangle inequality, and since
$x_{i(y,n)}\in B_n(y,T,2^{-(p+1)})$ we have
$S_n\varphi(x_{i(y,n)})\leq\varphi(y,n,2^{-(p+1)})$, so
\[
\sum_{(y,n)\in\Gamma_{p,k}} 2^{-sn+S_{n}\varphi(x_{i(y,n)})}
\leq 2^{-(p+k+1)}.
\]
Set
\[E=\bigl\{(i(y,n),n,p)\mid p,k\in\mathbb{N},\
(y,n)\in\Gamma_{p,k}\bigr\}.\]
For condition~\ref{ncov:mass}: each element of $E$ arises from at
least one triple $(p,k,(y,n))$, so
\[
\sum_{(i,n,p)\in E}2^{-sn+S_n\varphi(x_i)}
\leq\sum_{p\in \mathbb{N}}\sum_{k=1}^\infty
\sum_{(y,n)\in\Gamma_{p,k}} 2^{-sn+S_{n}\varphi(x_{i(y,n)})}
\leq\sum_{p\in \mathbb{N}}2^{-p}\leq 1<\infty.
\]
For condition~\ref{ncov:cover}: fix
$k_0\in\mathbb{N}$ and choose $k$ with $N^p_k\geq k_0$. The family $\Gamma_{p,k}$ covers $Z$ by
balls $B_n(y,T,2^{-(p+1)})$ with orders $n\geq N^p_k\geq k_0$, and
each such ball is contained in $B_n(x_{i(y,n)},T,2^{-p})$ with
$(i(y,n),n,p)\in E$; hence
$\{B_n(x_i,T,2^{-p})\mid (i,n,p)\in E,\ n\geq k_0\}$ covers $Z$.
Therefore $E$ is a null $s$-cover, so
$\inf\{s\mid Z\text{ has a null $s$-cover}\}
\leq\underline{P}(Z,T,\varphi)$; as
$s>\underline{P}(Z,T,\varphi)$ was arbitrary, the two quantities are
equal.
\end{proof}
\subsubsection{The point-to-set principle for lower pressure}\label{subsubsec:lower-pressure-pts}

Having discretized the lower pressure, we now effectivize it. Throughout the following sections we assume the oracle $A$ is admissible with respect to $(X,d,T,\varphi)$. 
\begin{definition}[$A$-effective null $s$-cover; effective lower
pressure]\label{def:eff-lower-pressure}
An \emph{$A$-effective null $s$-cover} is a null $s$-cover
$E\subset\mathbb{N}^3$ that is $A$-computably enumerable. The
\emph{$A$-effective lower pressure} of $\varphi$ on $Z$ is
\[\underline{P}^A_{\mathrm{eff}}(Z,T,\varphi)
=\inf\{s\mid Z \text{ has an $A$-effective null $s$-cover}\}.\]
\end{definition}
The next theorem connects the effective pressure of a set to the pointwise pressures of its points.
\begin{theorem}\label{thm:sup_eq_pressure_lower}
    \[\underline{P}^A_{\mathrm{eff}}(Z,T,\varphi)=\sup\limits_{x\in Z}\underline{P}^{A}(x,T,\varphi)\]
\end{theorem}
This follows from the next two lemmas.
\begin{lemma}\label{lem:eff-pressure-upper}
For $A$ admissible for $(X,d,T,\varphi)$,
    \[\underline{P}^A_{\mathrm{eff}}(Z,T,\varphi)\leq
    \sup_{x\in Z}\underline{P}^A(x,T,\varphi).\]
\end{lemma}
\begin{proof}
If the supremum is infinite there is nothing to prove. Fix a rational $s>\sup_{x\in Z}\underline{P}^A(x,T,\varphi)$ and
define $Y_{<s}=\{x\mid \underline{P}^A(x,T,\varphi)<s\}\supset Z$. We
construct an $A$-effective null $s$-cover for $Y_{<s}$. Consider
\[E=\bigl\{(i,n,p)\;\bigm|\; K^A(i,n,p)< sn-S_n\varphi(x_i)\bigr\},\]
which is $A$-c.e.\ by Remark~\ref{rem:ce-thresholds}, since
$S_n\varphi(x_i)$ is $A$-computable uniformly in $(i,n)$
(Lemma~\ref{lem:birkhoff-uniform}).\\
We verify that this is a null $s$-cover for $Y_{<s}$.
For condition \ref{ncov:mass}: membership in $E$ gives
$2^{-sn+S_n\varphi(x_i)} \leq 2^{-K^A(i,n,p)}$, so by Kraft's inequality over the
distinct triples,
\[
\sum_{(i,n,p)\in E} 2^{-sn+S_n\varphi(x_i)}
\;\leq\; \sum_{(i,n,p)\in E} 2^{-K^A(i,n,p)} \;\leq\; 1.
\]
For condition \ref{ncov:cover}, fix $p \in \mathbb{N}$ and $x \in Y_{<s}$; we show
$x \in B_n(x_i, T, 2^{-p})$ with $(i,n,p) \in E$ for infinitely many $n$.\\
Given $x\in Y_{<s}$ we have $\underline{P}^A(x,T,\varphi)<s$ and therefore $\underline{P}^A(x,T,\varphi,2^{-p})<s$ for each $p$. Take a sequence $n_k\uparrow\infty$ and realizers $x_{i_k}$ with $x\in B_{n_k}(x_{i_k},T,2^{-p})$ and $\frac{K^A(i_k)+S_{n_k}\varphi(x_{i_k})}{n_k}<s$. Since $K^A(i,n,p) \leq K^A(i) + 2\log n + O(1)$ for fixed $p$, it follows that
$K^A(i_k, n_k, p) < s n_k - S_{n_k}\varphi(x_{i_k})$ for all sufficiently large $k$.
\end{proof}

\begin{lemma}\label{lem:pointwise-leq-eff-pressure}
Let $Z \subset X$. For all $x \in Z$, $\underline{P}^A(x,T,\varphi) \leq \underline{P}^A_{\mathrm{eff}}(Z,T,\varphi)$.
\end{lemma}

\begin{proof}
Let $s > \underline{P}^A_{\mathrm{eff}}(Z,T,\varphi)$. Then $Z$ has an $A$-effective null $s$-cover $E$. As 
\[
\sum_{(i,n,p)\in E} 2^{-sn+S_n \varphi(x_i)} < \infty,
\]
by the coding theorem we have for $(i,n,p)\in E$
\[K^A(i)\leq K^A(i,n,p)+O(1)\leq sn-S_n \varphi( x_i)+O(1)\]
 Given $p$ we have a sequence $(i_k,n_k,2^{-p})\in E$ with $n_k\to\infty$ so that $x\in B_{n_k}(x_{i_k},T,2^{-p})$. Therefore $x_{i_k}\in B_{n_k}(x,T,2^{-p})$ and we get 
\[\min\limits_{x_i\in B_{n_k}(x,T,2^{-p})}\frac{K^A(i)+S_{n_k}\varphi(x_i)}{n_k}\leq \frac{K^A(i_k)+S_{n_k}\varphi(x_{i_k})}{n_k}\leq \frac{sn_k+O(1)}{n_k}\leq s+\frac{O(1)}{n_k}\]
Therefore on the subsequence $n_k$ we have an upper bound which decays to $s$, which gives an upper bound 
\[\underline{P}^A(x,T,\varphi,2^{-p})\leq s\]
Since this bound holds for every $p$ we get $\underline{P}^A(x,T,\varphi)\leq s$, and taking $s\downarrow \underline{P}^A_{\mathrm{eff}}(Z,T,\varphi)$ gives $\underline{P}^A(x,T,\varphi)\leq \underline{P}^A_{\mathrm{eff}}(Z,T,\varphi)$.

\end{proof}
It remains to produce an oracle achieving the value $\underline{P}(Z,T,\varphi)$. As before, we encode a sequence of near-optimal null $s$-covers into a single oracle by a countable join.
\begin{lemma}\label{lem:exists-oracle-pressure}
    There is an oracle $A$ so that 
    \[\underline{P}(Z,T,\varphi)=P^A_{\mathrm{eff}}(Z,T,\varphi)\]
\end{lemma}
\begin{proof}
As in Lemma~\ref{exists-oracle}, with
Lemma~\ref{lem:null_cover} in place of
Lemma~\ref{lem:h2-null-s-cover}: choose null $s_j$-covers
$E_j \subset \mathbb{N}^3$ for $Z$ with
$s_j \downarrow \underline{P}(Z,T,\varphi)$ and set
$A = \bigoplus_j E_j$. The reverse inequality holds for every oracle.
\end{proof}
Every oracle $B\geq_T A$ satisfies \[\underline{P}^B_{\mathrm{eff}}(Z,T,\varphi)\leq \underline{P}^A_{\mathrm{eff}}(Z,T,\varphi),\]
and every oracle $A$ satisfies
\[\underline{P}(Z,T,\varphi)\leq \underline{P}^A_{\mathrm{eff}}(Z,T,\varphi),\]
so the identity $\underline{P}(Z,T,\varphi)=\underline{P}^A_{\mathrm{eff}}(Z,T,\varphi)$ holds on an upper set of oracles. We thus obtain the point-to-set principle for Pesin--Pitskel topological pressure
\begin{theorem}[Point-to-set principle for lower topological pressure]
\label{thm:pts_lower_pressure}
Let $Z\subset X$. Then
\[\underline{P}(Z,T,\varphi)=\min \limits_{A\subset \mathbb{N}}\sup\limits_{x\in Z}\underline{P}^{A}(x,T,\varphi),\]
with the minimum attained on an upper set of Turing degrees.
\end{theorem}
\begin{proof}
This is the same as the proof of Theorem \ref{thm:pts-entropy}. For every oracle $A$ we have
$\underline{P}(Z,T,\varphi)\le P^A_{\mathrm{eff}}(Z,T,\varphi)$, and for admissible $A$,
$P^A_{\mathrm{eff}}(Z,T,\varphi)=\sup_{x\in Z}\underline{P}^{A}(x,T,\varphi)$ by
Theorem~\ref{thm:sup_eq_pressure_lower}. By Lemma~\ref{lem:exists-oracle-pressure} there
is an oracle $B$ with $\underline{P}(Z,T,\varphi)=P^B_{\mathrm{eff}}(Z,T,\varphi)$. Joining
with an admissible oracle (Remark~\ref{rem:admissible-cone}) and
using monotonicity (Lemma~\ref{lem:monotone}), we may take $B$
admissible, and the same holds for every oracle above $B$. The
minimum is therefore attained at $B$, on the upper cone above it.
\end{proof}

\subsubsection{Upper pressure and its discretization}
\label{subsubsec:upper-pressure-disc}
We now treat the upper pressure, in analogy with the point-to-set
principle for packing dimension. Packing pressure admits an
equivalent characterization as a \emph{modified upper capacity
pressure}, in exact analogy with the characterization of packing
dimension as modified upper box dimension. Since this is the
formulation we discretize, we adopt it as our working definition.

\begin{definition}[Modified upper capacity pressure]\label{def:mUC1}
Let $Z\subset X$ and let $n\in\mathbb{N}$, $\varepsilon>0$. Set
\[
N(Z,T,\varphi,n,\varepsilon)
=\inf\Bigg\{\sum_{x\in E} 2^{\varphi(x,n,\varepsilon)}
\;\Bigg|\; E\subset X \text{ countable},\
Z\subset \bigcup_{x\in E}B_n(x,T,\varepsilon)\Bigg\}.
\]
The \emph{upper capacity pressure} of $Z$ at scale $\varepsilon$ is
\[
P_{\mathrm{UC}}(Z,T,\varphi,\varepsilon)
=\limsup\limits_{n\to \infty}\frac{\log N(Z,T,\varphi,n,\varepsilon)}{n},
\]
the \emph{modified upper capacity pressure} at scale $\varepsilon$ is
\[
P_{\mathrm{mUC}}(Z,T,\varphi,\varepsilon)
=\inf\Bigg\{\,\sup_{i\in\mathbb{N}}
P_{\mathrm{UC}}(Z_i,T,\varphi,\varepsilon)
\;\Bigg|\; Z\subset\bigcup_{i\in\mathbb{N}} Z_i\Bigg\},
\]
where the infimum is over countable covers of $Z$, and the
\emph{modified upper capacity pressure} of $Z$ is
\[
P_{\mathrm{mUC}}(Z,T,\varphi)
=\lim\limits_{\varepsilon\to 0}P_{\mathrm{mUC}}(Z,T,\varphi,\varepsilon).
\]
(The existence of the limit is part of
Corollary~\ref{cor:sep-cov-pressure} in the appendix, where
$P_{\mathrm{mUC}}$ is compared with the separated-set quantity
$P^{\mathrm{sep}}_{\mathrm{mUC}}$ and the classical packing
pressure.)
\end{definition}

The relationship between modified upper capacity pressure and packing
pressure is the same as that between modified upper box dimension and
packing dimension:
\[
\overline{P}(Z,T,\varphi)=P_{\mathrm{mUC}}(Z,T,\varphi),
\]
as proven in the compact case in \cite{zhao_2017_a}. We extend this to
our setting in Appendix~\ref{app:packing-pressure}. Henceforth we take
this as our definition of packing pressure and write $\overline{P}$ in
place of $P_{\mathrm{mUC}}$.
Our first step is to discretize Definition~\ref{def:mUC1}.

\begin{definition}[Uniform $s$-net]\label{def:s-net}
Fix a dense sequence $(x_i)_{i\in\mathbb{N}}$ in $X$. A \emph{uniform $s$-net
for $Z$} is a set $E\subset\mathbb{N}^4$ such that, writing
$E_{n,p,k}=\{i\mid (i,n,p,k)\in E\}$:
\begin{enumerate}
    \item \label{snet:pressure} for all $p,k\in \mathbb{N}$,
    \[
    \limsup\limits_{n\to\infty}\frac{1}{n}
    \log\Bigg(\sum_{i\in E_{n,p,k}} 2^{S_n\varphi(x_i)}\Bigg)\leq s;
    \]
    \item \label{snet:cover} for each $p\in\mathbb{N}$,
    \[
    Z\subset\bigcup_{k\in\mathbb{N}}\bigcap_{n\in\mathbb{N}}
    \bigcup_{i\in E_{n,p,k}} B_n(x_i,T,2^{-p}).
    \]
\end{enumerate}
\end{definition}
In this formulation, the countable decomposition of $Z$ at scale
$\varepsilon=2^{-p}$ is given by the sets
\[
Z^{(p)}_k=\bigcap_{n\in\mathbb{N}}\bigcup_{i\in E_{n,p,k}} B_n(x_i,T,2^{-p}),
\qquad k\in\mathbb{N},
\]
with condition~\ref{snet:pressure} witnessing
$P_{\mathrm{UC}}(Z^{(p)}_k,T,\varphi,2^{-p})\leq s$ for each $k$ and condition \ref{snet:cover} witnessing $Z\subset \bigcup_{k\in\mathbb{N}}Z^{(p)}_k$\\
We now establish that this definition captures the notion of upper pressure.

\begin{lemma}\label{lem:s-net}
\[
\overline{P}(Z,T,\varphi)=\inf\{s\mid Z \text{ has a uniform
$s$-net}\}.
\]
\end{lemma}
\begin{proof}
($\geq$) Suppose $s>\overline{P}(Z,T,\varphi)$. Since
$P_{\mathrm{mUC}}(Z,T,\varphi,\varepsilon)\uparrow
\overline{P}(Z,T,\varphi)$ as $\varepsilon\to 0$, we have $P_{\mathrm{mUC}}(Z,T,\varphi,2^{-(p+1)})<s$ for all
$p$. For each $p$
there is a countable cover $Z\subset\bigcup_{k}Z^{(p)}_k$ with
$P_{\mathrm{UC}}(Z^{(p)}_k,T,\varphi,2^{-(p+1)})<s$ for all $k$. Thus
for each $n,k$ we may choose a countable set $F_{n,p,k}\subset X$
with
\[
Z^{(p)}_k\subset\bigcup_{y\in F_{n,p,k}}B_n\bigl(y,T,2^{-(p+1)}\bigr)
\qquad\text{and}\qquad
\limsup\limits_{n\to\infty}\frac{1}{n}
\log\sum_{y\in F_{n,p,k}} 2^{\varphi(y,n,2^{-(p+1)})}<s,
\]
taking $F_{n,p,k}$ to realize
$N(Z^{(p)}_k,T,\varphi,n,2^{-(p+1)})$ within a factor of $2$; the
factor contributes $\frac{1}{n}\log 2\to 0$ to the exponential rate.
Since the Bowen metric $d_n$ induces the topology of $X$, $(x_i)$ is
$d_n$-dense and we may choose for each $y\in F_{n,p,k}$ an index
$i(y)$ with $d_n\bigl(x_{i(y)},y\bigr)<2^{-(p+1)}$; then
$x_{i(y)}\in B_n(y,T,2^{-(p+1)})$, so
$S_n\varphi\bigl(x_{i(y)}\bigr)\leq \varphi(y,n,2^{-(p+1)})$.
Let $E=\{(i(y),n,p,k)\mid n,p,k\in\mathbb{N},\ y\in F_{n,p,k}\}$, so
that $E_{n,p,k}=\{i(y)\mid y\in F_{n,p,k}\}$. We claim $E$ is a
uniform $s$-net for $Z$.

By the triangle inequality,
$B_n(y,T,2^{-(p+1)})\subset B_n(x_{i(y)},T,2^{-p})$, so for every
$n$,
\[
Z^{(p)}_k
\subset\bigcup_{i\in E_{n,p,k}}B_n\bigl(x_i,T,2^{-p}\bigr).
\]
Intersecting over $n$ and taking the union over $k$ yields
condition~\ref{snet:cover}. For condition~\ref{snet:pressure}: each
$i\in E_{n,p,k}$ equals $i(y)$ for at least one $y\in F_{n,p,k}$, so
\[
\sum_{i\in E_{n,p,k}} 2^{S_n\varphi(x_i)}
\leq \sum_{y\in F_{n,p,k}} 2^{S_n\varphi(x_{i(y)})}
\leq \sum_{y\in F_{n,p,k}} 2^{\varphi(y,n,2^{-(p+1)})},
\]
and hence
\[
\limsup\limits_{n\to\infty}\frac{1}{n}
\log\sum_{i\in E_{n,p,k}} 2^{S_n\varphi(x_i)}
<s.
\]
So $E$ is a
uniform $s$-net, and
$\inf\{s\mid Z\text{ has a uniform
$s$-net}\}\leq\overline{P}(Z,T,\varphi)$.

($\leq$) Suppose $Z$ has a uniform $s$-net $E$, and fix
$p\in\mathbb{N}$. By condition~\ref{snet:cover}, the sets
$Z^{(p)}_k=\bigcap_n\bigcup_{i\in E_{n,p,k}}B_n(x_i,T,2^{-p})$
cover $Z$. For each $k$ and every $n$ we have
$Z^{(p)}_k\subset\bigcup_{i\in E_{n,p,k}}B_n(x_i,T,2^{-p})$ by
construction, so
\[
N(Z^{(p)}_k,T,\varphi,n,2^{-p})
\leq\sum_{i\in E_{n,p,k}} 2^{\varphi(x_i,n,2^{-p})}
\leq 2^{\,n\Var(\varphi,2^{-p})}\sum_{i\in E_{n,p,k}} 2^{S_n\varphi(x_i)},
\]
and condition~\ref{snet:pressure} gives
$P_{\mathrm{UC}}(Z^{(p)}_k,T,\varphi,2^{-p})\leq s+\Var(\varphi,2^{-p})$
for every $k$. Hence
$P_{\mathrm{mUC}}(Z,T,\varphi,2^{-p})\leq s+\Var(\varphi,2^{-p})$;
letting $p\to\infty$ gives $\overline{P}(Z,T,\varphi)\leq s$.
Taking the infimum over such $s$ gives
$\overline{P}(Z,T,\varphi)
\leq\inf\{s\mid Z\text{ has a uniform $s$-net}\}$.
\end{proof}
\subsubsection{The point-to-set principle for upper topological pressure}
We now effectivize packing pressure to an $A$-computable version
$\overline{P}^A_{\mathrm{eff}}$, with the goal of proving
\[
\overline{P}^A_{\mathrm{eff}}(Z,T,\varphi)
=\sup_{x\in Z}\overline{P}^A(x,T,\varphi).
\]
\begin{definition}[$A$-effective uniform $s$-net]
    An \emph{$A$-effective uniform $s$-net} is a uniform $s$-net $E\subset \mathbb{N}^4$ which is $A$-computably enumerable.
\end{definition}
We similarly define the $A$-effective upper pressure.
\begin{definition}[$A$-effective upper pressure]
    \[\overline{P}^A_{\mathrm{eff}}(Z,T,\varphi)=\inf\{s\mid Z \text{ has an $A$-effective uniform $s$-net}\}\]
\end{definition}
For any oracle $A$ we have $\overline{P}(Z,T,\varphi)\leq \overline{P}^A_{\mathrm{eff}}(Z,T,\varphi)$, the effective quantity being an infimum over fewer nets. In parallel with the Bowen entropy case, we prove the following theorem.
\begin{theorem}\label{thm:eff_upper_pointwise}
For $A$ admissible for $(X,d,T,\varphi)$ and any $Z\subset X$,
\[
\overline{P}^A_{\mathrm{eff}}(Z,T,\varphi)
=\sup_{x\in Z}\overline{P}^{A}(x,T,\varphi).
\]
\end{theorem}
We prove this in two lemmas.
\begin{lemma}\label{lem:eff_upper}
For $A$ admissible for $(X,d,T,\varphi)$,
    \[\overline{P}^A_{\mathrm{eff}}(Z,T,\varphi)\leq
    \sup_{x\in Z}\overline{P}^A(x,T,\varphi).\]
\end{lemma}
\begin{proof}
If the supremum is infinite there is nothing to prove. Fix a rational $s>\sup_{x\in Z}\overline{P}^A(x,T,\varphi)$ and
define $Y_{<s}=\{x\mid \overline{P}^A(x,T,\varphi)<s\}\supset Z$. We
construct an $A$-effective uniform $s$-net for $Y_{<s}$. Consider
\[E=\bigl\{(i,n,p,k)\;\bigm|\; K^A(i)< sn-S_n\varphi(x_i)+k\bigr\},\]
which is $A$-c.e.\ by Remark~\ref{rem:ce-thresholds}, since
$S_n\varphi(x_i)$ is $A$-computable uniformly in $(i,n)$
(Lemma~\ref{lem:birkhoff-uniform}). The condition is independent
of $p$, and the slices are increasing in $k$:
$E_{n,p,k}\subset E_{n,p,k+1}$.

We first verify condition 1. Fix $p,k\in\mathbb{N}$. By membership in
$E$ and Kraft's inequality over the distinct indices $i$,
\[\sum_{i\in E_{n,p,k}} 2^{S_n\varphi(x_i)}
\leq \sum_{i\in E_{n,p,k}} 2^{sn-K^A(i)+k}
=2^{sn+k}\sum_{i\in E_{n,p,k}} 2^{-K^A(i)}
\leq 2^{sn+k},\]
so
\[\limsup\limits_{n\to\infty}\frac{1}{n}
\log\Bigg(\sum_{i\in E_{n,p,k}} 2^{S_n\varphi(x_i)}\Bigg)\leq s.\]

We now verify condition 2. Fix $p\in\mathbb{N}$ and
$x\in Y_{<s}$. Since $\overline{P}^A(x,T,\varphi,2^{-p})<s$, there is
$N$ such that for every $n\geq N$ there is a realizer
$x_{i_n}\in B_n(x,T,2^{-p})$ with $K^A(i_n)< sn-S_n\varphi(x_{i_n})$,
so $(i_n,n,p,k)\in E$ for every $k\geq 0$. For each of the finitely
many $n<N$, choose by density any ideal point
$x_{i_n}\in B_n(x,T,2^{-p})$; then
$K^A(i_n)< sn-S_n\varphi(x_{i_n})+k_n$ for some finite $k_n$. Taking
$k=\max_{n<N}k_n$, monotonicity of the slices in $k$ gives
$x\in\bigcup_{i\in E_{n,p,k}}B_n(x_i,T,2^{-p})$ for every $n$,
hence $x\in Z^{(p)}_k$. Therefore
$Y_{<s}\subset\bigcup_k Z^{(p)}_k$ for every $p$.

Hence $E$ is an $A$-effective uniform $s$-net for
$Y_{<s}\supseteq Z$, so by monotonicity of
$\overline{P}^A_{\mathrm{eff}}$ under inclusion,
$\overline{P}^A_{\mathrm{eff}}(Z,T,\varphi)\leq s$. As $s$ was an
arbitrary rational above the supremum, the lemma follows.
\end{proof}
We now prove the lower bound.
\begin{lemma}\label{lem:eff_upper_pressure_geq}
For $A$ admissible for $(X,d,T,\varphi)$ and any $Z\subset X$,
\[\overline{P}^A_{\mathrm{eff}}(Z,T,\varphi)\geq
\sup_{x\in Z}\overline{P}^A(x,T,\varphi).\]
\end{lemma}
\begin{proof}
Let $t>s>\overline{P}^A_{\mathrm{eff}}(Z,T,\varphi)$ be rational and
fix an $A$-effective uniform $s$-net $E$ for $Z$. Fix $x\in Z$, and $p\in\mathbb{N}$. By
condition~\ref{snet:cover} there is $k\in\mathbb{N}$ with
$x\in\bigcap_{n}\bigcup_{i\in E_{n,p,k}}B_n(x_i,T,2^{-p})$; fix such
a $k$.

Consider the slices $E_{n,p,k}$ as $n$ varies. By
condition~\ref{snet:pressure} (since $2^{\frac{s+t}{2}n}$ is a strictly larger growth rate) there is $N_0$ such that
\[
\sum_{i\in E_{n,p,k}} 2^{S_n\varphi(x_i)}\leq 2^{\frac{s+t}{2}n}
\qquad\text{for all } n\geq N_0,
\]
and hence
\[
\sum_{n\geq N_0}\sum_{i\in E_{n,p,k}} 2^{S_n\varphi(x_i)-tn}
\leq \sum_{n\geq N_0} 2^{-\frac{t-s}{2}n}
<\infty,
\]
the inner sum being dominated by a convergent geometric series.
Consequently the function
\[
(i,n)\mapsto tn-S_n\varphi(x_i),
\qquad\text{defined for } n\geq N_0,\ i\in E_{n,p,k},
\]
satisfies the conditions of the extended coding theorem: it is A-upper semicomputable defined on an A-c.e. domain and it has bounded sum. Therefore by the extended coding theorem
(Lemma~\ref{lem:extended-coding-theorem}),
\[
K^A(i,n)\leq tn-S_n\varphi(x_i)+O(1)
\qquad\text{for all } n\geq N_0,\ i\in E_{n,p,k},
\]
where the constant depends only on $E$, $p$, $k$, and $t$ (and is independent of $n$ and $i$). Since
$K^A(i)\leq K^A(i,n)+O(1)$, we conclude
\[
K^A(i)+S_n\varphi(x_i)\leq tn+O(1)
\qquad\text{for all } n\geq N_0,\ i\in E_{n,p,k}.
\]
Now, by the choice of $k$, for each $n$ there is a witness
$i_n\in E_{n,p,k}$ with $x_{i_n}\in B_n(x,T,2^{-p})$, and the
displayed bound applies to it once $n\geq N_0$. Therefore
\[
\overline{P}^A(x,T,\varphi,2^{-p})
=\limsup\limits_{n\to\infty}\;\min\limits_{x_i\in B_n(x,T,2^{-p})}
\frac{K^A(i)+S_n\varphi(x_i)}{n}
\leq\limsup\limits_{n\to\infty}\frac{K^A(i_n)+S_n\varphi(x_{i_n})}{n}
\leq t.
\]
Letting $p\to\infty$ gives $\overline{P}^A(x,T,\varphi)\leq t$;
letting $t\downarrow s$ and then
$s\downarrow\overline{P}^A_{\mathrm{eff}}(Z,T,\varphi)$, and taking
the supremum over $x\in Z$ proves the lemma.
\end{proof}

\begin{lemma}\label{lem:oracle_exists_upper_pressure}
    Given $Z\subset X$ there is an oracle $A$ such that
    \[\overline{P}(Z,T,\varphi)=\overline{P}^A_{\mathrm{eff}}(Z,T,\varphi).\]
\end{lemma}
\begin{proof}
As in Lemmas~\ref{exists-oracle}
and~\ref{lem:exists-oracle-pressure}: by Lemma~\ref{lem:s-net},
choose uniform $s_j$-nets $E_j\subset\mathbb{N}^4$ for $Z$ with
$s_j\downarrow\overline{P}(Z,T,\varphi)$, and let
$A=\bigoplus_j E_j$. Each $E_j$ is $A$-c.e., hence an $A$-effective
uniform $s_j$-net, so
$\overline{P}^A_{\mathrm{eff}}(Z,T,\varphi)\leq s_j$ for every $j$;
the reverse inequality holds for every oracle.
\end{proof}
\begin{makethm}{theorem}{thm:pts_upper_pressure}
    Given $Z\subset X$ we have
    \[\overline{P}(Z,T,\varphi)
    =\min\limits_{A\subset \mathbb{N}}\,\sup_{x\in Z}\overline{P}^A(x,T,\varphi).\]
\end{makethm}
\begin{proof}
Identical to Theorems~\ref{thm:pts-entropy}
and~\ref{thm:pts_lower_pressure}: for every oracle $A$ we have
$\overline{P}(Z,T,\varphi)\leq\overline{P}^A_{\mathrm{eff}}(Z,T,\varphi)$
(Lemma~\ref{lem:s-net}), and for admissible $A$,
$\overline{P}^A_{\mathrm{eff}}(Z,T,\varphi)
=\sup_{x\in Z}\overline{P}^A(x,T,\varphi)$ by
Theorem~\ref{thm:eff_upper_pointwise}. By
Lemma~\ref{lem:oracle_exists_upper_pressure} there is an oracle $B$
with
$\overline{P}(Z,T,\varphi)=\overline{P}^B_{\mathrm{eff}}(Z,T,\varphi)$,
and every oracle $\geq_T B$ satisfies the same identity. Joining
with an admissible oracle (Remark~\ref{rem:admissible-cone}),
we may take $B$ admissible, and the same holds for every oracle
above $B$. The minimum is therefore attained at $B$, on the upper
cone above it.
\end{proof}

\subsection{Point-to-set principles for lower/upper BS dimension}
The BS (Barreira--Schmeling \cite{Barreira2000Sets}) dimension, also called the $u$-dimension, is the dimension quantity that solves the Bowen pressure equation, and is for that reason central to Section~\ref{sec:bowen}.\\
In this section we prove point-to-set principles for a generalization of the BS dimension: rather than requiring $u$ to be positive on a compact metric space, we allow $u$ to be any real-valued uniformly continuous function on a complete separable metric space and impose positivity only asymptotically on average, at the points of the set under study. This generality is what later permits the treatment of nonuniformly expanding systems, at the cost of some extra care in the arguments. 

\subsubsection{Pointwise BS dimension}\label{subsec:bs-lower}
Throughout Section~\ref{subsec:bs-lower}, $T:X\to X$ is a continuous
map on a complete separable metric space $(X,d)$ and
$u:X\to\mathbb{R}$ is uniformly continuous. We write
\[\underline{u}(x)=\liminf\limits_{n\to\infty}\frac{S_nu(x)}{n},
\qquad
\overline{u}(x)=\limsup\limits_{n\to\infty}\frac{S_nu(x)}{n}.\]
As in Section~\ref{subsec:lower-pressure}, for the pointwise
quantities we assume only that $(X,d)$ carries an $A$-computable
metric space structure $(x_i)_{i\in\mathbb{N}}$
(Definition~\ref{def:computable-metric}). Full admissibility of $A$ with respect to $(X,d,T,u)$ is not needed until
the effectivization arguments in sections \ref{subsubsec:lower-bs-pts} and \ref{subsec_pts_BS_upper}.

\begin{definition}[Lower/upper $A$-pointwise BS dimension]
\label{def:pointwise-bs}
For $x \in X$ with $\underline{u}(x)>0$ and $\varepsilon>0$,
\[\dim^A_{BS}(x,T,u,\varepsilon)
=\liminf\limits_{n\to\infty}\;\min\limits_{x_i\in B_n(x,T,\varepsilon)}
\frac{K^A(i)}{S_nu(x_i)},
\qquad
\dim^A_{BS}(x,T,u)
=\lim\limits_{\varepsilon\to 0}\dim^A_{BS}(x,T,u,\varepsilon),\]
\[\Dim^A_{BS}(x,T,u,\varepsilon)
=\limsup\limits_{n\to\infty}\;\min\limits_{x_i\in B_n(x,T,\varepsilon)}
\frac{K^A(i)}{S_nu(x_i)},
\qquad
\Dim^A_{BS}(x,T,u)
=\lim\limits_{\varepsilon\to 0}\Dim^A_{BS}(x,T,u,\varepsilon).\]
\end{definition}
\subsubsection{Dimensional definitions of lower BS dimension}
Let $T$ be a continuous transformation on a complete separable metric space $X$ and 
$u : X \to \mathbb{R}$ a uniformly continuous function. 
\begin{definition}[Lower BS dimension]\label{def:lower-bs}
Given $Z \subset X$, $\varepsilon > 0$, and $N \in \mathbb{N}$, let
$\mathcal{P}(Z, N, \varepsilon)$ be the collection of countable
families $\{(y_j, n_j)\} \subset X \times \{N, N+1, \ldots\}$ such
that $Z \subset \bigcup_j B_{n_j}(y_j, T, \varepsilon)$. For each
$s \in [0,\infty)$, set 
\[u(x,n,\varepsilon)=\sup\limits_{y\in B_n(x,T,\varepsilon)}S_nu(y)\]
and let
\begin{equation}\label{eq:lower-bs-measure}
m_{BS}(Z, s, u, N, \varepsilon)
= \inf_{\mathcal{P}(Z,N,\varepsilon)}
\sum_{(y_j, n_j)} 2^{-s\, u(y_j,n_j,\varepsilon)},
\qquad
m_{BS}(Z, s, u, \varepsilon)
= \lim\limits_{N \to \infty} m_{BS}(Z, s, u, N, \varepsilon).
\end{equation}
Suppose $\underline{u}(x) > 0$ for all $x \in Z$. Then for
$\varepsilon$ sufficiently small, the function
$s \mapsto m_{BS}(Z, s, u, \varepsilon)$ takes the value $\infty$ for
all $s$ below a critical value and $0$ for all $s$ above it
(Lemma~\ref{lem:bs-dichotomy} below); we denote the critical value by
\[
\dim_{BS}(Z,T,u,\varepsilon)
= \inf\{ s \geq 0 \mid m_{BS}(Z, s, u, \varepsilon) = 0 \}
= \sup\{ s \geq 0 \mid m_{BS}(Z, s, u, \varepsilon)
= \infty \}.
\]
The \emph{lower BS dimension} of $Z$ with respect to $u$ is
\[\dim_{BS}(Z,T,u)
= \lim\limits_{\varepsilon \to 0}
\dim_{BS}(Z,T,u,\varepsilon);\]
the limit exists because for $\varepsilon_1 < \varepsilon_2$ we have
$\mathcal{P}(Z, N, \varepsilon_1) \subset
\mathcal{P}(Z, N, \varepsilon_2)$ and
$u(y,n,\varepsilon_1) \leq u(y,n,\varepsilon_2)$, so (as $s\geq 0$)
$m_{BS}(Z, s, u, \varepsilon_1) \geq
m_{BS}(Z, s, u, \varepsilon_2)$ and
$\dim_{BS}(Z,T,u,\varepsilon_1) \geq
\dim_{BS}(Z,T,u,\varepsilon_2)$.
\end{definition}
\begin{remark}\label{rem:center-weight-bs}
As with the pressure quantities, the set-level definition is
sup-weighted, matching the classical convention, while the
discretizations below weight each ball by the value of $S_nu$ at an
ideal-point center. The comparison costs $n\Var(u,\varepsilon)$ per
ball. Under the standing hypothesis $\underline{u}>0$ the sup-weights
grow linearly along the relevant pieces of $Z$, so the loss is
absorbed by an arbitrarily small increase of the exponent
(Lemma~\ref{lem:bs_null_cover} and Lemma~\ref{lem:bs-s-net}, via the
exhaustion of Remark~\ref{rem:mBSC-dichotomy}).
\end{remark}
We now discretize.
\begin{definition}[Null $s$-cover for BS dimension]
    Let $Z \subset X$.  
A \emph{null $s$-cover} of $Z$ is a set $E \subset \mathbb{N}^{3}$ satisfying  
\begin{enumerate}
  \item \label{bscov:mass}
  \[
        \sum_{(i,n,p)\in E} 2^{-s S_nu (x_i)} \;<\; \infty;
        \]
  \item \label{bscov:cover} for every $k,p \in \mathbb{N}$, the family
        \[
        \bigl\{\, B_{n}(x_{i},T,2^{-p}) \;:\; (i,n,p)\in E,\; n \ge k \bigr\}
        \]
        is a cover of $Z$.
\end{enumerate}
\end{definition}
The discretization characterizes the lower BS dimension, by an argument parallel to Lemma~\ref{lem:null_cover}.
\begin{lemma}\label{lem:bs_null_cover}
Suppose $\underline{u}(x)>0$ for all $x\in Z$. Then
\[\dim_{BS}(Z,T,u)=\inf\{s\geq 0\mid Z \text{ has a null
$s$-cover}\}.\]
\end{lemma}
\begin{proof}
($\leq$) Suppose $Z$ has a null $s$-cover $E$, and let $t>s$. Fix
$p$ and $N$. By condition~\ref{bscov:cover} the family
$\{(x_i,n)\mid (i,n,p)\in E,\ n\geq N\}$ belongs to
$\mathcal{P}(Z,N,2^{-p})$, so
\[
m_{BS}(Z,t,u,N,2^{-p})
\leq\sum_{\substack{(i,n,p)\in E\\ n\geq N}}2^{-t\,u(x_i,n,2^{-p})}\leq \sum_{\substack{(i,n,p)\in E\\ n\geq N}}2^{-t\,S_nu(x_i)}.
\]
We now invoke the dichotomy (Lemma~\ref{lem:bs-dichotomy}), which
applies unconditionally under the standing hypothesis
$\underline{u}>0$ on $Z$. The first display shows
$m_{BS}(Z,s',u,2^{-p})\leq
\sum_{(i,n,p)\in E}2^{-s'S_nu(x_i)}<\infty$ at $s'=s$, so
$m_{BS}(Z,t,u,2^{-p})=0$ for every $t>s$ and $\varepsilon=2^{-p}$
small enough. Hence $\dim_{BS}(Z,T,u,2^{-p})\leq t$ for all small
$2^{-p}$, so $\dim_{BS}(Z,T,u)\leq t$, and letting $t\downarrow s$
gives $\dim_{BS}(Z,T,u)\leq s$.

($\geq$) We show that for all rationals $s>\dim_{BS}(Z,T,u)$
the set $Z$ has a null $s$-cover. As $s$ was arbitrary this gives
$\inf\{s\geq 0\mid Z\text{ has a null $s$-cover}\}
\leq\dim_{BS}(Z,T,u)$, and by the preceding inequality the two quantities are equal (if
$\dim_{BS}(Z,T,u)=\infty$ there is nothing to prove).

Since $\underline{u}>0$ on $Z$, we satisfy the measure dichotomy (Remark~\ref{rem:mBSC-dichotomy}). 
Since the scale quantities are nondecreasing as the scale shrinks
and $m_{BS}$ is monotone under inclusion,
$\dim_{BS}(Z,T,u,\varepsilon)
\leq\dim_{BS}(Z,T,u)<s$, so $m_{BS}(Z,s,u,2^{-(p+1)})=0$ for every
$p$. Since $m_{BS}(Z,s,u,N,\cdot)$ is nondecreasing in $N$, its
vanishing limit forces vanishing at every $N$, so for each $p$ and $N$ there is a countable family
$\Gamma_{p,N}\in\mathcal{P}(Z,N,2^{-(p+1)})$ with
\[
\sum_{(y,n)\in\Gamma_{p,N}} 2^{-s\,u(y,n,2^{-(p+1)})}
\leq 2^{-(p+N+1)}.
\]
Since $(x_i)$ is dense in each Bowen metric $d_n$, choose for each
$(y,n)\in\Gamma_{p,N}$ an ideal point $x_{i(y,n)}\in B_n(y,T,2^{-(p+1)})$ with $|u(y,n,2^{-(p+1)})-S_nu(x_{i(y,n)})|< 1$. Such a point exists because ideal points approximate the supremum arbitrarily closely. Then
$B_n(y,T,2^{-(p+1)})\subset B_n(x_{i(y,n)},T,2^{-p})$ by the
triangle inequality. Set
\[E=\bigl\{(i(y,n),n,p)\mid \exists N\ 
(y,n)\in\Gamma_{p,N}\bigr\}.\]
For condition 1: each $(i,n,p)\in E$ arises from at least one pair $(N,(y,n))$ with $(y,n)\in\Gamma_{p,N}$, and the choice of $x_{i(y,n)}$ gives $S_nu(x_{i(y,n)})\geq u(y,n,2^{-(p+1)})-1$, so
\[\sum\limits_{(i,n,p)\in E}2^{-sS_nu(x_i)}\leq \sum\limits_{p\in \mathbb{N}}\sum\limits_{N\in \mathbb{N}} \sum\limits_{(y,n)\in \Gamma_{p,N}}2^{s}\,2^{-s\,u(y,n,2^{-(p+1)})}\leq 2^{s}\sum\limits_{p\in \mathbb{N}}\sum\limits_{N\in \mathbb{N}}2^{-(p+N+1)}\leq 2^{s}<\infty\]
For condition 2: given $x\in Z$ and any $\Gamma_{p,N}$, there is a pair $(y,n)$ with $n\geq N$ and $x\in B_n(y,T,2^{-(p+1)})\subset B_n(x_{i(y,n)},T,2^{-p})$, which witnesses the covering property.
\end{proof}

\subsubsection{The point-to-set principle for lower BS dimension}
\label{subsubsec:lower-bs-pts}
We now effectivize the discretized lower BS dimension. The
results of this subsubsection require $A$ to be admissible for
$(X,d,T,u)$ (Definition~\ref{def:admissible}). Each statement
carries this hypothesis.

\begin{definition}[$A$-effective null $s$-cover; effective lower BS
dimension]\label{def:eff-lower-bs}
An \emph{$A$-effective null $s$-cover} is a null $s$-cover
$E\subset\mathbb{N}^3$ that is $A$-computably enumerable. The
\emph{$A$-effective lower BS dimension} of $Z$ with respect to $u$ is
\[\dim^A_{BS,\mathrm{eff}}(Z,T,u)
=\inf\{s\geq 0\mid Z \text{ has an $A$-effective null
$s$-cover}\}.\]
\end{definition}

We now prove the pointwise characterization of the effective lower
BS dimension.

\begin{lemma}\label{lem:sup_eq_bs}
For $A$ admissible for $(X,d,T,u)$ and any $Z\subset X$ with
$\underline{u}(x)>0$ for all $x\in Z$,
\[\dim^A_{BS,\mathrm{eff}}(Z,T,u)
=\sup_{x\in Z} \dim_{BS}^A(x,T,u).\]
\end{lemma}

It is proved via the following two lemmas.

\begin{lemma}\label{lem:eff-bs-upper}
Let $A$ be admissible for $(X,d,T,u)$ and let $Z \subset X$ with
$\underline{u}(x) > 0$ for all $x \in Z$. Then
\[
\dim^A_{BS,\mathrm{eff}}(Z,T,u) \;\leq\; \sup_{x \in Z}\, \dim^A_{BS}(x,T,u).
\]
\end{lemma}
\begin{proof}
If the supremum is infinite there is nothing to prove. Fix a rational
$s > \sup_{x \in Z} \dim^A_{BS}(x,T,u)$ and set
$Y_{<s} = \{x : \underline{u}(x)>0,\ \dim^A_{BS}(x,T,u) < s\} \supseteq Z$.
We construct an $A$-effective null $s$-cover of $Y_{<s}$. Consider
\[
E = \{(i,n,p) : K^A(i,n,p) < s\, S_n u(x_i)\},
\]
which is $A$-c.e.\ by Remark~\ref{rem:ce-thresholds}, since $s$ is rational and
$S_n u(x_i)$ is $A$-computable uniformly in $(i,n)$
(Lemma~\ref{lem:birkhoff-uniform}).

For condition 1: membership in $E$ gives
$2^{-s S_n u(x_i)} \leq 2^{-K^A(i,n,p)}$, so by Kraft's inequality over the
distinct triples,
\[
\sum_{(i,n,p)\in E} 2^{-s S_n u(x_i)}
\;\leq\; \sum_{(i,n,p)\in E} 2^{-K^A(i,n,p)} \;\leq\; 1.
\]

For condition 2, fix $p \in \mathbb{N}$ and $x \in Y_{<s}$; we show
$x \in B_n(x_i, T, 2^{-p})$ with $(i,n,p) \in E$ for infinitely many $n$.
Choose a rational $s'$ with $\dim^A_{BS}(x,T,u) < s' < s$, and choose
$p^* \geq p$ with $\Var(u, 2^{-p^*}) < \underline{u}(x)/2$. Since
$\liminf\limits_n S_n u(x)/n = \underline{u}(x)$, there is $N_0$ such that for all
$n \geq N_0$ and all $x_i \in B_n(x, T, 2^{-p^*})$,
\[
S_n u(x_i) \;\geq\; S_n u(x) - n \Var(u, 2^{-p^*})
\;\geq\; n\,\underline{u}(x)/4 \;>\; 0.
\]
Because the scale quantities are nondecreasing as the scale shrinks,
$\dim^A_{BS}(x,T,u,2^{-p^*}) \leq \dim^A_{BS}(x,T,u) < s'$, so for infinitely
many $n$ there is a realizer $x_i \in B_n(x, T, 2^{-p^*})$ with
$K^A(i) < s'\, S_n u(x_i)$. For such a realizer with $n \geq N_0$,
\[
K^A(i,n,p) \;\leq\; K^A(i) + 2\log n + 2\log p + O(1)
\;<\; s'\, S_n u(x_i) + 2\log n + 2\log p + O(1)
\;\leq\; s\, S_n u(x_i)
\]
by Remark~\ref{rem:composition-use}, once $n$ is large enough that
$(s-s')\, S_n u(x_i) \geq (s-s')\, n\,\underline{u}(x)/4$ absorbs
$2\log n + 2\log p + O(1)$; hence $(i,n,p) \in E$ for infinitely many $n$.
Since $d_n(x, x_i) < 2^{-p^*} \leq 2^{-p}$, each such ball satisfies
$x \in B_n(x_i, T, 2^{-p})$. Thus for every $k$,
\[
\{B_n(x_i, T, 2^{-p}) : (i,n,p) \in E,\ n \geq k\}
\quad \text{covers } Y_{<s},
\]
and $E$ is an $A$-effective null $s$-cover of $Y_{<s}$. A null $s$-cover of
$Y_{<s}$ is in particular one of $Z \subset Y_{<s}$, so
$\dim^A_{BS,\mathrm{eff}}(Z,T,u) \leq s$; as $s$ was an arbitrary rational
above the supremum, the lemma follows.
\end{proof}

\begin{lemma}\label{lem:pointwise-leq-eff-bs}
Let $Z \subset X$ with $\underline{u}(x) > 0$ for all $x \in Z$. For all $x \in Z$, $\dim^A_{BS}(x,T,u) \leq \dim^A_{BS,\mathrm{eff}}(Z,T,u)$.
\end{lemma}
\begin{proof}
Let $s > \dim^A_{BS,\mathrm{eff}}(Z,T,u)$. Then $Z$ has an $A$-effective null $s$-cover $E$. As
\[
\sum_{(i,n,p)\in E} 2^{-sS_n u( x_i)} < \infty,
\]
by the coding theorem we have
\[K^A(i,n,p)\leq sS_n u( x_i)+c\]
for some constant $c$, which does not depend on $i,n,p$. Fix $x \in Z$, and let $p$ be large enough that $\Var(u,2^{-p}) < u(x)$. By condition 2, we have a sequence $(i_k,n_k)$ with $n_k\to\infty$ so that $x\in B_{n_k}(x_{i_k},T,2^{-p})$. Therefore $x_{i_k}\in B_{n_k}(x,T,2^{-p})$, and since $\Var(u,2^{-p}) < u(x)$, we have $S_{n_k}u(x_{i_k}) \to \infty$. We get
\[\min\limits_{x_i\in B_{n_k}(x,T,2^{-p})}\frac{K^A(i)}{S_{n_k}u(x_{i})}\leq \frac{K^A(i_k,n_k,p)}{S_{n_k} u( x_{i_k})}\leq  \frac{sS_{n_k}u( x_{i_k})+c}{S_{n_k} u( x_{i_k})}\leq s+\frac{c}{S_{n_k} u( x_{i_k})}\]
Therefore on the subsequence $n_k$ we have an upper bound which decays to $s$, which gives an upper bound
\[\dim^A_{BS}(x,T,u,2^{-p})\leq s\]
As this holds for all sufficiently large $p$,
\[\dim^A_{BS}(x,T,u)\leq s\]
As $s > \dim^A_{BS,\mathrm{eff}}(Z,T,u)$ was arbitrary, the lemma follows.
\end{proof}
\begin{lemma}\label{lem:oracle_exists_lower_bs}
Given $Z\subset X$ with $\underline{u}(x)>0$ for all $x\in Z$, there
is an oracle $A$ such that
\[\dim_{BS}(Z,T,u)=\dim^A_{BS,\mathrm{eff}}(Z,T,u);\]
moreover every oracle $\geq_T A$ satisfies the same identity.
\end{lemma}
\begin{proof}
As in Lemmas~\ref{exists-oracle}
and~\ref{lem:exists-oracle-pressure}: by
Lemma~\ref{lem:bs_null_cover}, choose null $s_j$-covers
$E_j\subset\mathbb{N}^3$ for $Z$ with rational
$s_j\downarrow\dim_{BS}(Z,T,u)$, and let $A=\bigoplus_j E_j$. Each
$E_j$ is $A$-c.e., hence an $A$-effective null $s_j$-cover, so
$\dim^A_{BS,\mathrm{eff}}(Z,T,u)\leq s_j$ for every $j$; the reverse
inequality holds for every oracle, by Lemma~\ref{lem:bs_null_cover}
and the inclusion of effective covers among covers. The moreover
clause follows since any $B\geq_T A$ enumerates each $E_j$.
\end{proof}
\begin{theorem}[Point-to-set principle for lower BS dimension]
\label{thm:pts_lower_bs}
Let $Z\subset X$ with $\underline{u}(x)>0$ for all $x\in Z$. Then
\[\dim_{BS}(Z,T,u)
=\min\limits_{A\subset \mathbb{N}}\,\sup_{x\in Z}\dim_{BS}^A(x,T,u),\]
with the minimum attained on an upper set of Turing degrees.
\end{theorem}
\begin{proof}
Identical to Theorems~\ref{thm:pts-entropy}
and~\ref{thm:pts_lower_pressure}: for every oracle $A$ we have
$\dim_{BS}(Z,T,u)\leq\dim^A_{BS,\mathrm{eff}}(Z,T,u)$
(Lemma~\ref{lem:bs_null_cover}), and for admissible $A$,
$\dim^A_{BS,\mathrm{eff}}(Z,T,u)=\sup_{x\in Z}\dim_{BS}^A(x,T,u)$ by
Lemma~\ref{lem:sup_eq_bs}. By
Lemma~\ref{lem:oracle_exists_lower_bs} there is an oracle $B$ with
$\dim_{BS}(Z,T,u)=\dim^B_{BS,\mathrm{eff}}(Z,T,u)$, and every oracle
$\geq_T B$ satisfies the same identity. Joining with an admissible
oracle (Remark~\ref{rem:admissible-cone}), we may take $B$
admissible, and the same holds for every oracle above $B$. The
minimum is therefore attained at $B$, on the upper cone above it.
\end{proof}
\subsubsection{Dimensional definitions of upper BS dimension}
As with packing pressure, packing BS dimension admits an equivalent characterization as a \emph{modified upper capacity BS dimension}, and this is again the formulation we discretize and effectivize. We adopt it as our working definition.
\begin{makethm}{definition}{def:mBSC}[Modified upper capacity BS dimension]
Let $Z\subset X$ with $\underline{u}(x)>0$ for all $x\in Z$, and let
$s\geq 0$, $n\in\mathbb{N}$, $\varepsilon>0$. Recall the sup-weight $u(x,n,\varepsilon)$ of
Definition~\ref{def:lower-bs}. Set
\[
M(Z,T,u,s,n,\varepsilon)
=\inf\Bigg\{\sum_{x\in \Gamma} 2^{-s\, u(x,n,\varepsilon)}
\;\Bigg|\; \Gamma\subset X \text{ countable},\
Z\subset \bigcup_{x\in \Gamma}B_n(x,T,\varepsilon) \Bigg\}.
\]
\[M(Z,T,u,s,\varepsilon)=\limsup\limits_{n\to\infty} M(Z,T,u,s,n,\varepsilon)\]
The \emph{upper capacity BS dimension} of $Z$ at scale $\varepsilon$ is
\[
\dim_{\mathrm{BSC}}(Z,T,u,\varepsilon)
=\inf\Big\{s \geq 0 : M(Z,T,u,s,\varepsilon) <\infty\Big\},
\]
with the convention $\inf\emptyset = \infty$. The \emph{modified upper
capacity BS dimension} at scale $\varepsilon$ is
\[
\Dim_{BS}(Z,T,u,\varepsilon)
=\inf\Bigg\{\,\sup_{i\in\mathbb{N}}
\dim_{\mathrm{BSC}}(Z_i,T,u,\varepsilon)
\;\Bigg|\; Z\subset\bigcup_{i\in\mathbb{N}} Z_i\Bigg\},
\]
where the infimum is over countable covers of $Z$, and the
\emph{upper BS dimension} of $Z$ with respect to $u$ is
\[
\Dim_{BS}(Z,T,u)
=\lim\limits_{\varepsilon\to 0}\Dim_{BS}(Z,T,u,\varepsilon);
\]
the limit exists because for $\varepsilon_1 < \varepsilon_2$ every cover at
scale $\varepsilon_1$ is a cover at scale $\varepsilon_2$ and
$u(x,n,\varepsilon_1) \leq u(x,n,\varepsilon_2)$, so (as $s \geq 0$)
$M(Z,T,u,s,n,\varepsilon_1) \geq M(Z,T,u,s,n,\varepsilon_2)$ and the scale-
$\varepsilon$ quantities are nondecreasing as $\varepsilon \to 0$.
\end{makethm}

The relationship between the modified upper capacity BS dimension and
the packing BS dimension is the same as that between modified upper
capacity pressure and packing pressure: the two coincide, as we show
in Appendix~\ref{app:packing-bs} (Theorem~\ref{thm:packing-bs-equiv}).
Henceforth we take this as our
definition of the upper BS dimension and write $\Dim_{BS}$ throughout.

\begin{remark}\label{rem:mBSC-dichotomy}
We define $\dim_{\mathrm{BSC}}$ by the finiteness threshold because the
dichotomy
\[
\inf\{s : \limsup\limits_{n\to\infty}M(Z,T,u,s,n,\varepsilon) = 0\}
= \sup\{s : \limsup\limits_{n\to\infty}M(Z,T,u,s,n,\varepsilon) = \infty\}
\]
need not hold for arbitrary $Z\subset X$ at a fixed scale: the cover
centers range over $X$, where $S_n u$ may be negative, and the shift
argument deriving decay at $s' > s$ from finiteness at $s$ fails. The
modified quantity is nonetheless insensitive to this choice. For
$N\in\mathbb{N}$ and rational $\delta>0$ let
\[
C(Z,N,\delta)=\{x\in Z : S_nu(x)\geq \delta n \text{ for all } n\geq N\};
\]
since $\underline{u}>0$ on $Z$, these countably many sets exhaust $Z$.
Given any countable cover $Z\subset\bigcup_i Z_i$, the refinement
$Z'_{i,N,\delta} = Z_i\cap C(Z,N,\delta)$ is again a countable cover, and
on each piece the dichotomy holds: every cover ball meeting
$Z'_{i,N,\delta}$ contains a point $\tilde{x}$ with
$S_nu(\tilde{x})\geq\delta n$, so its sup-weight satisfies
$u(x,n,\varepsilon)\geq \delta n$, and finiteness of the $M$-sums at $s$
forces their decay at every $s'>s$ by the factor
$2^{-(s'-s)\delta n}$ (the argument of Lemma~\ref{lem:bs-dichotomy},
which applies verbatim). Since refining a cover does not increase any
$\dim_{\mathrm{BSC}}(Z_i,T,u,\varepsilon)$, we conclude
\[
\Dim_{BS}(Z,T,u,\varepsilon)
=\inf\Bigg\{\,\sup_{i}\,
\inf\{s : \limsup\limits_n M(Z_i,T,u,s,n,\varepsilon) = 0\}
\;\Bigg|\; Z\subset\bigcup_{i} Z_i\Bigg\}:
\]
the finiteness, vanishing, and infinite thresholds define the same modified
quantity, and either may be used below. We choose to use the finiteness threshold since it is convenient in the discretization.\\
The packing BS dimension of
Appendix~\ref{app:packing-bs} requires no such refinement: its packing
centers lie in $Z$ itself, and the dichotomy holds outright
(Lemma~\ref{lem:packing-bs-dichotomy}).
\end{remark}
We now discretize this definition.

\begin{definition}[Uniform $s$-net for BS dimension]
\label{def:bs-s-net}
A \emph{uniform $s$-net for $Z$} is a set $E\subset\mathbb{N}^4$
such that, writing $E_{n,p,k}=\{i\mid (i,n,p,k)\in E\}$:
\begin{enumerate}
    \item \label{bsnet:mass} for all $p,k\in\mathbb{N}$,
    \[
    \limsup\limits_{n\to\infty}\sum_{i\in E_{n,p,k}} 2^{-s\, S_nu(x_i)}<\infty;
    \]
    \item \label{bsnet:cover} for each $p\in\mathbb{N}$,
    \[
    Z\subset\bigcup_{k\in\mathbb{N}}\bigcap_{n\in\mathbb{N}}
    \bigcup_{i\in E_{n,p,k}} B_n(x_i,T,2^{-p}).
    \]
\end{enumerate}
In this formulation, the countable decomposition of $Z$ at scale
$2^{-p}$ is given by the sets
$Z^{(p)}_k=\bigcap_{n}\bigcup_{i\in E_{n,p,k}}B_n(x_i,T,2^{-p})$,
$k\in\mathbb{N}$: condition~\ref{bsnet:mass} witnesses
$\limsup\limits_{n\to\infty}M(Z^{(p)}_k,T,u,s,n,2^{-p})<\infty$ for each $k$,
and condition~\ref{bsnet:cover} witnesses
$Z\subset\bigcup_k Z^{(p)}_k$.
\end{definition}
We check that this discretization computes the upper BS dimension.
\begin{lemma}\label{lem:bs-s-net}
Suppose $\underline{u}(x)>0$ for all $x\in Z$. Then
\[
\Dim_{BS}(Z,T,u)=\inf\{s\geq 0\mid Z \text{ has a uniform
$s$-net}\}.
\]
\end{lemma}
\begin{proof}
($\leq$) Suppose $Z$ has a uniform $s$-net $E$, and fix
$p\in\mathbb{N}$. By condition~\ref{bsnet:cover}, the sets
$Z^{(p)}_k=\bigcap_n\bigcup_{i\in E_{n,p,k}}B_n(x_i,T,2^{-p})$
cover $Z$. For each $k$ and every $n$ we have
$Z^{(p)}_k\subset\bigcup_{i\in E_{n,p,k}}B_n(x_i,T,2^{-p})$ by
construction, so
\[
M(Z^{(p)}_k,T,u,s,n,2^{-p})
\leq\sum_{i\in E_{n,p,k}} 2^{-s\, S_nu(x_i)}<\infty,
\]
by condition~\ref{bsnet:mass}, hence
$\dim_{\mathrm{BSC}}(Z^{(p)}_k,T,u,2^{-p})\leq s$ for every $k$.
Therefore $\Dim_{BS}(Z,T,u,2^{-p})\leq s$; as $p$ was arbitrary,
$\Dim_{BS}(Z,T,u)\leq s$. Taking the infimum over such $s$ gives
$\Dim_{BS}(Z,T,u)
\leq\inf\{s\mid Z\text{ has a uniform $s$-net}\}$.

($\geq$) We show that for all rationals $s>\Dim_{BS}(Z,T,u)$ the
set $Z$ has a uniform $s$-net. This gives
$\inf\{s\geq 0\mid Z\text{ has a uniform $s$-net}\}\leq\Dim_{BS}(Z,T,u)$.

As in Lemma~\ref{lem:bs_null_cover}, write $Z=\bigcup_j C_j$ with
$C_j=C(Z,N_j,\rho_j)$ (Remark~\ref{rem:mBSC-dichotomy}). Uniform
$s$-nets merge across a countable union: if $E^{(j)}$ is a uniform
$s$-net for $C_j$, then
$E=\{(i,n,p,\langle j,k\rangle)\mid (i,n,p,k)\in E^{(j)}\}$, with
$\langle\cdot,\cdot\rangle$ a pairing of $\mathbb{N}^2$ with
$\mathbb{N}$, satisfies both conditions of
Definition~\ref{def:bs-s-net} for $Z$, since each condition is
imposed slice by slice. It therefore suffices to fix $j$, write
$C=C_j$, $N_0=N_j$, $\rho=\rho_j$, and construct a uniform
$s$-net for $C$.

Fix $p\in\mathbb{N}$. Since
$\Dim_{BS}(C,T,u,2^{-(p+1)})\leq\Dim_{BS}(Z,T,u)<s$, there is a
countable decomposition $C=\bigcup_k C^{(p)}_k$ with
$\dim_{\mathrm{BSC}}(C^{(p)}_k,T,u,2^{-(p+1)})<s$ for every $k$. After discarding cover balls that
miss $C^{(p)}_k$, every remaining ball of order $N\geq N_0$ has
sup-weight $u(y,N,2^{-(p+1)})\geq\rho N$, so we get $M(C^{(p)}_k,T,u,s,2^{-p})=0$. Therefore for all $N$ large
enough there is a countable family $\Gamma_{k,p,N}\subset X$ with
\[
\sum_{y\in \Gamma_{k,p,N}} 2^{-s\, u(y,N,2^{-(p+1)})}<1,
\qquad
C^{(p)}_k\subset\bigcup_{y\in \Gamma_{k,p,N}}B_N(y,T,2^{-(p+1)}).
\]
For the finitely many remaining orders $N$, let $\Gamma_{k,p,N}$ be
any countable family with the covering property, which exists by
separability; these orders do not affect the $\limsup$ below.

For each $y\in\Gamma_{k,p,N}$ pick an index $i(y,N)$ with
$x_{i(y,N)}\in B_N(y,T,2^{-(p+1)})$ and, for $N$ in the large range,
$|u(y,N,2^{-(p+1)})-S_Nu(x_{i(y,N)})|<1$; such indices exist since
the ideal points are $d_N$-dense and approximate the supremum
arbitrarily closely. Since $x_{i(y,N)}$ lies in the ball over which
the supremum is taken, $u(y,N,2^{-(p+1)})\geq S_Nu(x_{i(y,N)})$, so
\[
\sum_{y\in \Gamma_{k,p,N}}2^{-s\, S_Nu(x_{i(y,N)})}
\leq \sum_{y\in \Gamma_{k,p,N}}
2^{s\,(u(y,N,2^{-(p+1)})-S_Nu(x_{i(y,N)}))}\,
2^{-s\, u(y,N,2^{-(p+1)})}
\leq 2^{s}\sum_{y\in \Gamma_{k,p,N}} 2^{-s\, u(y,N,2^{-(p+1)})}
\leq 2^{s}.
\]
Let $E=\{(i(y,N),N,p,k)\mid p,k,N\in\mathbb{N},\
y\in \Gamma_{k,p,N}\}$, so that
$E_{N,p,k}=\{i(y,N)\mid y\in\Gamma_{k,p,N}\}$. Each $i\in E_{N,p,k}$
arises from at least one $y$, so
\[
\limsup\limits_{N\to\infty}\sum_{i\in E_{N,p,k}} 2^{-s\, S_Nu(x_i)}
\leq 2^{s}<\infty,
\]
which is condition~\ref{bsnet:mass}. For
condition~\ref{bsnet:cover}: by the triangle inequality
$B_N(y,T,2^{-(p+1)})\subset B_N(x_{i(y,N)},T,2^{-p})$, so for
every $N$,
$C^{(p)}_k\subset\bigcup_{i\in E_{N,p,k}}B_N(x_i,T,2^{-p})$;
intersecting over $N$ and taking the union over $k$ shows that $E$
is a uniform $s$-net for $C$.
\end{proof}
\subsubsection{The point-to-set principle for upper BS dimension}\label{subsec_pts_BS_upper}
As before, we effectivize the discretization. The
results of this section require $A$ to be admissible for
$(X,d,T,u)$ (Definition~\ref{def:admissible}). Each statement
carries this hypothesis.

\begin{definition}[$A$-effective uniform $s$-net; effective upper BS
dimension]\label{def:eff-upper-bs}
An \emph{$A$-effective uniform $s$-net} is a uniform $s$-net
$E\subset\mathbb{N}^4$ that is $A$-computably enumerable. The
\emph{$A$-effective upper BS dimension} of $Z$ with respect to $u$ is
\[\Dim_{BS,\mathrm{eff}}^A(Z,T,u)
=\inf\{s\geq 0\mid Z \text{ has an $A$-effective uniform
$s$-net}\}.\]
If $Z\subset Z'$ then
$\Dim^A_{BS,\mathrm{eff}}(Z,T,u)\leq
\Dim^A_{BS,\mathrm{eff}}(Z',T,u)$, since a uniform $s$-net for $Z'$
is one for $Z$.
\end{definition}

We now prove the pointwise characterization of the effective upper
BS dimension.

\begin{lemma}\label{lem:sup_eq_upper_bs}
For $A$ admissible for $(X,d,T,u)$ and any $Z\subset X$ with
$\underline{u}(x)>0$ for all $x\in Z$,
\[\Dim_{BS,\mathrm{eff}}^A(Z,T,u)
=\sup_{x\in Z}\Dim_{BS}^A(x,T,u).\]
\end{lemma}

Again the proof splits into two lemmas.

\begin{lemma}\label{lem:eff_upper_bs}
    \[\Dim^A_{BS,\mathrm{eff}}(Z,T,u)\leq
    \sup\limits_{x\in Z}\Dim_{BS}^A(x,T,u)\]
\end{lemma}
\begin{proof}
Fix a rational $s>\sup\limits_{x\in Z}\Dim_{BS}^A(x,T,u)$ and
define $Y_{<s}=\{x\mid \Dim_{BS}^A(x,T,u)<s, \underline{u}(x)>0\}\supset Z$. We
construct an $A$-effective uniform $s$-net for $Y_{<s}$. Consider
\[E=\bigl\{(i,n,p,k)\;\bigm|\; K^A(i)< sS_nu(x_i)+k\bigr\},\]
which is $A$-c.e.\ by Remark~\ref{rem:ce-thresholds}, since
$S_nu(x_i)$ is $A$-computable uniformly in $(i,n)$
(Lemma~\ref{lem:birkhoff-uniform}). The condition is independent
of $p$, and the slices are increasing in $k$:
$E_{n,p,k}\subset E_{n,p,k+1}$.

We first verify condition 2. Fix $p\in\mathbb{N}$ and $x\in Y_{<s}$, and choose
$p^* \geq p$ with $\Var(u, 2^{-p^*}) < \underline{u}(x)/2$.Then there is $N_0$ such that for all
$n \geq N_0$ and all $x_i \in B_n(x, T, 2^{-p^*})$,
\[
S_n u(x_i) \;\geq\; S_n u(x) - n \Var(u, 2^{-p^*})
\;\geq\; n\,\underline{u}(x)/4 \;>\; 0.
\]
Because the scale quantities are nondecreasing as the scale shrinks,
$\Dim^A_{BS}(x,T,u,2^{-p^*}) \leq \Dim^A_{BS}(x,T,u) < s$, so there is
$N\geq N_0$ such that for every $n\geq N$ there is a realizer
$x_{i_n}\in B_n(x,T,2^{-p^*})$ with $K^A(i_n)< sS_nu(x_{i_n})$,
so $(i_n,n,p,k)\in E$ for every $k\geq 0$. For each of the finitely
many $n<N$, choose by density any ideal point
$x_{i_n}\in B_n(x,T,2^{-p^*})$; then
$K^A(i_n)< sS_nu(x_{i_n})+k_n$ for some finite $k_n$. Taking
$k=\max_{n<N}k_n$, monotonicity of the slices in $k$ gives
$i_n\in E_{n,p,k}$ for every $n$. Since
$d_n(x,x_{i_n})<2^{-p^*}\leq 2^{-p}$, we have
$x\in B_n(x_{i_n},T,2^{-p})\subset\bigcup_{i\in E_{n,p,k}}B_n(x_i,T,2^{-p})$
for every $n$, hence $x\in Z^{(p)}_k$. Therefore
$Y_{<s}\subset\bigcup_k Z^{(p)}_k$ for every $p$.

We now verify condition 1. Fix $p,k\in\mathbb{N}$. By membership in
$E$ and Kraft's inequality over the distinct indices $i$,
\[\sum_{i\in E_{n,p,k}} 2^{-sS_nu(x_i)}
\leq \sum_{i\in E_{n,p,k}} 2^{-K^A(i)+k}
=2^{k}\sum_{i\in E_{n,p,k}} 2^{-K^A(i)}
\leq 2^{k}<\infty\]
Hence $E$ is an $A$-effective uniform $s$-net for
$Y_{<s}\supseteq Z$, so by monotonicity of
$\Dim^A_{BS,\mathrm{eff}}$ under inclusion,
$\Dim^A_{BS,\mathrm{eff}}(Z,T,u)\leq s$; as $s$ was an
arbitrary rational above the supremum, the lemma follows.
\end{proof}

We now prove the lower bound
\begin{lemma}\label{lem:eff_upper_bs_geq}
For $A$ admissible for $(X,d,T,u)$ and $Z\subset X$ with
$\underline{u}(x)>0$ for all $x\in Z$,
\[\Dim_{BS,\mathrm{eff}}^A(Z,T,u)\geq
\sup_{x\in Z}\Dim_{BS}^A(x,T,u).\]
\end{lemma}
\begin{proof}
Fix a rational $s>\Dim_{BS,\mathrm{eff}}^A(Z,T,u)$ and let $E\subset\mathbb{N}^4$ be an $A$-effective uniform $s$-net for $Z$. Fix $x\in Z$, choose a rational $\delta$ with $0<\delta<\underline{u}(x)$, fix a rational $t>s$, and fix $p\in\mathbb{N}$ large enough that $\Var(u,2^{-p})<\delta/2$. By condition~\ref{bsnet:cover} there is $k\in\mathbb{N}$ with
$x\in Z_k^{(p)}=\bigcap_{n\in\mathbb{N}}
\bigcup_{i\in E_{n,p,k}} B_n(x_i,T,2^{-p})$; fix such a $k$. Define the truncated slice family
\[E_{n,p,k}'=\{i\in E_{n,p,k}\mid S_nu(x_i)> \delta n/2\},\]
which is $A$-c.e.\ uniformly since $E$ is $A$-c.e.\ and the strict inequality is witnessed by a sufficiently precise $A$-computation of $S_nu(x_i)$ (Lemma~\ref{lem:birkhoff-uniform}, Remark~\ref{rem:ce-thresholds}).

We first observe that the truncated slices still capture $x$ for large $n$. Since $\liminf\limits_n S_nu(x)/n=\underline{u}(x)>\delta$, there is $N_1$ with $S_nu(x)>\delta n$ for all $n\geq N_1$. For $n\geq N_1$, condition~\ref{bsnet:cover} provides $i\in E_{n,p,k}$ with $x\in B_n(x_i,T,2^{-p})$; then
\[S_nu(x_i)\geq S_nu(x)-n\,\Var(u,2^{-p})>\delta n-\delta n/2=\delta n/2,\]
so $i\in E'_{n,p,k}$. Hence for every $n\geq N_1$ there is a witness $i_n\in E'_{n,p,k}$ with $x\in B_n(x_{i_n},T,2^{-p})$.

By condition~\ref{bsnet:mass} of Definition~\ref{def:bs-s-net}, $\limsup\limits_n\sum_{i\in E_{n,p,k}} 2^{-s S_nu(x_i)}<\infty$, so there are $M<\infty$ and $N_0\in\mathbb{N}$ with
\[
\sum_{i\in E_{n,p,k}} 2^{-s S_nu(x_i)}\leq M\qquad\text{for all } n\geq N_0;
\]
since $E'_{n,p,k}\subset E_{n,p,k}$ and the terms are positive, the same bound holds over $E'_{n,p,k}$. For $n\geq N_0$, the threshold defining $E'$ gives
\[
\sum_{i\in E'_{n,p,k}} 2^{-t S_nu(x_i)}
\leq \sum_{i\in E'_{n,p,k}} 2^{-(t-s)\delta n/2}\, 2^{-s S_nu(x_i)}
= 2^{-(t-s)\delta n/2}\sum_{i\in E'_{n,p,k}} 2^{-s S_nu(x_i)}
\leq M\,2^{-(t-s)\delta n/2}.
\]
Therefore
\[
\sum_{n\geq N}\sum_{i\in E'_{n,p,k}} 2^{-tS_nu(x_i)}
\leq M\sum_{n\geq N} 2^{-(t-s)\delta n/2}
\xrightarrow{\;N\to\infty\;} 0,
\]
the tail of a convergent geometric series, and we may fix $N\geq\max(N_0,N_1)$ large enough that
\[
\sum_{n\geq N}\sum_{i\in E'_{n,p,k}} 2^{-tS_nu(x_i)}\leq 1.
\]
Consequently the function
\[
(i,n)\mapsto tS_nu(x_i),
\qquad\text{defined for } n\geq N,\ i\in E'_{n,p,k},
\]
satisfies the hypotheses of the coding theorem: it satisfies the Kraft inequality by the choice of $N$, and it is upper semicomputable in $A$ since $E'_{n,p,k}$ is $A$-c.e.\ and $S_nu(x_i)$ is $A$-computable uniformly in $(i,n)$ (Lemma~\ref{lem:birkhoff-uniform}). By the coding theorem (Lemma~\ref{lem:coding-theorem}),
\[
K^A(i,n)\leq tS_nu(x_i)+O(1)
\qquad\text{for all } n\geq N,\ i\in E'_{n,p,k},
\]
where the constant depends only on $E$, $p$, $k$, $t$, and $\delta$. Since
$K^A(i)\leq K^A(i,n)+O(1)$ (Remark~\ref{rem:composition-use}),
\[
K^A(i)\leq tS_nu(x_i)+O(1)
\qquad\text{for all } n\geq N,\ i\in E'_{n,p,k}.
\]
The witnesses $i_n$ satisfy this bound once $n\geq N$, and $S_nu(x_{i_n})>\delta n/2\to\infty$, so
\[
\Dim_{BS}^A(x,T,u,2^{-p})
=\limsup\limits_{n\to\infty}\;\min\limits_{x_i\in B_n(x,T,2^{-p})}
\frac{K^A(i)}{S_nu(x_i)}
\leq\limsup\limits_{n\to\infty}\frac{tS_nu(x_{i_n})+O(1)}{S_nu(x_{i_n})}
\leq t.
\]
The scale quantities are nondecreasing as the scale shrinks, and the bound holds for every sufficiently large $p$ (namely those with $\Var(u,2^{-p})<\delta/2$), so letting $p\to\infty$ gives $\Dim_{BS}^A(x,T,u)\leq t$. As $t>s$ was an arbitrary rational, $\Dim_{BS}^A(x,T,u)\leq s$; as $x\in Z$ was arbitrary, $\sup_{x\in Z}\Dim_{BS}^A(x,T,u)\leq s$; and as $s>\Dim_{BS,\mathrm{eff}}^A(Z,T,u)$ was an arbitrary rational, the lemma follows.
\end{proof}

We now construct an achieving oracle.

\begin{lemma}\label{lem:oracle_exists_upper_bs}
Given $Z\subset X$ with $\underline{u}(x)>0$ for all $x\in Z$, there
is an oracle $A$ such that
\[\Dim_{BS}(Z,T,u)=\Dim^A_{BS,\mathrm{eff}}(Z,T,u);\]
moreover every oracle $\geq_T A$ satisfies the same identity.
\end{lemma}
\begin{proof}
As in Lemmas~\ref{exists-oracle},
\ref{lem:exists-oracle-pressure},
and~\ref{lem:oracle_exists_upper_pressure}: by
Lemma~\ref{lem:bs-s-net}, choose uniform $s_j$-nets
$E_j\subset\mathbb{N}^4$ for $Z$ with rational
$s_j\downarrow\Dim_{BS}(Z,T,u)$, and let $A=\bigoplus_j E_j$. Each
$E_j$ is $A$-c.e., hence an $A$-effective uniform $s_j$-net, so
$\Dim^A_{BS,\mathrm{eff}}(Z,T,u)\leq s_j$ for every $j$; the reverse
inequality holds for every oracle, by Lemma~\ref{lem:bs-s-net} and
the inclusion of effective nets among nets. The moreover clause
follows since any $B\geq_T A$ enumerates each $E_j$, and the reverse
inequality again holds for $B$.
\end{proof}
\begin{makethm}{theorem}{thm:pts_upper}[Point-to-set principle for upper BS dimension]
Let $Z\subset X$ with $\underline{u}(x)>0$ for all $x\in Z$. Then
\[\Dim_{BS}(Z,T,u)
=\min\limits_{A\subset \mathbb{N}}\,\sup_{x\in Z}\Dim_{BS}^A(x,T,u),\]
with the minimum attained on an upper set of Turing degrees.
\end{makethm}
\begin{proof}
Identical to Theorems~\ref{thm:pts-entropy},
\ref{thm:pts_lower_pressure}, and~\ref{thm:pts_upper_pressure}: for
every oracle $A$ we have
$\Dim_{BS}(Z,T,u)\leq\Dim^A_{BS,\mathrm{eff}}(Z,T,u)$
(Lemma~\ref{lem:bs-s-net}), and for admissible $A$,
$\Dim^A_{BS,\mathrm{eff}}(Z,T,u)=\sup_{x\in Z}\Dim_{BS}^A(x,T,u)$ by
Lemma~\ref{lem:sup_eq_upper_bs}. By
Lemma~\ref{lem:oracle_exists_upper_bs} there is an oracle $B$ with
$\Dim_{BS}(Z,T,u)=\Dim^B_{BS,\mathrm{eff}}(Z,T,u)$, and every oracle
$\geq_T B$ satisfies the same identity. Joining with an admissible
oracle (Remark~\ref{rem:admissible-cone}), we may take $B$
admissible, and the same holds for every oracle above $B$. The
minimum is therefore attained at $B$, on the upper cone above it.
\end{proof}
\section{The Bowen pressure equation for BS dimension and dimension theory for nonuniformly expanding systems}\label{sec:bowen}
\subsection{The pointwise Bowen pressure equation}\label{sec:bowen-equation}
The Bowen pressure equation, originating in Bowen's work on the dimension of quasicircles\cite{bowen1979hausdorff},
is one of the founding results of the dimension theory of dynamical systems. Barreira and Schmeling \cite{Barreira2000Sets} introduced the BS dimension, recasting the solution of the equation as a dimension quantity in its own right; Climenhaga \cite{CLIMENHAGA2014The} later extended the equation to a generalized BS dimension under a global covering condition.\\
In this section we establish the Bowen pressure equation for the generalized BS dimension defined above, removing the global condition of \cite{CLIMENHAGA2014The}. Rather than arguing through covers as is standard, we prove the equation pointwise and treat the passage to the set level as a demonstration of how point-to-set principles convert between dimension quantities.  \\
The section proceeds as follows: we define admissible scales, prove monotonicity of the pointwise pressure in $s$, establish the pointwise upper and lower bounds, deduce the pointwise pressure equation, and derive the set-level equation from the pointwise one. \\
\begin{definition}[Admissible scale]\label{def:admissible-scale}
Let $x\in X$ with $0<\underline{u}(x)\leq\overline{u}(x)<\infty$. A
scale $\varepsilon>0$ is \emph{admissible for $x$} if
$\Var(u,\varepsilon)<\underline{u}(x)/2$ and there is $N$
such that for all $n\geq N$ and all
$x_i\in B_n(x,T,\varepsilon)$,
\[
\underline{u}(x)-2\Var(u,\varepsilon)
\;\leq\;\frac{S_nu(x_i)}{n}
\;\leq\;\overline{u}(x)+2\Var(u,\varepsilon).
\]
Every $\varepsilon'<\varepsilon$ is then admissible, and every such
$x$ has admissible scales by uniform continuity of $u$. Note that
admissibility forces $\underline{u}(x)-2\Var(u,\varepsilon)>0$.

Furthermore, for admissible $\varepsilon>0$ and $s\geq 0$, all four
quantities
\[
\underline{P}^A(x,T,-su,\varepsilon),\quad
\underline{P}^A(x,T,-su),\quad
\overline{P}^A(x,T,-su,\varepsilon),\quad
\overline{P}^A(x,T,-su)
\]
are bounded below by
$-s\bigl(\overline{u}(x)+2\Var(u,\varepsilon)\bigr)>-\infty$.
\end{definition}
We first demonstrate monotonicity of the pressure at the pointwise level.
\begin{lemma}[Pointwise monotonicity of the pressure family]
\label{lem:monotone_decreasing}
Suppose $0<\underline{u}(x)\leq\overline{u}(x)<\infty$. Then for
every admissible $\varepsilon > 0$ and $s < s'$, we have respectively
\begin{align}
  \underline{P}^A(x,T,-s'u,\varepsilon)
  &\leq \underline{P}^A(x,T,-su,\varepsilon)
  - (s'-s)\bigl(\underline{u}(x)-2\Var(u,\varepsilon)\bigr),
  \label{eq:mono-lower}\\
  \overline{P}^A(x,T,-s'u,\varepsilon)
  &\leq \overline{P}^A(x,T,-su,\varepsilon)
  - (s'-s)\bigl(\underline{u}(x)-2\Var(u,\varepsilon)\bigr),
  \label{eq:mono-upper}
\end{align}
and at scale $0$,
\begin{align}
  \underline{P}^A(x,T,-s'u)
  &\leq \underline{P}^A(x,T,-su) - (s'-s)\,\underline{u}(x),
  \label{eq:mono-lower-scale0}\\
  \overline{P}^A(x,T,-s'u)
  &\leq \overline{P}^A(x,T,-su) - (s'-s)\,\underline{u}(x).
  \label{eq:mono-upper-scale0}
\end{align}
In particular, since admissibility gives
$\Var(u,\varepsilon)<\underline{u}(x)/2$, whenever the
right-hand quantities are finite the maps
$s\mapsto\underline{P}^A(x,T,-su,\varepsilon)$ and
$s\mapsto\overline{P}^A(x,T,-su,\varepsilon)$ are strictly
decreasing, with slope at most
$-\underline{u}(x)/2$; likewise at scale $0$ with slope at most
$-\underline{u}(x)$.
\end{lemma}

\begin{proof}
Let $\varepsilon$ be admissible for $x$ and let $N$ be as in
Definition~\ref{def:admissible-scale}, so that
$S_nu(x_i)/n\geq\underline{u}(x)-2\Var(u,\varepsilon)$ for
all $n\geq N$ and all $x_i\in B_n(x,T,\varepsilon)$. Since
$s'-s>0$, the factorization
\[
  \frac{K^A(i) - s'\,S_n u(x_i)}{n}
  = \frac{K^A(i) - s\,S_n u(x_i)}{n} - (s'-s)\,\frac{S_n u(x_i)}{n}
\]
gives, for every such $i$ and $n\geq N$,
\[
  \frac{K^A(i) - s'\,S_n u(x_i)}{n}
  \;\leq\;
  \frac{K^A(i) - s\,S_n u(x_i)}{n}
  \;-\;(s'-s)\bigl(\underline{u}(x)-2\Var(u,\varepsilon)\bigr).
\]
Taking the minimum over $x_i\in B_n(x,T,\varepsilon)$ (the bound
applies to every realizer, in particular to a minimizer of the
right-hand side) and then $\liminf$ (resp.\ $\limsup$) as
$n\to\infty$ yields \eqref{eq:mono-lower}
(resp.\ \eqref{eq:mono-upper}).

Letting $\varepsilon\to 0$ through admissible scales,
$\Var(u,\varepsilon)\to 0$ by uniform continuity of $u$, and
both sides converge to their scale-$0$ counterparts, giving
\eqref{eq:mono-lower-scale0} and \eqref{eq:mono-upper-scale0}. The
strict-decrease clause follows from
\eqref{eq:mono-lower}--\eqref{eq:mono-upper-scale0} together with
$\Var(u,\varepsilon)<\underline{u}(x)/2$ (admissibility) and
$\underline{u}(x)>0$.
\end{proof}
\begin{lemma}[Negative pressure bounds the dimension]
\label{lem:neg-pressure-upper}
Suppose $0<\underline{u}(x)\leq\overline{u}(x)<\infty$ and let
$\varepsilon>0$ be admissible for $x$. Then
\begin{equation}\label{eq:upper-Bowen}
  \underline{P}^A(x,T,-su,\varepsilon) < 0
  \;\implies\;
  \dim^A_{BS}(x,T,u,\varepsilon)
  \;\leq\; s +
  \frac{\underline{P}^A(x,T,-su,\varepsilon)}
  {\overline{u}(x)+2\Var(u,\varepsilon)},
\end{equation}
\begin{equation}\label{eq:upper-packing}
  \overline{P}^A(x,T,-su,\varepsilon) < 0
  \;\implies\;
  \Dim^A_{BS}(x,T,u,\varepsilon)
  \;\leq\; s +
  \frac{\overline{P}^A(x,T,-su,\varepsilon)}
  {\overline{u}(x)+2\Var(u,\varepsilon)},
\end{equation}
and the same implications hold at scale $0$, with denominator
$\overline{u}(x)$. In particular,
\[
  \underline{P}^A\bigl(x,T,
  -\dim^A_{BS}(x,T,u,\varepsilon)\,u,\;\varepsilon\bigr)\geq 0
  \qquad\text{and}\qquad
  \overline{P}^A\bigl(x,T,
  -\Dim^A_{BS}(x,T,u,\varepsilon)\,u,\;\varepsilon\bigr)\geq 0,
\]
and likewise at scale $0$.
\end{lemma}

\begin{proof}
Let $\varepsilon>0$ be admissible for $x$, so that for all
$x_i \in B_n(x,T,\varepsilon)$ and $n$ sufficiently large,
\[
  0<\underline{u}(x)-2\Var(u,\varepsilon)
  \leq\frac{S_n u(x_i)}{n}
  \leq \overline{u}(x)+2\Var(u,\varepsilon)
\]
(Definition~\ref{def:admissible-scale}). Suppose
$\underline{P}^A(x,T,-su,\varepsilon)< 0$. Then there is a
subsequence $n_k\to\infty$ so that
\[\underline{P}^A(x,T,-su,\varepsilon)
=\lim\limits_{k\to\infty}\min\limits_{x_i \in B_{n_k}(x,T,\varepsilon)}
\frac{K^A(i) - s\,S_{n_k} u(x_i)}{n_k}\]
and for all $k$
\[\min\limits_{x_i \in B_{n_k}(x,T,\varepsilon)}
\frac{K^A(i) - s\,S_{n_k} u(x_i)}{n_k}<0.\]
Using the factorization
\[
  \frac{K^A(i) - s\,S_{n_k} u(x_i)}{n_k}
  = \left(\frac{K^A(i)}{S_{n_k} u(x_i)} - s\right)
  \frac{S_{n_k} u(x_i)}{n_k}
\]
and the two-sided bound above (the minimum on the left is negative, so its minimizer $x_i$ has negative bracket, and dividing by
$S_{n_k}u(x_i)/n_k\leq\overline{u}(x)+2\Var(u,\varepsilon)$
weakens a negative quantity), we get
\[\min\limits_{x_i \in B_{n_k}(x,T,\varepsilon)}
\left(\frac{K^A(i)}{S_{n_k} u(x_i)} - s\right)
<\frac{1}{\overline{u}(x)+2\Var(u,\varepsilon)}
\min\limits_{x_i \in B_{n_k}(x,T,\varepsilon)}
\frac{K^A(i) - s\,S_{n_k} u(x_i)}{n_k}<0.\]
Since the $\liminf$ over all $n$ is at most the limit along the
subsequence $n_k$,
\[\dim_{BS}^A(x,T,u,\varepsilon)
=\liminf\limits_{n\to\infty}\min\limits_{x_i \in B_n(x,T,\varepsilon)}
\frac{K^A(i)}{S_n u(x_i)}
\leq s+\frac{\underline{P}^A(x,T,-su,\varepsilon)}
{\overline{u}(x)+2\Var(u,\varepsilon)}<s,\]
which is the first implication. For the ``in particular'' clause,
suppose
$\underline{P}^A(x,T,-su,\varepsilon)<0$ at
$s=\dim_{BS}^A(x,T,u,\varepsilon)$. Then
\[\dim_{BS}^A(x,T,u,\varepsilon)
\leq \dim_{BS}^A(x,T,u,\varepsilon)
+\frac{\underline{P}^A(x,T,-su,\varepsilon)}
{\overline{u}(x)+2\Var(u,\varepsilon)}
<\dim_{BS}^A(x,T,u,\varepsilon),\]
a contradiction. The $\limsup$ statements follow by the same
argument along a subsequence realizing
$\overline{P}^A(x,T,-su,\varepsilon)$, and the scale-$0$ statements
by letting $\varepsilon\to 0$ through admissible scales, with
$\Var(u,\varepsilon)\to 0$.
\end{proof}

\begin{lemma}[Positive pressure bounds the dimension from below]
\label{lem:pos-pressure-lower}
Suppose $0<\underline{u}(x)\leq\overline{u}(x)<\infty$ and let
$\varepsilon>0$ be admissible for $x$. Then
\[
  \underline{P}^A(x,T,-su,\varepsilon) > 0
  \;\implies\;
  \dim^A_{BS}(x,T,u,\varepsilon)
  \;\geq\; s+
  \frac{\underline{P}^A(x,T,-su,\varepsilon)}
  {\overline{u}(x)+2\Var(u,\varepsilon)},
\]
\[
  \overline{P}^A(x,T,-su,\varepsilon) > 0
  \;\implies\;
  \Dim^A_{BS}(x,T,u,\varepsilon)
  \;\geq\; s+
  \frac{\overline{P}^A(x,T,-su,\varepsilon)}
  {\overline{u}(x)+2\Var(u,\varepsilon)},
\]
and the same implications hold at scale $0$, with denominator
$\overline{u}(x)$. In particular,
\[
  \underline{P}^A\bigl(x,T,
  -\dim^A_{BS}(x,T,u,\varepsilon)\,u,\;\varepsilon\bigr)\leq 0
  \qquad\text{and}\qquad
  \overline{P}^A\bigl(x,T,
  -\Dim^A_{BS}(x,T,u,\varepsilon)\,u,\;\varepsilon\bigr)\leq 0,
\]
and likewise at scale $0$.
\end{lemma}

\begin{proof}
We do the $\limsup$ inequality first, as it requires choosing a
subsequence. Suppose $\overline{P}^A(x,T,-su,\varepsilon)>0$ and take
a subsequence $n_k\uparrow\infty$ so that
$\min\limits_{x_i \in B_{n_k}(x,T,\varepsilon)}
\frac{K^A(i) - s\,S_{n_k} u(x_i)}{n_k}>0$ for all $k$ and
\[\lim\limits_{k\to\infty}\min\limits_{x_i \in B_{n_k}(x,T,\varepsilon)}
\frac{K^A(i) - s\,S_{n_k} u(x_i)}{n_k}
=\overline{P}^A(x,T,-su,\varepsilon).\]
Fix $k$ large enough that the admissible-scale bounds
(Definition~\ref{def:admissible-scale}) hold at $n_k$. Since each
term of the minimum is the product
$\bigl(\frac{K^A(i)}{S_{n_k}u(x_i)}-s\bigr)\frac{S_{n_k}u(x_i)}{n_k}$
with positive second factor, and the minimum is positive, every
bracket is positive; evaluating at a bracket-minimizing $i$ and
using
$\frac{S_{n_k}u(x_i)}{n_k}\leq\overline{u}(x)
+2\Var(u,\varepsilon)$,
\[
  \min\limits_{x_i \in B_{n_k}(x,T,\varepsilon)}
  \frac{K^A(i) - s\,S_{n_k} u(x_i)}{n_k}
  \;\leq\;
  \bigl(\overline{u}(x)+2\Var(u,\varepsilon)\bigr)
  \min\limits_{x_i \in B_{n_k}(x,T,\varepsilon)}
  \left(\frac{K^A(i)}{S_{n_k} u(x_i)} - s\right).
\]
Taking $k\to\infty$, with $\limsup$ on the right,
\[\overline{P}^A(x,T,-su,\varepsilon)
\leq \bigl(\overline{u}(x)+2\Var(u,\varepsilon)\bigr)
\left(\limsup\limits_{n\to\infty}
\min\limits_{x_i \in B_n(x,T,\varepsilon)}
\frac{K^A(i)}{S_{n} u(x_i)}-s\right)
=\bigl(\overline{u}(x)+2\Var(u,\varepsilon)\bigr)
\bigl(\Dim^A_{BS}(x,T,u,\varepsilon)-s\bigr).\]
The $\liminf$ case is easier: if
$\underline{P}^A(x,T,-su,\varepsilon)>0$ then
$\min\limits_{x_i \in B_n(x,T,\varepsilon)}
\frac{K^A(i) - s\,S_{n} u(x_i)}{n}>0$ for all $n$ large enough, so
the same per-$n$ inequality holds at every large $n$, and taking
$\liminf$ on both sides gives
\[\underline{P}^A(x,T,-su,\varepsilon)
\leq \bigl(\overline{u}(x)+2\Var(u,\varepsilon)\bigr)
\bigl(\dim^A_{BS}(x,T,u,\varepsilon)-s\bigr).\]
Rearranging both inequalities yields the two implications. The
``in particular'' clause follows by substituting
$s=\dim^A_{BS}(x,T,u,\varepsilon)$
(resp.\ $s=\Dim^A_{BS}(x,T,u,\varepsilon)$): a positive pressure
there would force the dimension to exceed itself. The scale-$0$
statements follow by letting $\varepsilon\to 0$ through admissible
scales, with $\Var(u,\varepsilon)\to 0$.
\end{proof}

\begin{theorem}[Pointwise Bowen pressure equation]\label{thm:pointwise-bowen}
    Let $x \in X$ with $0 < \underline{u}(x) \leq \overline{u}(x) < \infty$, let $A$ be any oracle, and suppose $\varepsilon > 0$ is admissible for $x$. Then
    \begin{itemize}
        \item If $\underline{K}^A(x,T,\varepsilon)<\infty$ then $\dim^A_{\mathrm{BS}}(x,T,u,\varepsilon)$ is the unique solution to $\underline{P}^A(x,T,-su,\varepsilon)=0$;
        \item If $\overline{K}^A(x,T,\varepsilon)<\infty$ then $\Dim^A_{\mathrm{BS}}(x,T,u,\varepsilon)$ is the unique solution to $\overline{P}^A(x,T,-su,\varepsilon)=0$;
        \item If $\underline{K}^A(x,T)<\infty$ then $\dim^A_{\mathrm{BS}}(x,T,u)$ is the unique solution to $\underline{P}^A(x,T,-su)=0$;
        \item If $\overline{K}^A(x,T)<\infty$ then $\Dim^A_{\mathrm{BS}}(x,T,u)$ is the unique solution to $\overline{P}^A(x,T,-su)=0$.
    \end{itemize}
\end{theorem}
\begin{remark}\label{rem:infinite-complexity}
    The hypothesis $\underline{K}^A(x,T,\varepsilon) < \infty,\ \overline{K}^A(x,T,\varepsilon) < \infty,\ \underline{K}^A(x,T) < \infty,\ \overline{K}^A(x,T) < \infty$ excludes only a trivial case. If, say, $\underline{K}^A(x,T) = \infty$, then $\dim^A_{\mathrm{BS}}(x,T,u) =\infty$, while by \eqref{eq:pressure-complexity-sandwich} every pressure $\underline{P}^A(x,T,-su)$ is infinite, so no root exists. Bowen's equation survives only in its critical-value form $\dim^A_{\mathrm{BS}}(x,T,u) = \inf\{s : \underline{P}^A(x,T,-su) \leq 0\} = \inf \emptyset = \infty$. Note that in the compact setting we always have $\underline{K}^A(x,T,\varepsilon)<\infty,\ \overline{K}^A(x,T,\varepsilon) < \infty$ by bounding by $\log(K)$ where $K$ is the size of any $\varepsilon$-net, and as such the finite scale pointwise Bowen pressure equation is always valid in the compact setting. 
\end{remark}
\begin{proof}
    We first observe that assuming finite of the respective orbit complexity the associated pressures are finite for every $s \geq 0$. For example in the case of $\underline{K}^A(x,T,\varepsilon)<\infty$ factoring the potential sum as in Lemma \ref{lem:monotone_decreasing}, for any admissible $\varepsilon$ and any $s \geq 0$,
    \begin{equation}\label{eq:pressure-complexity-sandwich}
        \big|\,\underline{P}^A(x,T,-su,\varepsilon) - \underline{K}^A(x,T,\varepsilon)\,\big| \leq s\big(\overline{u}(x) + 2\Var(u,\varepsilon)\big),
    \end{equation}
    and likewise for the other quantities. A similar inequality shows the associated BS dimensions are finite. 
    
    \emph{Existence.} By the ``in particular'' clauses of Lemmas \ref{lem:pos-pressure-lower} and \ref{lem:neg-pressure-upper},
    \[
        \underline{P}^A\big(x,T,-\dim^A_{\mathrm{BS}}(x,T,u,\varepsilon)\,u,\varepsilon\big)=0,
        \qquad
        \overline{P}^A\big(x,T,-\Dim^A_{\mathrm{BS}}(x,T,u,\varepsilon)\,u,\varepsilon\big)=0,
    \]
    and at scale $0$,
    \[
        \underline{P}^A\big(x,T,-\dim^A_{\mathrm{BS}}(x,T,u)\,u\big)=0,
        \qquad
        \overline{P}^A\big(x,T,-\Dim^A_{\mathrm{BS}}(x,T,u)\,u\big)=0.
    \]

    \emph{Uniqueness.} By Lemma \ref{lem:monotone_decreasing} the pressures are strictly monotone decreasing in $s$, with slope at most $-\underline{u}(x) + 2\Var(u,\varepsilon) < 0$ for admissible $\varepsilon$, so each has at most one root.
\end{proof}
The general situation is the following.
\begin{corollary}[Critical-value characterizations]
\label{cor:crit-char}
Suppose $0<\underline{u}(x)\leq\overline{u}(x)<\infty$ and let
$\varepsilon>0$ be admissible for $x$. Then
\[
  \dim^A_{BS}(x,T,u,\varepsilon)
  = \sup\{s \mid \underline{P}^A(x,T,-su,\varepsilon) > 0\}
  = \inf\{s \mid \underline{P}^A(x,T,-su,\varepsilon) \leq 0\},
\]
\[
  \Dim^A_{BS}(x,T,u,\varepsilon)
  = \sup\{s \mid \overline{P}^A(x,T,-su,\varepsilon) > 0\}
  = \inf\{s \mid \overline{P}^A(x,T,-su,\varepsilon) \leq 0\},
\]
and the same identities hold at scale $0$.
\end{corollary}
\begin{proof}
Suppose first that the relevant orbit complexity is finite
($\underline{K}^A(x,T,\varepsilon)<\infty$ for the first pair of
identities, and correspondingly in the other three cases). By
Theorem~\ref{thm:pointwise-bowen},
$\dim^A_{BS}(x,T,u,\varepsilon)$ is the unique zero of the map
$s\mapsto\underline{P}^A(x,T,-su,\varepsilon)$, which is strictly
decreasing by Lemma~\ref{lem:monotone_decreasing}. Hence the
pressure is positive for $s<\dim^A_{BS}(x,T,u,\varepsilon)$ and
negative for $s>\dim^A_{BS}(x,T,u,\varepsilon)$, and both
identities follow.

If instead the orbit complexity is infinite then, as in
Remark~\ref{rem:infinite-complexity}, the pointwise BS dimension is
infinite and $\underline{P}^A(x,T,-su,\varepsilon)=\infty$ for
every $s$. The first defining set is then all of $\mathbb{R}$ and
the second is empty, so with the convention
$\inf\emptyset=+\infty$,
\[
\dim^A_{BS}(x,T,u,\varepsilon)
  = \sup\{s \mid \underline{P}^A(x,T,-su,\varepsilon) > 0\}
  = \inf\{s \mid \underline{P}^A(x,T,-su,\varepsilon) \leq 0\}
  = \infty.
\]
The upper and scale-$0$ cases are identical.
\end{proof}
\subsection{Point-to-set principles for the Bowen pressure equation}
\label{subsec:bowen-pts}
Having established the pointwise Bowen pressure equation, we upgrade
it to the standard set-level version via the point-to-set principles
of Section~\ref{sec_point_to_set}. The passage requires monotonicity in $s$
of the set-level quantities, which we obtain from the pointwise
monotonicity of Lemma~\ref{lem:monotone_decreasing} by taking
suprema. Note that strictness survives this passage only under a
uniform lower bound $\underline{u}\geq\alpha>0$ on $Z$. Without it,
near-suprema may be attained at points with $\underline{u}$
arbitrarily small, where the pointwise decrease rate degenerates.

\begin{corollary}[Monotonicity of the supremum]
\label{cor:sup-monotone}
Suppose $0<\underline{u}(x)\leq\overline{u}(x)<\infty$ for all
$x\in Z$, and let $s<s'$. Then
\[
\sup_{x\in Z}\underline{P}^A(x,T,-s'u)
\leq \sup_{x\in Z}\underline{P}^A(x,T,-su),
\qquad
\sup_{x\in Z}\overline{P}^A(x,T,-s'u)
\leq \sup_{x\in Z}\overline{P}^A(x,T,-su).
\]
If moreover $\underline{u}(x)\geq \alpha>0$ for all $x\in Z$ then
\[
\sup_{x\in Z}\underline{P}^A(x,T,-s'u)
\leq \sup_{x\in Z}\underline{P}^A(x,T,-su)-(s'-s)\,\alpha,
\qquad
\sup_{x\in Z}\overline{P}^A(x,T,-s'u)
\leq \sup_{x\in Z}\overline{P}^A(x,T,-su)-(s'-s)\,\alpha.
\]
\end{corollary}
\begin{proof}
By \eqref{eq:mono-lower-scale0}--\eqref{eq:mono-upper-scale0}, for
every $x\in Z$,
$\underline{P}^A(x,T,-s'u)\leq
\underline{P}^A(x,T,-su)-(s'-s)\,\underline{u}(x)
\leq\underline{P}^A(x,T,-su)$, and taking suprema over $Z$ gives
the first pair. Under the uniform bound, the middle term is at most
$\underline{P}^A(x,T,-su)-(s'-s)\,\alpha$, and the supremum of a
uniform translate is the translated supremum. The upper case is
identical.
\end{proof}

We now transfer this to the classical pressure.

\begin{lemma}[Monotonicity of the set-level pressure]
\label{lem:set-monotone}
Suppose $0<\underline{u}(x)\leq\overline{u}(x)<\infty$ for all
$x\in Z$, and let $s<s'$. Then
\[
\underline{P}(Z,T,-s'u)\leq \underline{P}(Z,T,-su),
\qquad
\overline{P}(Z,T,-s'u)\leq \overline{P}(Z,T,-su).
\]
If moreover $\underline{h}(Z,T)<\infty$ (respectively $\overline{h}(Z,T)<\infty$) and $\underline{u}(x)\geq \alpha>0$ for all $x\in Z$, then
\[
\underline{P}(Z,T,-s'u)\leq \underline{P}(Z,T,-su)-(s'-s)\,\alpha,
\qquad
\overline{P}(Z,T,-s'u)\leq \overline{P}(Z,T,-su)-(s'-s)\,\alpha.
\]
\end{lemma}
\begin{proof}
By the point-to-set principles for lower/upper pressure
(Theorems~\ref{thm:pts_lower_pressure}
and~\ref{thm:pts_upper_pressure} respectively) and
Corollary~\ref{cor:sup-monotone}, for every oracle $A$,
\[
\underline{P}(Z,T,-s'u)
\leq\sup_{x\in Z}\underline{P}^A(x,T,-s'u)
\leq \sup_{x\in Z}\underline{P}^A(x,T,-su),
\]
and taking the minimum over $A$ on the right gives
$\underline{P}(Z,T,-s'u)\leq\underline{P}(Z,T,-su)$. Under the
uniform bound the middle inequality improves by $(s'-s)\alpha$, and
the same minimization gives the second inequality. The upper case is
identical.

For the ``moreover'' clause, strictness gives
$\underline{P}(Z,T,-s'u)<\underline{P}(Z,T,-su)$ whenever $s<s'$, so
$s\mapsto\underline{P}(Z,T,-su)$ is injective; in particular if
$\underline{P}(Z,T,-Du)>0$ at some $D$, then every $s\leq D$ has
$\underline{P}(Z,T,-su)>0$, and the critical parameter exceeds $D$.
\end{proof}
We next express the classical quantities in the Bowen pressure equation as oracle minimizations of their pointwise counterparts.
\begin{lemma}\label{lem:crit-param-min}
Suppose $0<\underline{u}(x)\leq\overline{u}(x)<\infty$ for every $x\in Z$. Then
\begin{equation}\label{eq:crit-sup-lower}
\sup\{s\mid\underline{P}(Z,T,-su)> 0\}=\min\limits_{A\subset \mathbb{N}}\sup\{s\mid \sup\limits_{x\in Z}\underline{P}^A(x,T,-su)> 0\}
\end{equation}
\begin{equation}\label{eq:crit-inf-lower}
\inf\{s\mid \underline{P}(Z,T,-su)\leq 0\}=\min\limits_{A\subset \mathbb{N}}\inf\{s\mid \sup\limits_{x\in Z}\underline{P}^A(x,T,-su)\leq 0\}
\end{equation}
\begin{equation}\label{eq:crit-sup-upper}
\sup\{s\mid\overline{P}(Z,T,-su)> 0\}=\min\limits_{A\subset \mathbb{N}}\sup\{s\mid \sup\limits_{x\in Z}\overline{P}^A(x,T,-su)> 0\}
\end{equation}
\begin{equation}\label{eq:crit-inf-upper}
\inf\{s\mid \overline{P}(Z,T,-su)\leq 0\}=\min\limits_{A\subset \mathbb{N}}\inf\{s\mid \sup\limits_{x\in Z}\overline{P}^A(x,T,-su)\leq 0\}
\end{equation}
\end{lemma}

\begin{proof}
We prove the lower identities \eqref{eq:crit-sup-lower}
and~\eqref{eq:crit-inf-lower}; the upper ones follow symmetrically.

For every oracle $A$ we have
$\underline{P}(Z,T,-su)\leq\sup_{x\in Z}\underline{P}^A(x,T,-su)$
(Theorem~\ref{thm:pts_lower_pressure}), whence the inclusions
\[
    \{s\mid \underline{P}(Z,T,-su)>0\}
    \subset
    \{s\mid \sup_{x\in Z}\underline{P}^A(x,T,-su)>0\},
    \qquad
    \{s\mid \sup_{x\in Z}\underline{P}^A(x,T,-su)\leq 0\}
    \subset
    \{s\mid \underline{P}(Z,T,-su)\leq 0\},
\]
and therefore, for every $A$,
\begin{equation}\label{eq:one-sided}
    \sup\{s\mid \underline{P}(Z,T,-su)>0\}
    \leq \sup\{s\mid \sup_{x\in Z}\underline{P}^A(x,T,-su)>0\},
    \qquad
    \inf\{s\mid \underline{P}(Z,T,-su)\leq 0\}
    \leq \inf\{s\mid \sup_{x\in Z}\underline{P}^A(x,T,-su)\leq 0\}.
\end{equation}
This is one direction of both identities; it remains to construct a
single oracle achieving equality in both.

Let $\{s_i\}_{i\in\mathbb{N}}$ be a countable dense subset of
$\mathbb{R}$. By Theorem~\ref{thm:pts_lower_pressure}, for each $i$
there is an oracle $A_i$ with
$\underline{P}(Z,T,-s_i u)
=\sup_{x\in Z}\underline{P}^{A_i}(x,T,-s_i u)$. Let
$A=\bigoplus_i A_i$. For each $i$, monotonicity in the oracle
(Lemma~\ref{lem:monotone}) and the bound above give a two sided bound on the join:
\begin{equation}\label{eq:dense-agree}
    \underline{P}(Z,T,-s_i u)
    \leq\sup_{x\in Z}\underline{P}^{A}(x,T,-s_i u)
    \leq\sup_{x\in Z}\underline{P}^{A_i}(x,T,-s_i u)
    =\underline{P}(Z,T,-s_i u).
\end{equation}
Thus the set-level and $A$-relative quantities agree on the dense
set $\{s_i\}$.

Now suppose equality fails in the sup identity, so that by
\eqref{eq:one-sided} the inequality is strict; by density choose
$s_i$ strictly between the two values. Then
$\underline{P}(Z,T,-s_i u)\leq 0$ while
$\sup_{x\in Z}\underline{P}^A(x,T,-s_i u)>0$, contradicting
\eqref{eq:dense-agree}; here we use that a value $>0$ below the
right-hand supremum forces positivity at $s_i$ itself by monotonicity
in $s$ (Corollary~\ref{cor:sup-monotone}). Hence $A$ achieves equality in the sup
identity, and the same argument with the roles of $>0$ and $\leq 0$
exchanged gives the inf identity. Finally, any $B\geq_T A$
satisfies \eqref{eq:dense-agree} with $A$ replaced by $B$, by the
same upper and lower bound, so the argument applies verbatim to $B$: the minima
are attained on the upper cone above $A$.
\end{proof}
\begin{lemma}[The supremum of dimensions is the critical parameter]
\label{lem:sup-dim-crit}
Suppose $0<\underline{u}(x)\leq\overline{u}(x)<\infty$ for every
$x\in Z$. Then
\[
\sup_{x\in Z}\dim_{BS}^A(x,T,u)
=\sup\{s\mid \sup_{x\in Z}\underline{P}^A(x,T,-su)> 0\}
=\inf\{s\mid \sup_{x\in Z}\underline{P}^A(x,T,-su)\leq 0\},
\]
\[
\sup_{x\in Z}\Dim_{BS}^A(x,T,u)
=\sup\{s\mid \sup_{x\in Z}\overline{P}^A(x,T,-su)> 0\}
=\inf\{s\mid \sup_{x\in Z}\overline{P}^A(x,T,-su)\leq 0\}.
\]
Moreover $\sup_{x\in Z}\underline{P}^A(x,T,-\underline{s}_0u)\leq 0$ at
$\underline{s}_0=\sup_{x\in Z}\dim^A_{BS}(x,T,u)$, and likewise in the upper
case.
\end{lemma}
\begin{proof}
Write $F(s)=\sup_{x\in Z}\underline{P}^A(x,T,-su)$ and
$D=\sup_{x\in Z}\dim^A_{BS}(x,T,u)$. By monotonicity of $F$
(Corollary~\ref{cor:sup-monotone}) the sets $\{s\mid F(s)>0\}$ and
$\{s\mid F(s)\leq 0\}$ are complementary rays, so
\[\sup\{s\mid F(s)> 0\}= \inf\{s\mid F(s)\leq 0\}.\]
If $F(s)>0$, some $x\in Z$ has $\underline{P}^A(x,T,-su)>0$, and the
sign pattern of the pointwise Bowen equation
(Theorem~\ref{thm:pointwise-bowen}) gives
$\dim^A_{BS}(x,T,u)> s$, so $D\geq s$; hence
$\sup\{s\mid F(s)>0\}\leq D$. Conversely, if $F(s)\leq 0$ then
$\underline{P}^A(x,T,-su)\leq 0$ for every $x\in Z$, and the sign
pattern gives $\dim^A_{BS}(x,T,u)\leq s$ for every $x$, so
$D\leq s$; hence $D\leq\inf\{s\mid F(s)\leq 0\}$. Combining the
three displays proves the lower identity; the upper one is
identical.
\end{proof}
\begin{makethm}{theorem}{thm:BPE}
Suppose $0<\underline{u}(x)\leq \overline{u}(x)<\infty$ for every
$x\in Z$. Then
\[
\dim_{BS}(Z,T,u)
=\sup\{s\mid\underline{P}(Z,T,-su)> 0\}
=\inf\{s\mid \underline{P}(Z,T,-su)\leq 0\},
\]
\[
\Dim_{BS}(Z,T,u)
=\sup\{s\mid\overline{P}(Z,T,-su)> 0\}
=\inf\{s\mid \overline{P}(Z,T,-su)\leq 0\}.
\]
\end{makethm}
\begin{proof}
By the point-to-set principle for lower BS dimension
(Theorem~\ref{thm:pts_lower_bs}) and
Lemma~\ref{lem:sup-dim-crit}, applied at each oracle,
\[
\dim_{BS}(Z,T,u)
=\min\limits_{A\subset \mathbb{N}}\sup_{x\in Z}\dim_{BS}^A(x,T,u)
=\min\limits_{A\subset \mathbb{N}}
\sup\{s\mid \sup_{x\in Z}\underline{P}^A(x,T,-su)> 0\}
=\min\limits_{A\subset \mathbb{N}}
\inf\{s\mid \sup_{x\in Z}\underline{P}^A(x,T,-su)\leq 0\},
\]
and by Lemma~\ref{lem:crit-param-min} the last two minima equal
$\sup\{s\mid\underline{P}(Z,T,-su)> 0\}$ and
$\inf\{s\mid \underline{P}(Z,T,-su)\leq 0\}$ respectively, giving
the lower identities. The upper identities follow identically,
using Theorem~\ref{thm:pts_upper}, the upper half of
Lemma~\ref{lem:sup-dim-crit}, and the upper half of
Lemma~\ref{lem:crit-param-min}.
\end{proof}

We now consider existence and uniqueness of solutions to the
set-level Bowen pressure equations
\[\underline{P}(Z,T,-su)=0,
\qquad \overline{P}(Z,T,-su)=0.\]
\label{subsec:existence-uniqueness}
\begin{makethm}{theorem}{thm:BPE_unique}
Assume $\underline{h}(Z,T)<\infty$ (respectively $\overline{h}(Z,T)<\infty$). Consider the lower and upper Bowen pressure equations
$\underline{P}(Z,T,-su)=0$ and $\overline{P}(Z,T,-su)=0$.

\hangindent=2em\hangafter=1\noindent
\emph{(Existence)}\enspace If
$0<\underline{u}(x)\leq\overline{u}(x)\leq\beta<\infty$ for all
$x\in Z$, then the equations have solutions, given by
$s_0=\dim_{BS}(Z,T,u)$ and $s_0=\Dim_{BS}(Z,T,u)$ respectively.

\hangindent=2em\hangafter=1\noindent
\emph{(Uniqueness)}\enspace If
$0<\alpha\leq\underline{u}(x)\leq\overline{u}(x)<\infty$ for all
$x\in Z$, then the maps $s\mapsto\underline{P}(Z,T,-su)$ and
$s\mapsto\overline{P}(Z,T,-su)$ are strictly decreasing; in particular each equation has at most one solution.

\noindent In particular, under
$0<\alpha\leq\underline{u}\leq\overline{u}\leq\beta<\infty$ on $Z$,
the equations have unique solutions, given by the lower and upper
BS dimensions respectively. 
\end{makethm}
\begin{remark}
    Note that finite BS dimension under the assumption of $0<\alpha\leq\underline{u}\leq\overline{u}\leq\beta<\infty$ on $Z$ is equivalent to finite entropy, as well as finiteness of pressures of $-su$.
\end{remark}
\begin{proof}
\emph{(i) Existence.} Suppose
$0<\underline{u}\leq\overline{u}\leq\beta$ on $Z$. Set
$s_0:=\dim_{BS}(Z,T,u)$, and let $A$ be an oracle simultaneously
achieving the point-to-set principles for the lower BS dimension
(Theorem~\ref{thm:pts_lower_bs}) and for the lower pressure at the
potential $-s_0u$ (Theorem~\ref{thm:pts_lower_pressure}); such $A$
exists by joining the two achieving oracles (Lemma~\ref{lem:join-cone}),
each principle being preserved on the upper cone. Then
$s_0=\sup_{x\in Z}\dim^A_{BS}(x,T,u)$ and
$\underline{P}(Z,T,-s_0u)=\sup_{x\in Z}\underline{P}^A(x,T,-s_0u)$.
Since $\dim^A_{BS}(x,T,u)\leq s_0$ for every $x\in Z$, the pointwise
sign pattern (Theorem~\ref{thm:pointwise-bowen}) gives
$\underline{P}^A(x,T,-s_0u)\leq 0$ pointwise, hence
$\sup_{x\in Z}\underline{P}^A(x,T,-s_0u)\leq 0$; the content is the
reverse inequality. Given $\eta>0$, choose $x_\eta\in Z$ with
$\dim_{BS}^A(x_\eta,T,u)>s_0-\eta$; if
$\underline{P}^A(x_\eta,T,-s_0u)<0$ then
Lemma~\ref{lem:neg-pressure-upper} (at scale $0$) gives
\[\underline{P}^A(x_\eta,T,-s_0u)
\geq \overline{u}(x_\eta)\bigl(\dim_{BS}^A(x_\eta,T,u)-s_0\bigr)
\geq -\eta\,\beta.\]
Letting $\eta\to 0$ yields
$\sup_{x\in Z}\underline{P}^A(x,T,-s_0u)=0$. We conclude
\[\underline{P}(Z,T,-s_0u)
=\sup_{x\in Z}\underline{P}^A(x,T,-s_0u)=0,\]
so $s_0$ solves the lower Bowen pressure equation.

\emph{(ii) Uniqueness.} Suppose $\underline{u}\geq\alpha>0$ on $Z$.
By Lemma~\ref{lem:set-monotone}, the map
$s\mapsto\underline{P}(Z,T,-su)$ is strictly decreasing with gap
$\alpha$, and a strictly decreasing map has at most one zero.
\end{proof}

\begin{corollary}[Constant exponent]\label{cor:constant-exponent}
If $\underline{u}(x)=\overline{u}(x)=\alpha>0$ for all $x\in Z$,
then
\[\dim_{BS}(Z,T,u)=\frac{\underline{h}(Z,T)}{\alpha},
\qquad
\Dim_{BS}(Z,T,u)=\frac{\overline{h}(Z,T)}{\alpha}.\]
\end{corollary}
\begin{proof}
Since $\dim_{BS}(Z,T,u)=\infty\iff \underline{h}(Z,T)=\infty,\ \Dim_{BS}(Z,T,u)=\infty \iff \overline{h}(Z,T)=\infty$ under the assumption of uniform upper and lower bounds on the $\underline{u}/\overline{u}$ this is trivially true in either of those cases. We therefore work in the finite entropy setting.

We prove the lower identity; the upper one is identical. First, for
any oracle $A$ and any $x$ with
$\underline{u}(x)=\overline{u}(x)=\alpha$,
\begin{equation}\label{eq:const-exp-split}
\underline{P}^A(x,T,-su)=\underline{K}^A(x,T)-s\alpha:
\end{equation}
at any admissible scale $\varepsilon$, every realizer satisfies
$|S_nu(x_i)/n-\alpha|\leq\Var(u,\varepsilon)$ for large $n$
(Definition~\ref{def:admissible-scale}), so the minimized quantity
splits as
$\min\limits_i\frac{K^A(i)}{n}-s\alpha\pm|s|\,\Var(u,\varepsilon)$,
and letting $\varepsilon\to 0$ gives
\eqref{eq:const-exp-split}.

Now write $s_0=\dim_{BS}(Z,T,u)$ and let $A$ simultaneously achieve
the point-to-set principles for $\dim_{BS}(Z,T,u)$ and for
$\underline{h}(Z,T)$
(Theorems~\ref{thm:pts_lower_bs} and~\ref{thm:pts-entropy}, joined
via Lemma~\ref{lem:join-cone}). By the pointwise sign pattern
(Theorem~\ref{thm:pointwise-bowen}), for every $x\in Z$ we have
$\underline{P}^A(x,T,-s_0u)\leq 0$, which by
\eqref{eq:const-exp-split} reads
$\underline{K}^A(x,T)\leq s_0\alpha$; taking the supremum,
$\underline{h}(Z,T)\leq s_0\alpha$. Conversely, for $\eta>0$ choose
$x_\eta\in Z$ with $\dim^A_{BS}(x_\eta,T,u)>s_0-\eta$; then
$\underline{P}^A(x_\eta,T,-(s_0-\eta)u)\geq 0$ by the sign pattern,
so \eqref{eq:const-exp-split} gives
$\underline{K}^A(x_\eta,T)\geq(s_0-\eta)\alpha$ and hence
$\underline{h}(Z,T)\geq(s_0-\eta)\alpha$. Letting $\eta\to 0$
yields $\underline{h}(Z,T)=s_0\alpha$.
\end{proof}

\subsection{Hausdorff and packing dimension for conformal systems}
\label{subsec:conformal}
In this section we recover and extend the main results of
\cite{climenhaga-2010} on the Bowen pressure equation using the
point-to-set machinery. Our treatment departs from
\cite{climenhaga-2010} in three ways: we treat Hausdorff and
packing dimension simultaneously, obtaining the packing-side
statements, which do not appear in the literature, by the same
argument; we work with complete separable metric spaces, replacing
compactness by pointwise dynamical hypotheses on the set under study
(the sets $\mathcal{A}$, $\mathcal{B}$, $\mathcal{C}$ of Definition~\ref{def:conformal-sets}
below); and we phrase the results in terms of the lower and upper BS
dimensions, passing to the Bowen pressure equation through the
results of Section~\ref{subsec:existence-uniqueness}. The
relationship between BS dimension and the Bowen equation is
discussed in \cite{CLIMENHAGA2014The}.
\begin{definition}[Conformal system]\label{def:conformal}
A dynamical system $(X,d,T)$, with $X$ a complete separable metric
space and $T$ uniformly continuous, is \emph{conformal} if the
derivative
\[a(x)=\lim\limits_{y\to x}\frac{d(Tx,Ty)}{d(x,y)}\]
exists in $(0,\infty)$ for every $x\in X$ and the function
\[\zeta(x,y) = \begin{cases}
\log d(Tx, Ty) - \log d(x,y) & x \neq y, \\
\log a(x) & x = y,
\end{cases}\]
is uniformly continuous on $\{(x,y)\mid d(x,y)<r_0\}$ for some
$r_0>0$. We write $u=\log a$. Note that no uniform bounds on $a$ are
assumed; the sets of Definition~\ref{def:conformal-sets} play that
role.
\end{definition}

\begin{definition}[The sets $\mathcal{A}$, $\mathcal{B}$,
$\mathcal{C}$]\label{def:conformal-sets}
For $0\leq\alpha_1\leq\alpha_2\leq\infty$, the set of points with
exponents in $(\alpha_1,\alpha_2)$ is
\[\mathcal{A}(\alpha_1,\alpha_2)
=\{x\mid \alpha_1<\underline{u}(x)\leq
\overline{u}(x)<\alpha_2\}.\]
The set of points with \emph{tempered contraction} is
\[
\mathcal{B}=\Bigl\{x\;\Bigm|\;
\inf_{\substack{n \in \mathbb{N} \\ 0 \leq k \leq n}}
\bigl( S_{n-k} u(T^kx) + n \varepsilon \bigr) > -\infty
\text{ for every } \varepsilon > 0\Bigr\},
\]
and the set of points with \emph{sublinear growth} is
\[\mathcal{C}=\Bigl\{x\;\Bigm|\;
\lim\limits_{n\to\infty}\frac{u(T^nx)}{n}=0\Bigr\}.\]
In the compact setting $u$ is bounded, so $\mathcal{C}=X$.
\end{definition}

For $0<\delta<r_0$ set
\[\Var(\zeta,\delta)
=\sup\{\,|\zeta(x,y)-u(x)| : 0<d(x,y)\leq\delta\,\},\]
which tends to $0$ as $\delta\to 0$ by uniform conformality. This is
the modulus of the derivative, distinct from the modulus
$\Var(u,\cdot)$ of the potential used in
Section~\ref{sec:bowen-equation}. Throughout this subsection we
abbreviate $\varepsilon_\delta=\Var(\zeta,\delta)$.
\begin{lemma}[Ball comparison; {cf.\ \cite[Lemma 6.1]{climenhaga-2010}}]
\label{lem:ball-inclusion}
Let $(X,d,T)$ be conformal, $x\in\mathcal{B}$, and
$0<\delta<r_0$. Then with
$\eta_\delta(x)=\exp_2\left(\inf_{n,\,0\leq k\leq n}
(S_{n-k}u(T^kx)+n\varepsilon_\delta)\right)>0$, for every $n$,
 \[B\bigl(x,\eta_\delta(x)\,\delta\,
2^{-(S_nu(x)+n\varepsilon_\delta)}\bigr)
\subset B_n(x,T,\delta)\subset
B\bigl(x,\delta\,2^{-(S_nu(x)-n\varepsilon_\delta)}\bigr);\]
\end{lemma}
\begin{proof}
This is \cite[Lemma 6.1]{climenhaga-2010}, whose proof
invokes only the uniform continuity of $\zeta$ near the diagonal
and the tempered contraction condition defining $\mathcal{B}$, and
so applies verbatim in the complete separable setting (with base
$2$ in place of $e$).
\end{proof}
We now prove that on
$\mathcal{A}(0,\infty)\cap\mathcal{B}\cap\mathcal{C}$, the pointwise
Hausdorff and packing dimensions are computed by the corresponding
BS dimensions for the potential $u=\log a$.
\begin{theorem}\label{thm:pointwise-conformal}
    Suppose $x \in \mathcal{A}(0,\infty) \cap \mathcal{B}\cap \mathcal{C}$. Then for any oracle $A$,
    \[
        \dim^A(x) = \dim^A_{BS}(x,T,u),
        \quad
        \Dim^A(x) = \Dim^A_{BS}(x,T,u).
    \]
\end{theorem}
\begin{proof}
Fix $0<\delta<r_0$ small enough that
$\varepsilon_\delta<\underline{u}(x)$, and write
$\varepsilon=\varepsilon_\delta$, $\eta=\eta_\delta(x)$. This is
possible since $\varepsilon_\delta\to 0$ as $\delta\to 0$ and
$\underline{u}(x)>0$. It guarantees $S_nu(x)-n\varepsilon\to\infty$,
so both scale families below are eventually decreasing to $0$. Set
\[\rho_n=\eta\,\delta\,2^{-(S_nu(x)+n\varepsilon)},
\qquad
\rho_n'=\delta\,2^{-(S_nu(x)-n\varepsilon)}.\]
Then we have 
\[\log\rho_{n+1}/\log\rho_n=\frac{S_{n+1}u(x)+n\varepsilon+O(1)}{S_{n}u(x)+n\varepsilon+O(1)}=1+\frac{u(T^{n}x)+O(1)}{S_{n}u(x)+n\varepsilon+O(1)}=1+o(1)\]
\[\log\rho'_{n+1}/\log\rho'_n=\frac{S_{n+1}u(x)-n\varepsilon+O(1)}{S_{n}u(x)-n\varepsilon+O(1)}=1+\frac{u(T^{n}x)+O(1)}{S_{n}u(x)-n\varepsilon+O(1)}=1+o(1)\]
since $x\in\mathcal{C}$ gives $u(T^{n}x)=o(n)$, while
$x\in\mathcal{A}(0,\infty)$ gives
$\liminf\limits_n S_nu(x)/n=\underline{u}(x)>0$, so that
$S_nu(x)\geq n(\underline{u}(x)-o(1))$ and both denominators grow at
least linearly. The choice $2\varepsilon<\underline{u}(x)$ keeps the
$\rho_n'$ denominator positive and linearly growing. By
Lemma~\ref{lem:scale-seq}, both scale families compute the
pointwise dimensions:
\begin{equation}\label{eq:reparam}
\dim^A(x)=\liminf\limits_{n\to\infty}\frac{K^A_{\rho_n}(x)}{-\log\rho_n}
=\liminf\limits_{n\to\infty}\frac{K^A_{\rho_n'}(x)}{-\log\rho_n'},
\end{equation}
and the same with $\limsup$ and $\Dim^A(x)$.\\
For any $x_i\in B_{n}(x,T,\delta)$ we have
\[
|S_nu(x_i)-S_nu(x)|<n\varepsilon,
\]
therefore for any $x_i\in B_{n}(x,T,\delta)$ we have
\[
|{-\log \rho_n}-S_nu(x_i)|\leq 2n\varepsilon+O(1),
\qquad
|{-\log \rho_n'}-S_nu(x_i)|\leq 2n\varepsilon+O(1).
\]
This gives (for $n$ large enough so $S_nu(x_i)>0$ for all $x_i\in B_n(x,T,\delta)$)
\[
1-\frac{2\varepsilon+O(1)}{\underline{u}(x)+\varepsilon+o(1)}
\leq\frac{S_nu(x_i)}{-\log \rho_n}
\leq 1+\frac{2\varepsilon+O(1)}{\underline{u}(x)+\varepsilon+o(1)}.
\]
We now apply this to the complexity
\begin{multline*}
\left(1-\frac{2\varepsilon+O(1)}{\underline{u}(x)+o(1)}\right)
\frac{\min\limits_{x_i\in B_n(x,T,\delta)}K^A(i)}{-\log \rho_n'}
\leq \min_{x_i\in B_n(x,T,\delta)}\frac{K^A(i)}{S_nu(x_i)}
\leq \left(1+\frac{2\varepsilon+O(1)}{\underline{u}(x)+o(1)}\right)
\frac{\min\limits_{x_i\in B_n(x,T,\delta)}K^A(i)}{-\log \rho_n}.
\end{multline*}
The two inclusions of Lemma~\ref{lem:ball-inclusion} reverse
under minima:
\begin{equation}\label{eq:min-sandwich}
K^A_{\rho_n'}(x)
\;\leq\;\min_{x_i\in B_n(x,T,\delta)}K^A(i)
\;\leq\;K^A_{\rho_n}(x),
\end{equation}
and combining this with the above inequality gives
\[
\left(1-\frac{2\varepsilon}{\underline{u}(x)+o(1)}\right)
\frac{K^A_{\rho_n'}(x)}{-\log \rho_n'}
\leq \min_{x_i\in B_n(x,T,\delta)}\frac{K^A(i)}{S_nu(x_i)}
\leq \left(1+\frac{2\varepsilon}{\underline{u}(x)+o(1)}\right)
\frac{K^A_{\rho_n}(x)}{-\log \rho_n}.
\]
Taking the liminf and limsup respectively gives
\[
\left(1-\frac{2\varepsilon}{\underline{u}(x)}\right)\dim^A(x)
\leq \dim^A_{BS}(x,T,u,\delta)
\leq \left(1+\frac{2\varepsilon}{\underline{u}(x)}\right)\dim^A(x),
\]
\[
\left(1-\frac{2\varepsilon}{\underline{u}(x)}\right)\Dim^A(x)
\leq \Dim^A_{BS}(x,T,u,\delta)
\leq \left(1+\frac{2\varepsilon}{\underline{u}(x)}\right)\Dim^A(x),
\]
and taking $\delta\to0$ gives
\[
\dim^A(x)\leq \dim^A_{BS}(x,T,u)\leq \dim^A(x),
\]
\[
\Dim^A(x)\leq \Dim^A_{BS}(x,T,u)\leq \Dim^A(x).\qedhere
\]
\end{proof}

\begin{theorem}[Dimensions of nonuniformly expanding subsets of conformal systems]
\label{thm:conformal-dim}
Let $(X,d,T)$ be conformal and let
$Z\subset\mathcal{A}(0,\infty)\cap\mathcal{B}\cap\mathcal{C}$.
Then
\[\dim_{H}(Z)=\dim_{BS}(Z,T,u),
\qquad
\dim_P(Z)=\Dim_{BS}(Z,T,u).\]
\end{theorem}
\begin{proof}
By Theorem~\ref{thm:pointwise-conformal}, for every oracle $A$ and
every $x\in Z$ we have $\dim^A(x)=\dim^A_{BS}(x,T,u)$ and
$\Dim^A(x)=\Dim^A_{BS}(x,T,u)$. Hence, by the point-to-set
principles for Hausdorff dimension \cite{lutz-2023} and
for lower BS dimension (Theorem~\ref{thm:pts_lower_bs}),
\[\dim_H(Z)
=\min\limits_{A\subset \mathbb{N}}\sup_{x\in Z}\dim^A(x)
=\min\limits_{A\subset \mathbb{N}}\sup_{x\in Z}\dim^A_{BS}(x,T,u)
=\dim_{BS}(Z,T,u),\]
and by the principles for packing dimension
\cite{lutz-2023} and upper BS dimension
(Theorem~\ref{thm:pts_upper}),
\[\dim_P(Z)
=\min\limits_{A\subset \mathbb{N}}\sup_{x\in Z}\Dim^A(x)
=\min\limits_{A\subset \mathbb{N}}\sup_{x\in Z}\Dim^A_{BS}(x,T,u)
=\Dim_{BS}(Z,T,u).\]
\end{proof}
We now combine Theorem~\ref{thm:conformal-dim} with the results of
Section~\ref{subsec:existence-uniqueness}.
\repthm{thm:BPE}
\repthm{thm:BPE_unique}
\begin{theorem}[Bowen's equation for conformal systems]
\label{thm:conformal-bowen}
Let $(X,d,T)$ be conformal and let
$Z\subset\mathcal{A}(0,\infty)\cap\mathcal{B}\cap\mathcal{C}$.
Then
\[\dim_{H}(Z)
=\sup\{s\mid\underline{P}(Z,T,-su)> 0\}
=\inf\{s\mid \underline{P}(Z,T,-su)\leq 0\},\]
\[\dim_P(Z)
=\sup\{s\mid\overline{P}(Z,T,-su)> 0\}
=\inf\{s\mid \overline{P}(Z,T,-su)\leq 0\}.\]
If moreover
$Z\subset\mathcal{A}(\alpha,\beta)$ for some
$0<\alpha\leq\beta<\infty$ and $\underline{h}(Z,T)<\infty$ (respectively $\overline{h}(Z,T)<\infty$), then the lower and upper Bowen pressure
equations
\[\underline{P}(Z,T,-su)=0,
\qquad \overline{P}(Z,T,-su)=0\]
have unique solutions, given by $\dim_H(Z)$ and $\dim_P(Z)$
respectively.
\end{theorem}
\begin{proof}
The displayed characterizations are Theorem~\ref{thm:BPE} combined
with Theorem~\ref{thm:conformal-dim}, the hypotheses of the former
holding since $Z\subset\mathcal{A}(0,\infty)$. The final
statement is Theorem~\ref{thm:BPE_unique} combined with
Theorem~\ref{thm:conformal-dim}, the hypotheses
$\alpha\leq\underline{u}\leq\overline{u}\leq\beta$ on $Z$ holding
since $Z\subset\mathcal{A}(\alpha,\beta)$.
\end{proof}

\appendix
\section{Equivalence with the classical packing-type definitions}
\subsection{Packing pressure and the equivalence with modified upper
capacity pressure}\label{app:packing-pressure}

In this appendix we prove the equivalence, promised in
Section~\ref{subsubsec:upper-pressure-disc}, between the modified
upper capacity pressure $P_{\mathrm{mUC}}$ adopted there as the
working definition of upper pressure and the classical packing
pressure. The route passes through the intermediate quantity $P_{\mathrm{mUC}}^{\mathrm{sep}}$ defined via separated sets (definition \ref{def:UC-separated}), and
the chain of comparisons is
\[
P^{\mathrm{pack}}
\;\overset{\text{Lem.~\ref{lem:pack-geq-sepmUC}+Lem.~\ref{lem:pack-leq-sepmUC}}}{=\joinrel=}\;
P_{\mathrm{mUC}}^{\mathrm{sep}}
\;\overset{\text{Cor.~\ref{cor:sep-cov-pressure}}}{=\joinrel=}\;
P_{\mathrm{mUC}},
\]
all three quantities being weighted by the sup-weights
$\varphi(x,n,\varepsilon)$. The comparisons cost at most a
scale-halving and a factor $2^{\,n\Var(\varphi,2\varepsilon)}$,
which vanish at scale $0$ (Lemma~\ref{lem:span-sep}).

\begin{remark}[Center weights versus sup weights]
\label{rem:weights-appendix}
The set-level definitions of the body and of this appendix weight a
Bowen ball $B_n(x,T,\varepsilon)$ by the supremum of the potential
over the ball, while the discretized covers and nets of the body
weight it at an ideal-point center. If $y\in B_n(x_i,T,\varepsilon)$
then $d(T^ky,T^kx_i)<\varepsilon$ for all $k<n$, so
$|S_n\varphi(y)-S_n\varphi(x_i)|\leq n\Var(\varphi,\varepsilon)$ and
hence
\[
0\;\leq\;\varphi(x_i,n,\varepsilon)-S_n\varphi(x_i)
\;\leq\; n\Var(\varphi,\varepsilon),
\]
and likewise
$0\leq u(x_i,n,\varepsilon)-S_nu(x_i)\leq n\Var(u,\varepsilon)$.
These are the only estimates needed to pass between the two
weightings: the pressure discretizations absorb them additively
(critical values shift by at most $\Var(\varphi,\varepsilon)$,
vanishing at scale $0$), and the BS discretizations absorb them into
an arbitrarily small increase of the exponent using
$\underline{u}>0$ (Lemmas~\ref{lem:null_cover}, \ref{lem:s-net},
\ref{lem:bs_null_cover}, and~\ref{lem:bs-s-net}).
\end{remark}

\begin{definition}[Packing pressure]\label{def:packing-pressure}
For $Z \subset X$, $s \in \mathbb{R}$, $\varepsilon>0$, and
$N \in \mathbb{N}$, let $\mathcal{P}^*(Z,N,\varepsilon)$ denote the
collection of finite or countable disjoint families
$\{B_{n_i}(x_i,T,\varepsilon)\}_i$ with $x_i \in Z$ and
$n_i \geq N$, and set
\[
M(Z, s, \varphi, N, \varepsilon)
= \sup_{\mathcal{P}^*(Z,N,\varepsilon)} \sum_{i}
2^{-s n_i + \varphi(x_i,n_i,\varepsilon)}.
\]
$M$ is nonincreasing in $N$; let
\[
m^*(Z, s, \varphi, \varepsilon)
= \lim\limits_{N \to \infty} M(Z, s, \varphi, N, \varepsilon),
\qquad
m^{**}(Z, s, \varphi, \varepsilon)
= \inf\Bigl\{ \sum_{i} m^*(Z_i, s, \varphi, \varepsilon)
: Z\subset\bigcup_{i} Z_i \Bigr\}.
\]
The function $s\mapsto m^{**}(Z,s,\varphi,\varepsilon)$ takes the
value $\infty$ below a critical value and $0$ above it, and we set
\[
\overline{P}{}^{\mathrm{pack}}(Z,T,\varphi,\varepsilon)
= \inf\{s : m^{**}(Z, s, \varphi, \varepsilon) = 0\}
= \sup\{s : m^{**}(Z, s, \varphi, \varepsilon) = \infty\},
\qquad
\overline{P}^{\mathrm{pack}}(Z,T,\varphi)
= \lim\limits_{\varepsilon \to 0}
\overline{P}{}^{\mathrm{pack}}(Z,T,\varphi,\varepsilon).
\]
The packing entropy is
$\overline{h}(Z,T)=\overline{P}{}^{\mathrm{pack}}(Z,T,0)$.
\end{definition}
\begin{lemma}[Scale monotonicity of packing pressure up to variation]
\label{lem:pack-scale-monotone}
For $0<\varepsilon'\leq\varepsilon$,
\[
\overline{P}{}^{\mathrm{pack}}(Z,T,\varphi,\varepsilon')
\;\geq\;
\overline{P}{}^{\mathrm{pack}}(Z,T,\varphi,\varepsilon)
-\Var(\varphi,\varepsilon).
\]
Consequently the limit defining
$\overline{P}{}^{\mathrm{pack}}(Z,T,\varphi)$ exists.
\end{lemma}
\begin{proof}
A disjoint family $\{B_{n_i}(x_i,T,\varepsilon)\}_i$ admissible
for $M(Z,s,\varphi,N,\varepsilon)$ remains disjoint after
shrinking the radii to $\varepsilon'$, and since each center lies
in its own ball,
\[
\varphi(x_i,n_i,\varepsilon')\geq S_{n_i}\varphi(x_i)
\geq \varphi(x_i,n_i,\varepsilon)-n_i\Var(\varphi,\varepsilon).
\]
Hence
$M(Z,s,\varphi,N,\varepsilon')\geq
M(Z,s+\Var(\varphi,\varepsilon),\varphi,N,\varepsilon)$, and
this inequality passes through $m^*$ and $m^{**}$, shifting the
critical value by at most $\Var(\varphi,\varepsilon)$, which is
the displayed bound. For the limit: for every $\varepsilon>0$ the
bound gives
\[
\liminf\limits_{\varepsilon'\to 0}
\overline{P}{}^{\mathrm{pack}}(Z,T,\varphi,\varepsilon')
\geq \overline{P}{}^{\mathrm{pack}}(Z,T,\varphi,\varepsilon)
-\Var(\varphi,\varepsilon),
\]
and letting $\varepsilon\to 0$, so that
$\Var(\varphi,\varepsilon)\to 0$ by uniform continuity, gives
$\liminf_{\varepsilon\to 0}
\overline{P}{}^{\mathrm{pack}}(Z,T,\varphi,\varepsilon)
\geq\limsup_{\varepsilon\to 0}
\overline{P}{}^{\mathrm{pack}}(Z,T,\varphi,\varepsilon)$.
\end{proof}
\begin{remark}\label{rem:packing-def-correction}
In \cite{zhao_2017_a} the definition was given for compact metric
spaces. It parses verbatim on complete separable spaces with
$\varphi$ uniformly continuous, and the properties used below hold
with the same proofs: the $\infty/0$ dichotomy in $s$ (immediate, as
the weight $2^{-sn_i}$ with $n_i\geq N$ gives
$m^{**}(Z,s',\cdot)\leq 2^{-(s'-s)N}$-factored comparisons),
monotonicity of the critical value in $\varepsilon$, and countable
stability of $m^{**}$ (built into its definition). No arguments use compactness.
\end{remark}
\begin{definition}[Upper capacity pressure via separated sets]
\label{def:UC-separated}
For $Z \subset X$, $n\in\mathbb{N}$, and $\varepsilon>0$, recall
that $F\subset Z$ is \emph{$(n,\varepsilon)$-separated} if
distinct $x,y\in F$ satisfy $y\notin B_n(x,T,\varepsilon)$. Set
\[
Q_n(Z,T,\varphi,\varepsilon)
= \sup\Bigl\{ \sum_{x \in F} 2^{\varphi(x,n,\varepsilon)}
: F\subset Z \text{ is } (n, \varepsilon)\text{-separated}
\Bigr\},
\]
weighting each separated point by the same sup-weight used in
Definition~\ref{def:packing-pressure} and in the BS-side
Definition~\ref{def:mUCsep}.
The \emph{upper capacity pressure via separated sets} at scale $\varepsilon$ is
\[
P_{\mathrm{UC}}^{\mathrm{sep}}(Z,T,\varphi,\varepsilon)
=\limsup\limits_{n \to \infty}
\frac{1}{n} \log Q_n(Z,T,\varphi,\varepsilon);
\]
The \emph{modified upper capacity pressure via separated sets} at scale $\varepsilon$ is
\[P_{\mathrm{mUC}}^{\mathrm{sep}}(Z,T,\varphi,\varepsilon)=\inf_{\{\cup_iZ_i=Z\}} \sup_i P^{\mathrm{sep}}_{\mathrm{UC}}(Z_i,T,\varphi, \varepsilon)\]
The \emph{modified upper capacity pressure via separated sets} is
\[P_{\mathrm{mUC}}^{\mathrm{sep}}(Z,T,\varphi)=\lim\limits_{\varepsilon\to 0}P_{\mathrm{mUC}}^{\mathrm{sep}}(Z,T,\varphi,\varepsilon)\]
(the existence of this limit follows from the same argument as Lemma \ref{lem:pack-scale-monotone}).
\end{definition}
We first show that the scale-$0$ limit defining
$P^{\mathrm{sep}}_{\mathrm{mUC}}(Z,T,\varphi)$ equals
$P^{\mathrm{pack}}(Z,T,\varphi)$ via lemmas~\ref{lem:pack-geq-sepmUC}
and~\ref{lem:pack-leq-sepmUC}. 
\begin{lemma}\label{lem:pressure-sep-pack-eq}
     \[P^{\mathrm{sep}}_{\mathrm{mUC}}(Z,T,\varphi)=P^{\mathrm{pack}}(Z,T, \varphi)\]
\end{lemma}
\begin{remark}
    This lemma was proven in \cite{zhao_2017_a} with the same proof, but we reproduce it here for the following reasons. In the cited paper the definition of modified upper capacity pressure is stated as being defined with a countable decomposition at scale 0 instead of each scale separately. The proof was stated for this quantity while the theorem was proved for the standard modified upper capacity pressure instead. This was confirmed via private correspondence with the authors, who shared the corrected version of the statement. Simultaneously the lemma was only stated in the compact case while the proof only uses uniform continuity of the potential $\varphi$. To avoid confusion we decided to reproduce the lemma entirely.
\end{remark}
We then show
$P^{\mathrm{sep}}_{\mathrm{mUC}}(Z,T,\varphi)
=P_{\mathrm{mUC}}(Z,T,\varphi)$, which concludes our justification
of using $P_{\mathrm{mUC}}(Z,T,\varphi)$ as our definition.
\begin{lemma}\label{lem:pack-geq-sepmUC}
    \[P^{\mathrm{sep}}_{\mathrm{mUC}}(Z,T,\varphi)\leq P^{\mathrm{pack}}(Z,T, \varphi)\]
\end{lemma}
\begin{proof}
Fix any $a > P^{\mathrm{pack}}(Z,T, \varphi) = \lim\limits_{\varepsilon \to 0} P^{\mathrm{pack}}(Z,T, \varphi, \varepsilon)$. For $\varepsilon > 0$ small enough we have $P^{\mathrm{pack}}(Z,T, \varphi, \varepsilon) < a$, so $m^{**}(Z, a, \varphi, \varepsilon) = 0$. We can choose $\{Z_i\}_{i=1}^{\infty}$ such that $\bigcup_{i=1}^{\infty} Z_i \supset Z$ and $\sum_{i=1}^{\infty} m^*(Z_i, a, \varphi, \varepsilon) < 1$. Since $M$ is nonincreasing in $N$, for each $i$ and all $N$ large enough we have $M(Z_i, a, \varphi, N, \varepsilon) < 1$. Next, we will prove for any $i$,
\[
P^{\mathrm{sep}}_{\mathrm{UC}}(Z_i,T, \varphi, 2\varepsilon) \leq a + \Var(\varphi,2\varepsilon).
\]
Let $E_i$ be any $(N, 2\varepsilon)$-separated set of $Z_i$ with $N$ as above. The family $\{B_N(x,T, \varepsilon)\}_{x \in E_i}$ is disjoint: if two of its balls met, their centers would satisfy $d_N(x,x') < 2\varepsilon$, contradicting separation. Its centers lie in $Z_i$ and its orders equal $N$, so it is admissible for $M(Z_i, a, \varphi, N, \varepsilon)$. Hence
\[
\sum_{x \in E_i} 2^{-Na + \varphi(x,N,\varepsilon)} \leq M(Z_i, a, \varphi, N, \varepsilon) < 1.
\]
Since each $x \in E_i$ lies in its own ball,
$\varphi(x,N,2\varepsilon) \leq S_N\varphi(x)
+ N\Var(\varphi,2\varepsilon)
\leq \varphi(x,N,\varepsilon) + N\Var(\varphi,2\varepsilon)$, so
\[
\sum_{x \in E_i} 2^{\varphi(x,N,2\varepsilon)} \leq 2^{N(a + \Var(\varphi,2\varepsilon))}.
\]
Taking the supremum over such $E_i$ gives $Q_N(Z_i,T,\varphi,2\varepsilon) \leq 2^{N(a+\Var(\varphi,2\varepsilon))}$ for all large $N$, hence $P^{\mathrm{sep}}_{\mathrm{UC}}(Z_i,T,\varphi,2\varepsilon) \leq a + \Var(\varphi,2\varepsilon)$. The decomposition $\{Z_i\}$ therefore witnesses $P^{\mathrm{sep}}_{\mathrm{mUC}}(Z,T,\varphi,2\varepsilon) \leq a + \Var(\varphi,2\varepsilon)$ for all sufficiently small $\varepsilon$. Letting $\varepsilon \to 0$, so that $\Var(\varphi,2\varepsilon)\to 0$ by uniform continuity, gives $\lim\limits_{\varepsilon\to 0}P^{\mathrm{sep}}_{\mathrm{mUC}}(Z,T,\varphi,\varepsilon)\leq a$, and $a \downarrow P^{\mathrm{pack}}(Z,T,\varphi)$ finishes the proof.
\end{proof}
\begin{lemma}\label{lem:pack-leq-sepmUC}
    \[P^{\mathrm{pack}}(Z,T, \varphi)\;\leq\; P^{\mathrm{sep}}_{\mathrm{mUC}}(Z,T,\varphi)\]
\end{lemma}
\begin{proof}
Fix rationals $0 < t < s < P^{\mathrm{pack}}(Z,T, \varphi)$. Since $P^{\mathrm{pack}}(Z,T,\varphi) = \lim\limits_{\varepsilon\to 0}P^{\mathrm{pack}}(Z,T,\varphi,\varepsilon)$, for all sufficiently small $\varepsilon > 0$ we have $P^{\mathrm{pack}}(Z,T,\varphi,\varepsilon) > s$, hence $m^{**}(Z, s, \varphi, \varepsilon) = \infty$. Fix such an $\varepsilon$ and any countable decomposition $\bigcup_{i=1}^{\infty} Z_i = Z$. Since $m^{**}$ is an infimum over such decompositions, there are $\alpha > 0$ and some $Z_i$ with $m^*(Z_i, s, \varphi, \varepsilon) > \alpha$.

As $m^*(Z_i,s,\varphi,\varepsilon) = \lim\limits_{N\to\infty} M(Z_i,s,\varphi,N,\varepsilon)$ with $M$ nonincreasing in $N$, for every $N$ there is a disjoint family $\{B_{n_j}(x_j, T,\varepsilon)\}_{j\geq 1}$ with $x_j \in Z_i$ and $n_j \geq N$ such that
\[
\sum_{j\geq 1} 2^{-s n_j + \varphi(x_j,n_j,\varepsilon)} \geq \frac{\alpha}{2}.
\]
Let $E_k = \{x_j : n_j = k\}$; regrouping,
\[
\sum_{k\geq N} \sum_{x \in E_k} 2^{-sk + \varphi(x,k,\varepsilon)} \geq \frac{\alpha}{2}.
\]
We claim there exists $k \geq N$ such that
\[
\sum_{x \in E_k} 2^{-tk + \varphi(x,k,\varepsilon)} \geq \frac{\bigl(1 - 2^{-(s-t)}\bigr)\alpha}{2}:
\]
otherwise, splitting $2^{-sk} = 2^{-(s-t)k}\,2^{-tk}$,
\[
\sum_{k\geq N} \sum_{x \in E_k} 2^{-sk + \varphi(x,k,\varepsilon)}
= \sum_{k\geq N} 2^{-(s-t)k}\sum_{x \in E_k} 2^{-tk + \varphi(x,k,\varepsilon)}
< \sum_{k\geq 1} 2^{-(s-t)k}\,\frac{\bigl(1 - 2^{-(s-t)}\bigr)\alpha}{2}
= 2^{-(s-t)}\,\frac{\alpha}{2}
< \frac{\alpha}{2},
\]
a contradiction.

Fix such a $k \geq N$. The disjointness of the balls $\{B_k(x,T,\varepsilon)\}_{x\in E_k}$ makes $E_k$ a $(k,\varepsilon)$-separated subset of $Z_i$: if $d_k(x,x') < \varepsilon$ for distinct $x,x' \in E_k$, then $x' \in B_k(x,T,\varepsilon)$, meeting both balls. Its weighted sum is exactly of the form appearing in $Q_k$, so
\[
Q_k(Z_i,T,\varphi, \varepsilon)
\;\geq\; \sum_{x \in E_k} 2^{\varphi(x,k,\varepsilon)}
\;\geq\; 2^{tk}\,\frac{\bigl(1 - 2^{-(s-t)}\bigr)\alpha}{2}.
\]
As such a $k$ exists above every $N$, taking $\limsup\limits_{k\to\infty}\frac{1}{k}\log$ gives
\[
P^{\mathrm{sep}}_{\mathrm{UC}}(Z_i,T,\varphi,\varepsilon) \;\geq\; t.
\]
Since the decomposition was arbitrary, $P^{\mathrm{sep}}_{\mathrm{mUC}}(Z,T,\varphi,\varepsilon) \geq t$ for every sufficiently small $\varepsilon$, so $\lim\limits_{\varepsilon\to 0}P^{\mathrm{sep}}_{\mathrm{mUC}}(Z,T,\varphi,\varepsilon)\geq t$. As $t < s < P^{\mathrm{pack}}(Z,T,\varphi)$ were arbitrary rationals, the lemma follows.
\end{proof}

We now show that the upper capacity pressure definitions via covers or separated sets are equivalent.
\begin{lemma}[Spanning-separated comparison]\label{lem:span-sep}
For every $Z\subset X$, $n\in\mathbb{N}$, and $\varepsilon>0$,
\[
N(Z,T,\varphi,n,2\varepsilon)
\;\leq\; Q_n(Z,T,\varphi,2\varepsilon)
\;\leq\; 2^{\,n\,\Var(\varphi,2\varepsilon)}\,
N(Z,T,\varphi,n,\varepsilon).
\]
\end{lemma}

\begin{proof}
If $d_n(x,y)<\varepsilon$ then $d(T^ix,T^iy)<\varepsilon$ for $0\le i<n$, so
\begin{equation}
|S_n\varphi(x)-S_n\varphi(y)|\le n\operatorname{Var}(\varphi,\varepsilon).
\tag{$\ast$}
\end{equation}

For the left inequality, let $F\subset Z$ be
$(n,2\varepsilon)$-separated; enlarging $F$ to a maximal such set
only increases its weighted sum, so we may assume $F$ maximal. By
maximality every $y\in Z$ satisfies $d_n(y,x)<2\varepsilon$ for some
$x\in F$, so $Z\subset\bigcup_{x\in F}B_n(x,T,2\varepsilon)$, and
this cover's weight in $N$ is exactly the separated sum:
\[
N(Z,T,\varphi,n,2\varepsilon)
\le\sum_{x\in F}2^{\varphi(x,n,2\varepsilon)}
\le Q_n(Z,T,\varphi,2\varepsilon).
\]

For the right inequality, let $F\subset Z$ be
$(n,2\varepsilon)$-separated and let $E\subset X$ be countable with
$Z\subset\bigcup_{x\in E}B_n(x,T,\varepsilon)$. For each $y\in F$
pick $\pi(y)\in E$ with $y\in B_n(\pi(y),T,\varepsilon)$. If
$\pi(y)=\pi(y')$ then $d_n(y,y')<2\varepsilon$, so $\pi$ is
injective. Since $y\in B_n(\pi(y),T,\varepsilon)$ we have
$S_n\varphi(y)\le \varphi(\pi(y),n,\varepsilon)$, and $(\ast)$
applied at scale $2\varepsilon$ gives
$\varphi(y,n,2\varepsilon)\le
S_n\varphi(y)+n\Var(\varphi,2\varepsilon)$, so
\[
2^{\varphi(y,n,2\varepsilon)}
\le 2^{\,n\Var(\varphi,2\varepsilon)}\,
2^{\varphi(\pi(y),n,\varepsilon)}.
\]
Summing over $y\in F$ and using injectivity, then taking the
supremum over $F$ and the infimum over $E$, yields
$Q_n(Z,T,\varphi,2\varepsilon)\le
2^{\,n\Var(\varphi,2\varepsilon)}\,N(Z,T,\varphi,n,\varepsilon)$.
\end{proof}
\begin{corollary}[Separated and covering modified upper capacity pressures agree at scale 0]\label{cor:sep-cov-pressure}
For $Z\subset X$ and $\varepsilon>0$,
\[
P_{\mathrm{mUC}}(Z,T,\varphi,2\varepsilon)
\;\leq\;P_{\mathrm{mUC}}^{\mathrm{sep}}(Z,T,\varphi,2\varepsilon)
\;\leq\;P_{\mathrm{mUC}}(Z,T,\varphi,\varepsilon)+\Var(\varphi,2\varepsilon).
\]
Consequently the limit
$P_{\mathrm{mUC}}(Z,T,\varphi)=\lim\limits_{\varepsilon\to0}P_{\mathrm{mUC}}(Z,T,\varphi,\varepsilon)$
exists and
\[
P_{\mathrm{mUC}}(Z,T,\varphi)
\;=\;P_{\mathrm{mUC}}^{\mathrm{sep}}(Z,T,\varphi)
\;=\;P^{\mathrm{pack}}(Z,T,\varphi).
\]
\end{corollary}
\begin{proof}
Fix $Z\subset X$ and $\varepsilon>0$. Taking $\frac{1}{n}\log$
and $\limsup\limits_{n\to\infty}$ in Lemma~\ref{lem:span-sep},
the correction factor contributes exactly
$\Var(\varphi,2\varepsilon)$:
\[
P_{\mathrm{UC}}(Z,T,\varphi,2\varepsilon)
\;\leq\;P^{\mathrm{sep}}_{\mathrm{UC}}(Z,T,\varphi,2\varepsilon)
\;\leq\;P_{\mathrm{UC}}(Z,T,\varphi,\varepsilon)+\Var(\varphi,2\varepsilon).
\]
Now let $Z\subset\bigcup_i Z_i$ be any countable cover. Applying the
display to each $Z=Z_i$, taking $\sup_i$ (the additive constant
$\Var(\varphi,2\varepsilon)$ passing through the supremum) and then
the infimum over the same class of countable covers on both sides
yields the first claim.

For the second claim, combine the first claim with
Lemma~\ref{lem:pressure-sep-pack-eq}. As $\varepsilon\to 0$,
\[
\limsup\limits_{\varepsilon\to 0}
P_{\mathrm{mUC}}(Z,T,\varphi,2\varepsilon)
\leq\limsup\limits_{\varepsilon\to 0}
P_{\mathrm{mUC}}^{\mathrm{sep}}(Z,T,\varphi,2\varepsilon)
= P^{\mathrm{pack}}(Z,T,\varphi),
\]
while
\[
P^{\mathrm{pack}}(Z,T,\varphi)
=\liminf\limits_{\varepsilon\to 0}
P_{\mathrm{mUC}}^{\mathrm{sep}}(Z,T,\varphi,2\varepsilon)
\leq\liminf\limits_{\varepsilon\to 0}
\bigl(P_{\mathrm{mUC}}(Z,T,\varphi,\varepsilon)
+\Var(\varphi,2\varepsilon)\bigr)
=\liminf\limits_{\varepsilon\to 0}
P_{\mathrm{mUC}}(Z,T,\varphi,\varepsilon),
\]
using $\Var(\varphi,2\varepsilon)\to 0$. Hence
$P_{\mathrm{mUC}}(Z,T,\varphi,\varepsilon)$ converges as
$\varepsilon\to 0$ to
$P^{\mathrm{pack}}(Z,T,\varphi)
=P_{\mathrm{mUC}}^{\mathrm{sep}}(Z,T,\varphi)$.
\end{proof}

\subsection{Dimensional definition of BS dimension}\label{app:packing-bs}
\subsubsection{Lower BS dimension}
In this section we establish basic facts about the lower BS dimension. These require proof because our definition generalizes the classical one.
\begin{lemma}[Countable subadditivity]\label{lem:bs-subadd}
For every $s\geq 0$, $N\in\mathbb{N}$, and $\varepsilon>0$,
\[
m_{BS}\Big(\bigcup_{i\in\mathbb{N}} Z_i,\, s, u, N, \varepsilon\Big)
\;\leq\; \sum_{i\in\mathbb{N}} m_{BS}(Z_i, s, u, N, \varepsilon),
\qquad
m_{BS}\Big(\bigcup_{i\in\mathbb{N}} Z_i,\, s, u, \varepsilon\Big)
\;\leq\; \sum_{i\in\mathbb{N}} m_{BS}(Z_i, s, u, \varepsilon).
\]
\end{lemma}

\begin{proof}
Write $Z=\bigcup_i Z_i$. Fix $\delta>0$ and for each $i$ choose
$\Gamma_i\in\mathcal{P}(Z_i,N,\varepsilon)$ with
\[
\sum_{(y,n)\in\Gamma_i} 2^{-s\,u(y,n,\varepsilon)} \;\leq\; m_{BS}(Z_i,s,u,N,\varepsilon) + \delta 2^{-i}.
\]
Then $\Gamma=\bigcup_i\Gamma_i$ is a countable family of pairs $(y,n)$ with $n\geq N$ whose
Bowen balls cover $Z$, so $\Gamma\in\mathcal{P}(Z,N,\varepsilon)$ and
\[
m_{BS}(Z,s,u,N,\varepsilon)
\;\leq\; \sum_i \sum_{(y,n)\in\Gamma_i} 2^{-s\,u(y,n,\varepsilon)}
\;\leq\; \sum_i m_{BS}(Z_i,s,u,N,\varepsilon) + 2\delta.
\]
Letting $\delta\to 0$ gives the first inequality. The second follows by letting
$N\to\infty$: each summand $m_{BS}(Z_i,s,u,N,\varepsilon)$ is nondecreasing in $N$, so
by monotone convergence
$\sum_i m_{BS}(Z_i,s,u,N,\varepsilon)\to\sum_i m_{BS}(Z_i,s,u,\varepsilon)$.
\end{proof}
\begin{lemma}[Dichotomy for $m_{BS}$]\label{lem:bs-dichotomy}
Suppose $\underline{u}(x)>0$ for all $x\in Z$, and let $\varepsilon>0$. If
$m_{BS}(Z,s,u,\varepsilon)<\infty$, then $m_{BS}(Z,s',u,\varepsilon)=0$ for every
$s'>s$. Consequently
\[
\inf\{s\geq 0 \mid m_{BS}(Z,s,u,\varepsilon)=0\}
=\sup\{s\geq 0 \mid m_{BS}(Z,s,u,\varepsilon)=\infty\}.
\]
\end{lemma}

\begin{proof}
For $N\in\mathbb{N}$ and rational $\delta>0$, set
\[
C(Z,N,\delta)=\{x\in Z : S_nu(x)\geq \delta n \text{ for all } n\geq N\}.
\]
Since $\underline{u}>0$ on $Z$, these countably many sets exhaust $Z$.

Fix one piece $C=C(Z,N,\delta)$, let $N'\geq N$, and let
$\Gamma\in\mathcal{P}(C,N',\varepsilon)$. Each ball $B_n(y,T,\varepsilon)$ with
$(y,n)\in\Gamma$ contains some $\tilde x\in C$, so its weight satisfies
\[
u(y,n,\varepsilon)\;\geq\;S_nu(\tilde x)
\;\geq\;\delta n\;\geq\;\delta N'.
\]
Hence, since $s'-s>0$,
\[
\sum_{(y,n)\in\Gamma}2^{-s'\,u(y,n,\varepsilon)}
=\sum_{(y,n)\in\Gamma}2^{-(s'-s)\,u(y,n,\varepsilon)}\,2^{-s\,u(y,n,\varepsilon)}
\;\leq\;2^{-(s'-s)\delta N'}\sum_{(y,n)\in\Gamma}2^{-s\,u(y,n,\varepsilon)}.
\]
Taking the infimum over $\Gamma$ and using $C\subset Z$ together with monotonicity
in $N'$,
\[
m_{BS}(C,s',u,N',\varepsilon)
\;\leq\;2^{-(s'-s)\delta N'}\,m_{BS}(Z,s,u,\varepsilon)
\;\xrightarrow[N'\to\infty]{}\;0,
\]
so $m_{BS}(C,s',u,\varepsilon)=0$, and by countable subadditivity
(Lemma~\ref{lem:bs-subadd}) $m_{BS}(Z,s',u,\varepsilon)=0$.

For the displayed identity: the first part shows the sets
$\{s\mid m_{BS}=\infty\}$ and $\{s\mid m_{BS}=0\}$ are complementary rays in
$[0,\infty)$ up to at most one boundary point, so their defining $\sup$ and $\inf$
coincide.
\end{proof}

\subsubsection{Upper BS dimension}
The modified upper capacity BS dimension was defined in Definition~\ref{def:mBSC}. We now define the packing BS dimension.
\begin{definition}
Suppose $Z \subset X$ satisfies $\underline{u}(x) > 0$ for all $x \in Z$. For
$\alpha \in \mathbb{R}$, $0 < \varepsilon \in \mathbb{R}$, $u:X\to \mathbb{R}$ uniformly continuous and
$N \in \mathbb{N}$, let
\[
M(Z, \alpha, u, \varepsilon, N)
= \sup_{\mathcal{P}^*(Z,N,\varepsilon)} \sum_{i=1}^{\infty}
\exp_2\!\left(-\alpha\, u(x_i,n_i,\varepsilon)\right),
\]
where $\mathcal{P}^*(Z, N, \varepsilon)$ denotes the finite or countable disjoint
set family $\{B_{n_i}(x_i,T,\varepsilon)\}$ with $x_i \in Z$,
$n_i \geq N$. Clearly, $M(Z, \alpha, u, \varepsilon, N)$ is a nonincreasing
function in $N$. Let
\[
m^*(Z, \alpha, u, \varepsilon) = \lim\limits_{N \to \infty} M(Z, \alpha, u, \varepsilon, N).
\]
Moreover, let
\[
m^{**}(Z, \alpha, u, \varepsilon)
= \inf\left\{ \sum_{i=1}^{\infty} m^*(Z_i, \alpha, u, \varepsilon)
: \bigcup_{i=1}^{\infty} Z_i \supseteq Z \right\}.
\]
Then
\[
\dim_{\mathrm{BSP}}(Z, \varepsilon)
= \inf\{\alpha : m^{**}(Z, \alpha, u, \varepsilon) = 0\}=\sup\{\alpha : m^{**}(Z, \alpha, u, \varepsilon) = +\infty\}
\]
Finally, we define the BS-packing (or BS-P) dimension
\[
\dim_{\mathrm{BSP}} Z := \lim\limits_{\varepsilon \to 0} \dim_{\mathrm{BSP}}(Z, \varepsilon).
\]
Note that $\dim_{\mathrm{BSP}}(Z, \varepsilon)$ is increasing as $\varepsilon \to 0$,
so the limit exists.
\end{definition}

We show that the two thresholds defining the finite-scale quantity coincide, so that $\dim_{\mathrm{BSP}}(Z,\varepsilon)$ is well defined.
\begin{lemma}\label{lem:packing-bs-dichotomy}
For every $\varepsilon > 0$,
\[
\inf\{\alpha : m^{**}(Z, \alpha, u, \varepsilon) = 0\}
= \sup\{\alpha : m^{**}(Z, \alpha, u, \varepsilon) = +\infty\}.
\]
\end{lemma}

\begin{proof}
Fix $\varepsilon > 0$. It suffices to show that whenever
$m^{**}(Z, \alpha, u, \varepsilon) < \infty$, we have
$m^{**}(Z, \beta, u, \varepsilon) = 0$ for every $\beta > \alpha$; the two
displayed quantities are then forced to coincide.

Let $Z = \bigcup_i Z_i$ be a decomposition with
$\sum_{i=1}^{\infty} m^*(Z_i, \alpha, u, \varepsilon) < \infty$, and refine it by
the $\mathbb{N}^2$-indexed family
\[
C(Z_i, N, 2^{-N}) = \{x \in Z_i : S_n u(x) \geq 2^{-N} n \text{ for all } n \geq N\}.
\]
Since every $x \in Z$ satisfies $\liminf\limits_{n} \tfrac{1}{n} S_n u(x) > 0$, each $x$
lies in $C(Z_i, N, 2^{-N})$ for some $N$, so
$\bigcup_{i,N} C(Z_i, N, 2^{-N}) = Z$.

Fix $i, N$ and write $C = C(Z_i, N, 2^{-N})$. We show
$m^*(C, \beta, u, \varepsilon) = 0$. Let $n \geq N$. For any family
$\mathcal{P}^*(C, n, \varepsilon)$ each center $x_i \in C$ satisfies
$S_{n_i} u(x_i) \geq 2^{-N} n_i$, and since $x_i \in B_{n_i}(x_i,T,\varepsilon)$,
\[
u(x_i,n_i,\varepsilon)
\;\geq\; 2^{-N} n_i \;\geq\; 2^{-N} n.
\]
Splitting $\beta = (\beta - \alpha) + \alpha$,
\[
M(C, \beta, u, \varepsilon, n)
= \sup_{\mathcal{P}^*(C,n,\varepsilon)} \sum_{i=1}^{\infty}
\exp_2\!\left(-(\beta-\alpha)\, u(x_i,n_i,\varepsilon)
           -\alpha\, u(x_i,n_i,\varepsilon)\right)
\]
\[
\leq \exp_2\!\left(-(\beta-\alpha)\, 2^{-N} n\right)\,
M(C, \alpha, u, \varepsilon, n)
\]
Since $M(C, \alpha, u, \varepsilon, n)$ is uniformly bounded in $n$, we get $m^*(C, \beta, u, \varepsilon) = 0$.
\end{proof}
The packing BS dimension is most easily related to the modified upper capacity BS dimension defined via a separated set definition rather than the modified upper capacity BS dimension defined via coverings. We will later show these coincide at scale 0.
\begin{definition}[Upper capacity BS dimension, separated set definition]
\label{def:mUCsep}
Let $Z\subset X$ and let $u\colon X\to\mathbb{R}$ be uniformly continuous with
$\underline{u}(x)>0$ for all $x\in Z$. For $\alpha\in\mathbb{R}$, $n\in\mathbb{N}$,
$\varepsilon>0$ set
\[
Q_n(Z,T,u,\alpha,\varepsilon)
=\sup\Bigg\{\sum_{x\in F}
2^{-\alpha\, u(x,n,\varepsilon)}
\;\Bigg|\; F\subset Z \text{ is an } (n,\varepsilon)\text{-separated set}\Bigg\},
\]
\[
q(Z,T,u,\alpha,\varepsilon)=\limsup\limits_{n\to\infty}Q_n(Z,T,u,\alpha,\varepsilon).
\]
The \emph{upper capacity BS dimension} of $Z$ at scale $\varepsilon$ is the critical value
\[
\dim^{\mathrm{sep}}_{\mathrm{BSC}}(Z,T,u,\varepsilon)
=\inf\{\alpha : q(Z,T,u,\alpha,\varepsilon) = 0\}
\]
and the \emph{modified upper capacity BS dimension} of $Z$ at scale $\varepsilon$ is
obtained by decomposing into countably many pieces:
\[
\Dim^{\mathrm{sep}}_{BS}(Z,T,u,\varepsilon)
=\inf\Bigg\{\,\sup_{i\in\mathbb{N}}
\dim^{\mathrm{sep}}_{\mathrm{BSC}}(Z_i,T,u,\varepsilon)
\;\Bigg|\; Z\subset\bigcup_{i\in\mathbb{N}} Z_i\Bigg\}.
\]
Finally,
\[
\Dim^{\mathrm{sep}}_{BS}(Z,T,u)
=\lim\limits_{\varepsilon\to0}\Dim^{\mathrm{sep}}_{BS}(Z,T,u,\varepsilon).
\]
\end{definition}
We now show relations between the packing BS dimension and the separated set version of the modified upper capacity BS dimension.
\begin{lemma}\label{lem:mUCsep-leq-BSP}
Suppose $\underline{u}(x)>0$ for all $x\in Z$. Then
\[
\Dim^{\mathrm{sep}}_{BS}(Z,T,u)\;\leq\;\dim_{\mathrm{BSP}}(Z,T,u).
\]
\end{lemma}

\begin{proof}
Since $\underline u>0$ on $Z$, both quantities are nonnegative, so we may fix any
$a>\dim_{\mathrm{BSP}}(Z,T,u)=\lim\limits_{\delta\to 0}\dim_{\mathrm{BSP}}(Z,T,u,\delta)$
with $a\geq 0$. For $\delta>0$ small enough we have
$\dim_{\mathrm{BSP}}(Z,T,u,\delta)<a$, so $m^{**}(Z,a,u,\delta)=0$. We can choose
$\{Z_i\}_{i=1}^{\infty}$ such that $\bigcup_{i=1}^{\infty}Z_i\supset Z$ and
$\sum_{i=1}^{\infty}m^{*}(Z_i,a,u,\delta)<1$; in particular
for each $i$ and all $N$ large enough,
$M(Z_i,a,u,\delta,N)<1$, since $M$ is non-increasing in $N$ with limit
$m^{*}(Z_i,a,u,\delta)<1$.

We first bound the separated-set sums on each piece. Let $F\subset Z_i$ be an
$(N,2\delta)$-separated set. The balls $\{B_N(x,T,\delta)\}_{x\in F}$ are pairwise
disjoint (two centers $\delta$-close to a common point would be $d_N$-closer than
$2\delta$), have centers in $Z_i$, and orders $\geq N$, so this family belongs to
$\mathcal{P}^{*}(Z_i,N,\delta)$. Since $B_N(x,T,\delta)\subset B_N(x,T,2\delta)$ and
$a\geq 0$,
\[
\sum_{x\in F} 2^{-a\, u(x,N,2\delta)}
\;\leq\;
\sum_{x\in F} 2^{-a\, u(x,N,\delta)}
\;\leq\; M(Z_i,a,u,\delta,N)\;<\;1.
\tag{$*$}
\]
Taking the supremum over such $F$ gives $Q_N(Z_i,T,u,a,2\delta)\leq 1$ for every $N$.

We now upgrade this to a dimension bound by refining the decomposition. For
$N_0\in\mathbb{N}$ and rational $\rho>0$ set
\[
C(Z,N_0,\rho)=\{x\in Z : S_nu(x)\geq\rho n \text{ for all } n\geq N_0\};
\]
since $\underline u>0$ on $Z$, the countably many sets
$Z_{i,N_0,\rho}=Z_i\cap C(N_0,\rho)$ cover $Z$. Fix one such piece and let $a'>a$.
For $N\geq N_0$ and any $(N,2\delta)$-separated $F\subset Z_{i,N_0,\rho}$, each
$x\in F$ lies in its own ball, so
$u(x,N,2\delta)\geq S_Nu(x)\geq\rho N$, and therefore, using
$(*)$,
\[
\sum_{x\in F} 2^{-a'\, u(x,N,2\delta)}
\;\leq\; 2^{-(a'-a)\rho N}
\sum_{x\in F} 2^{-a\, u(x,N,2\delta)}
\;\leq\; 2^{-(a'-a)\rho N}.
\]
Hence $q(Z_{i,N_0,\rho},T,u,a',2\delta)=0$ for every $a'>a$, so
$\dim^{\mathrm{sep}}_{\mathrm{BSC}}(Z_{i,N_0,\rho},T,u,2\delta)\leq a$ uniformly
over all pieces. The decomposition $\{Z_{i,N_0,\rho}\}$ therefore witnesses
\[
\Dim^{\mathrm{sep}}_{BS}(Z,T,u,2\delta)\;\leq\;a.
\]
As this holds for all sufficiently small $\delta$, letting $\delta\to 0$ gives
$\Dim^{\mathrm{sep}}_{BS}(Z,T,u)\leq a$, and letting
$a\downarrow\dim_{\mathrm{BSP}}(Z,T,u)$ completes the proof.
\end{proof}
\begin{lemma}\label{lem:BSP-leq-mUCsep}
Suppose $\underline{u}(x)>0$ for all $x\in Z$. Then
\[
\dim_{\mathrm{BSP}}(Z,T,u)\;\leq\;\Dim^{\mathrm{sep}}_{BS}(Z,T,u).
\]
\end{lemma}

\begin{proof}
Fix any $0<t<s<\dim_{\mathrm{BSP}}(Z,T,u)$ and any countable decomposition
$\bigcup_{i}Z_i=Z$. It suffices to prove that for all $\delta>0$ small enough there
is some $i$ with $\dim^{\mathrm{sep}}_{\mathrm{BSC}}(Z_i,T,u,\delta)\geq t$, since then
$\Dim^{\mathrm{sep}}_{BS}(Z,T,u,\delta)\geq t$ for all small $\delta$, and letting
$\delta\to 0$, $t\uparrow s$, $s\uparrow\dim_{\mathrm{BSP}}(Z,T,u)$ gives the claim.

For $N_0\in\mathbb{N}$ and rational $\rho>0$ let
$C(Z,N_0,\rho)=\{x\in Z: S_nu(x)\geq\rho n \text{ for all } n\geq N_0\}$; since
$\underline u>0$ on $Z$, the refined countable family
$\{Z_i\cap C(N_0,\rho)\}_{i,N_0,\rho}$ also covers $Z$.

Since $s<\dim_{\mathrm{BSP}}(Z,T,u)$, choose $\delta>0$ with
$m^{**}(Z,s,u,\delta)>0$. As $m^{**}$ is an infimum over all countable
decompositions, there are $\alpha>0$ and a refined piece
$Z'=Z_i\cap C(Z,N_0,\rho)$ with $m^{*}(Z',s,u,\delta)>\alpha$. Then for every
$N\geq N_0$ there is a disjoint family $\{B_{n_j}(x_j,T,\delta)\}_{j\geq 1}$ with
$x_j\in Z'$, $n_j\geq N$, and, writing
$W_j=u(x_j,n_j,\delta)$,
\[
\sum_{j\geq 1} 2^{-s\,W_j}\;\geq\;\frac{\alpha}{2}.
\]
Since $x_j\in Z'\subset C(Z,N_0,\rho)$ lies in its own ball, $W_j\geq
S_{n_j}u(x_j)\geq\rho\,n_j$. Let $E_k=\{x_j : n_j=k\}$, so that
\[
\sum_{k\geq N}\;\sum_{x\in E_k} 2^{-s\, u(x,k,\delta)}
\;\geq\;\frac{\alpha}{2}.
\]
We claim there exists $k\geq N$ with
\[
\sum_{x\in E_k} 2^{-t\, u(x,k,\delta)}
\;\geq\;\frac{\big(1-2^{-(s-t)\rho}\big)\alpha}{2}:
\]
otherwise, using $u(x,k,\delta)\geq\rho k$ to split
$2^{-s(\cdot)}=2^{-(s-t)(\cdot)}\,2^{-t(\cdot)}\leq 2^{-(s-t)\rho k}\,2^{-t(\cdot)}$,
\[
\sum_{k\geq N}\sum_{x\in E_k} 2^{-s\, u(x,k,\delta)}
\;\leq\;\sum_{k\geq 1} 2^{-(s-t)\rho k}\cdot
\frac{\big(1-2^{-(s-t)\rho}\big)\alpha}{2}
\;=\;2^{-(s-t)\rho}\,\frac{\alpha}{2}\;<\;\frac{\alpha}{2}.
\]
The disjointness of the balls $\{B_k(x,T,\delta)\}_{x\in E_k}$ makes $E_k$ a
$(k,\delta)$-separated subset of $Z'$, and its weighted sum is exactly of the form
appearing in $Q_k$, so
\[
Q_k(Z',T,u,t,\delta)\;\geq\;\frac{\big(1-2^{-(s-t)\rho}\big)\alpha}{2}.
\]
As such a $k$ exists above every $N\geq N_0$, we get
$q(Z',T,u,t,\delta)\geq\big(1-2^{-(s-t)\rho}\big)\alpha/2>0$, hence
$\dim^{\mathrm{sep}}_{\mathrm{BSC}}(Z',T,u,\delta)\geq t$. By monotonicity under
inclusion, $\dim^{\mathrm{sep}}_{\mathrm{BSC}}(Z_i,T,u,\delta)\geq t$ as required.
\end{proof}
We now show the equivalence between the covering and separated versions of the modified upper capacity BS dimension at scale 0. We first note a classical relation at finite scales.
\begin{lemma}[Spanning--separated comparison, BS weights]\label{lem:span-sep-bs}
Let $u:X\to\mathbb{R}$ be uniformly continuous and $s\geq 0$. For every $Z\subset X$, $n\in\mathbb{N}$, and $\varepsilon>0$,
\[
M(Z,T,u,s,n,2\varepsilon)
\;\leq\; Q_n(Z,T,u,s,2\varepsilon)
\;\leq\; M(Z,T,u,s,n,\varepsilon).
\]
\end{lemma}
\begin{proof}
For the left inequality, let $F\subset Z$ be $(n,2\varepsilon)$-separated; enlarging $F$ to a maximal such set only increases its weighted sum (the weights are positive), so we may assume $F$ maximal. By maximality every $y\in Z$ satisfies $d_n(y,x)<2\varepsilon$ for some $x\in F$, so $Z\subset\bigcup_{x\in F}B_n(x,T,2\varepsilon)$, and this cover's weight in $M$ is exactly the separated sum:
\[
M(Z,T,u,s,n,2\varepsilon)\;\leq\;\sum_{x\in F}2^{-s\,u(x,n,2\varepsilon)}\;\leq\;Q_n(Z,T,u,s,2\varepsilon).
\]

For the right inequality, let $F\subset Z$ be $(n,2\varepsilon)$-separated and let $E\subset X$ be countable with $Z\subset\bigcup_{x\in E}B_n(x,T,\varepsilon)$. For each $y\in F$ pick $\pi(y)\in E$ with $d_n(y,\pi(y))<\varepsilon$. If $\pi(y)=\pi(y')$ then $d_n(y,y')<2\varepsilon$, so $\pi$ is injective. Moreover $d_n(y,\pi(y))<\varepsilon$ gives, by the triangle inequality in $d_n$,
\[
B_n(\pi(y),T,\varepsilon)\subset B_n(y,T,2\varepsilon),
\qquad\text{hence}\qquad
u(\pi(y),n,\varepsilon)\leq u(y,n,2\varepsilon),
\]
and since $s\geq 0$,
\[
2^{-s\,u(y,n,2\varepsilon)}\;\leq\;2^{-s\,u(\pi(y),n,\varepsilon)}.
\]
Summing over $y\in F$ and using injectivity of $\pi$,
\[
\sum_{y\in F}2^{-s\,u(y,n,2\varepsilon)}
\;\leq\;\sum_{x\in E}2^{-s\,u(x,n,\varepsilon)}.
\]
Taking the supremum over $F$ and the infimum over $E$ yields
$Q_n(Z,T,u,s,2\varepsilon)\leq M(Z,T,u,s,n,\varepsilon)$.
\end{proof}
\begin{corollary}[Separated and covering upper BS dimensions agree at scale 0]\label{cor:sep-cov-bs}
Let $Z\subset X$ with $\underline{u}(x)>0$ for all $x\in Z$. Then for every $\varepsilon>0$,
\[
\Dim_{BS}(Z,T,u,2\varepsilon)\;\leq\;\Dim^{sep}_{BS}(Z,T,u,2\varepsilon)\;\leq\;\Dim_{BS}(Z,T,u,\varepsilon),
\]
and consequently
\[
\Dim_{BS}(Z,T,u)\;=\;\Dim^{sep}_{BS}(Z,T,u).
\]
\end{corollary}
\begin{proof}
Fix $Z\subset X$, $\varepsilon>0$, and $s\geq 0$. Taking $\limsup\limits_{n\to\infty}$ in the upper and lower bounds of Lemma~\ref{lem:span-sep-bs},
\[
\limsup\limits_{n\to\infty}M(Z,T,u,s,n,2\varepsilon)
\;\leq\;\limsup\limits_{n\to\infty}Q_n(Z,T,u,s,2\varepsilon)
\;\leq\;\limsup\limits_{n\to\infty}M(Z,T,u,s,n,\varepsilon),
\]
so finiteness of the right quantity implies finiteness of the middle, and finiteness of the middle implies finiteness of the left. Taking the infimum over $s\geq 0$ at which each is finite,
\[
\dim_{BSC}(Z,T,u,2\varepsilon)\;\leq\;\dim^{sep}_{BSC}(Z,T,u,2\varepsilon)\;\leq\;\dim_{BSC}(Z,T,u,\varepsilon).
\]
Now let $Z\subset\bigcup_i Z_i$ be any countable cover. Applying the display to each $Z=Z_i$ and taking $\sup_i$,
\[
\sup_i \dim_{BSC}(Z_i,T,u,2\varepsilon)\;\leq\;\sup_i \dim^{sep}_{BSC}(Z_i,T,u,2\varepsilon)\;\leq\;\sup_i \dim_{BSC}(Z_i,T,u,\varepsilon).
\]
The modified quantities on both sides are infima over the same class of countable covers of $Z$, so
\[
\Dim_{BS}(Z,T,u,2\varepsilon)\;\leq\;\Dim^{sep}_{BS}(Z,T,u,2\varepsilon)\;\leq\;\Dim_{BS}(Z,T,u,\varepsilon),
\]
which is the first claim. Both scale-$\varepsilon$ quantities are nondecreasing as $\varepsilon\to 0$, so letting $\varepsilon\to 0$ in the upper and lower bounds squeezes out the scale-halving and yields $\Dim_{BS}(Z,T,u)=\Dim^{sep}_{BS}(Z,T,u)$.
\end{proof}
Assembling the comparisons of this subsection yields the equivalence
promised in Section~\ref{subsec_pts_BS_upper}.

\begin{theorem}[Packing BS dimension equals modified upper capacity BS
dimension]
\label{thm:packing-bs-equiv}
Let $Z\subset X$ with $\underline{u}(x)>0$ for all $x\in Z$. Then
\[
\Dim_{BS}(Z,T,u)
=\Dim^{\mathrm{sep}}_{BS}(Z,T,u)
=\dim_{\mathrm{BSP}}(Z,T,u).
\]
In particular the modified upper capacity BS dimension of
Definition~\ref{def:mBSC}, adopted as the working definition of the
upper BS dimension throughout the body of the paper, coincides with
the packing BS dimension.
\end{theorem}

\begin{proof}
The first equality is Corollary~\ref{cor:sep-cov-bs}. The second is
the conjunction of Lemma~\ref{lem:BSP-leq-mUCsep}
($\dim_{\mathrm{BSP}}\leq\Dim^{\mathrm{sep}}_{BS}$) and
Lemma~\ref{lem:mUCsep-leq-BSP}
($\Dim^{\mathrm{sep}}_{BS}\leq\dim_{\mathrm{BSP}}$).
\end{proof}
\section{Declarations}
The large language models Claude Opus 4 and Fable 5 were used extensively in the creation of this paper, in the drafting, outlining, proof reading, and editing phases of writing the paper.  
\printbibliography

\end{document}